\documentclass[reqno]{amsart}

\usepackage[all]{xy}
\usepackage{enumerate}
\usepackage{amscd}
\usepackage{enumerate}
\usepackage{amssymb,amsmath}
\usepackage{graphicx}
\usepackage{tikz}

\theoremstyle{plain}
  \newtheorem{theorem}{Theorem}[section]
  
  \newtheorem{lemma}[theorem]{Lemma}
  \newtheorem{proposition}[theorem]{Proposition}
  
 \newtheorem{obs}[theorem]{Observation}
  \theoremstyle{definition}
  \newtheorem{definition}[theorem]{Definition}
  \newtheorem{example}[theorem]{Example}

\DeclareMathOperator{\dom}{dom}
\DeclareMathOperator{\img}{img}
\DeclareMathOperator{\mul}{mul}

\DeclareMathOperator{\Id}{Id}

\DeclareMathOperator{\coker}{coker}
\DeclareMathOperator{\supp}{supp}

\DeclareMathOperator{\Confg}{Confg}

\newcommand\reg{\mathrm{reg}}

\newcommand\lf{\mathrm{l.f.}}
\newcommand\Nov{\mathrm{Nov}}

\newcommand{\X}{\mathbb X}
\newcommand{\R}{\mathbb R}

\newcommand\itemref[1]{(\ref{#1})}
\newcommand{\BM}{{\rm BM}}

\begin{document}

\title{Topology of  
angle valued maps, bar codes and Jordan blocks.}

\author{Dan Burghelea}

\address{Dept. of Mathematics,
         The Ohio State University,
         231 West 18th Avenue,
         Columbus, OH 43210, USA.
        }

\email{burghele@mps.ohio-state.edu}

\author{Stefan Haller}

\address{Vienna, Austria.
        } 

\email{stefan.haller.42@gmail.com}

\thanks{
  Part of this work was done while the second author enjoyed the  hospitality of the Ohio State University. 
  The present version of this paper was finalized while  first author was visiting the Bernoulli Center at EPFL Lausanne and MPIM Bonn.  
  The first author acknowledges partial support from NSF grant MCS 0915996.
  The second author acknowledges the support of the Austrian Science Fund, grant P19392-N13.}

\keywords{}

\subjclass{}

\date{\today}

\begin{abstract}\

In this paper one presents a collection of results about the ``bar codes'' and ``Jordan blocks'' introduced in \cite {BD11} as
{\it computer friendly} invariants of a tame angle-valued map and one relates these invariants to the Betti numbers, Novikov Betti numbers and 
the monodromy of the underlying space and map.

Among others, one organizes the bar codes as two configurations of points in $\mathbb C\setminus 0$ and one establishes their main properties: stability property and when the  underlying space is a closed topological manifold, Poincar\'e duality property. One also provides an alternative {\it computer friendly} definition of the monodromy of an angle valued map based on the  algebra of {\it linear relations} as well as a refinement of 
Morse and Morse-Novikov inequalities.  

\end{abstract}

\maketitle 

\setcounter{tocdepth}{1}
\tableofcontents

\section{The results}\label{S1}

In this paper a \emph{nice space} is a friendlier name for a locally compact ANR (Absolute  Neighborhood Retract).\footnote{A metrizable, locally compact, finite dimensional locally contractible space is nice, see \cite{Hu}.}
Finite dimensional simplicial complexes and finite dimensional topological manifolds are nice spaces but the class is considerably larger. 
A \emph{tame} map is a proper continuous map $f\colon X\to \mathbb R$ or $f\colon X\to \mathbb S^1$, defined on a nice space $X$,
which satisfies: 
\begin{enumerate}[(i)]
\item 
each fiber of $f$ is a neighborhood deformation retract, and
\item
away from a discrete set $\Sigma\subset\mathbb R$ or $\Sigma\subset\mathbb S^1$ 
the restriction of $f$ to $X\setminus f^{-1}(\Sigma)$ is a fibration, cf.~\cite{BD11}.  In particular for $t\notin \Sigma(f)$ there exists a neighborhood $U\ni t$ such that for any $t'\in U,$ the inclusion $f^{-1}(t')\subset f^{-1}(U)$ is a homotopy equivalence.
\end{enumerate}

\noindent All proper 
simplicial maps and proper smooth generic maps defined on a smooth manifold,\footnote{Here ``generic'' means that for any $x\in M$ the quotient algebra of germs of smooth functions 
at $x$ by the ideal of partial derivatives is a finite dimensional vector space.} in particular proper real or angle valued Morse maps, are tame.   
At least for spaces homeomorphic to simplicial complexes the set of tame maps is residual in the space of all continuous maps 
and weakly homotopy equivalent to the space of all continuous maps (equipped with compact open topology).\footnote{
In case that the  space $X$ is homeomorphic to a finite dimensional simplicial complex, this is consequence of the approximability of continuous maps by pl-maps.} 

Most of the time we will have an a priory fixed field $\kappa$  and homology, Novikov homology, Betti numbers, etc. will be considered with respect to this field.  
For simplicity in writing,  the field  $\kappa$ will be omitted from the notations.  

In this paper we consider a tame map, $f\colon X\to \mathbb S^1$, and as in \cite{BD11}, one associates to the map $f$:
\begin{enumerate}[(i)]
\item\label{I:i}
the set of critical angles $0<\theta_1<\theta_2<\cdots<\theta_m\leq2\pi$,  
\item\label{I:ii}
for any $r=0,1,\dotsc,\dim X$, four types of intervals of real numbers, 
\begin{enumerate}[(1)]
\item
closed  ($[a,b]$), 
\item
open  ($(a,b)$), 
\item
closed-open ($[a,b)$),  
\item
open-closed ($(a,b]$),
\end{enumerate}
subsequently called \emph{$r$-bar codes}, whose ends mod $2\pi$ are critical angles,
with $0 < a \leq 2\pi,$ 
\item\label{I:iii}  
for any $r=0,1,\dotsc,\dim X$,
a collection of  isomorphism classes of indecomposable pairs $J=(V_J,T_J)$, where $T_J$ is a linear automorphism of a finite dimensional $\kappa$-vector space $V_J,$ subsequently called \emph{Jordan blocks}. 
\end{enumerate}
The bar codes can be also regarded as  equivalence classes of intervals  as above modulo \emph{translation by an integer multiple of $2\pi$,}
with ends mod $2\pi$  critical angles.

Recall that a pair $(V,T)$ 
is \emph{indecomposable} if not isomorphic to the sum of two nontrivial pairs. 
In this case  if $T$ has $\lambda\in\kappa$ as an eigenvalue all other eigenvalues are equal to $\lambda$, and $(V,T)$ is isomorphic to $(\kappa^k,T(\lambda,k))$ where $T(\lambda, k)$ is the $k\times k$ matrix 
\begin{equation}\label{E1}
T(\lambda,k)=
\begin{pmatrix}
\quad\lambda\quad & \quad1\quad       & \quad 0\quad      & \quad\cdots\quad  & \quad0\quad      \\
0       & \lambda & 1      & \ddots  & \vdots \\
0       & 0       & \ddots & \ddots  & 0      \\
\vdots  & \ddots  & \ddots & \lambda & 1      \medskip\\ 
0       & \cdots  & 0      & 0       & \lambda
\end{pmatrix}
\end{equation} if $k\geq 2,$  and  $T(\lambda,1)= (\lambda)$ if $k=1.$
In \cite {BD11} the indecomposable pairs $(\kappa^k,T(\lambda,k))$ were called \emph{Jordan cells}. 
When $\kappa$ is algebraically closed all Jordan blocks are Jordan cells.

We  denote by $\mathcal B_r^c(f)$, $\mathcal B_r^o(f)$, $\mathcal B_r^{co}(f)$, $\mathcal B_r^{oc}(f)$ 
the collections  of closed, open, closed-open and open-closed $r$-bar codes and by $\mathcal J_r(f)$ the collection of $r$-Jordan blocks. For brevity we also write $\mathcal B_r(f):= \mathcal B^c_r(f)\sqcup \mathcal B_r(f) \sqcup \mathcal B_r(f)\sqcup \mathcal B_r(f).$
Each bar code or Jordan block appears  with a multiplicity possibly larger than one. All these collections are multisets  which means each element appears with multiplicity. 
For $u\in \kappa\setminus 0$ we denote by $\mathcal J_{r,u}(f)$ the sub-collection of $r$-Jordan blocks with eigenvalue $u.$
In view of the definitions in Section~\ref{SS3}, cf.\ also \cite{BD11}, each tame map has finitely many bar codes and Jordan blocks. 

It was shown in \cite{BD11} that for simplicial maps these invariants are effectively computable and an algorithm for their calculation was proposed.
Existence of such algorithms is what we  mean by {\it computer friendly invariants}. All these invariants are described in Section \ref {SS3}.

In order to formulate the results, we recall that any continuous map $f\colon X\to \mathbb S^1$ determines an integral cohomology class $\xi_f\in H^1(X;\mathbb Z)$ via pull back of a fixed generator in $H^1(\mathbb S^1;\mathbb Z)\cong\mathbb Z$.
By homotopy invariance, homotopic maps $f_1,f_2\colon X\to \mathbb S^1$ determine the same class, $\xi_{f_1}=\xi_{f_2}$.
If $X$ an ANR, then this assignment induces a bijection between the set of homotopy classes of maps $X\to\mathbb S^1$ and $H^1(X;\mathbb Z)$.
In other words, any class $\xi\in H^1(X;\mathbb Z)$ is of the form $\xi= \xi_f$ for some continuous angle-valued map $f$ which is unique up to homotopy.
This follows from the fact that the circle $\mathbb S^1$ is an Eilenberg--MacLane space $K(\mathbb Z,1)$, see~\cite[Section~4.3]{H02}.

We say that two pairs $(X_1,\xi_1)$ and $(X_2,\xi_2)$, $\xi_1,\xi_2 \in H^1(X;\mathbb Z)$, are homotopy equivalent if there exists a homotopy equivalence $\omega\colon X_1\to X_2$ s.t.\ $\omega^\ast(\xi_2)= \xi_1$.  

The basic algebraic topology  invariants  associated with a pair $(X,\xi)$, $\xi\in H^1(X;\mathbb Z)$, a field $\kappa$, and a positive integer $r\in\mathbb N_0$ we consider in this paper are:
\begin{enumerate}[(1)]
\item
the singular homology $H_r(X)$, a $\kappa$-vector space whose dimension, when finite, is called the Betti number $\beta_r(X)$;
\item
the Novikov homology $H^N_r(X;\xi)$, a vector space over the field of Laurant power series $\kappa[t^{-1},t]]$ with coefficients in $\kappa$, whose dimension, when finite, is called the Novikov--Betti number $\beta^N_r(X;\xi)$; and
\item 
the $r$-monodromy, an isomorphism class of pairs $(V_r,T_r)$ where $V_r$ is a $\kappa$-vector space and $T_r\colon V_r\to V_r$ is a linear isomorphism.
\end{enumerate}
If $X$ is a compact ANR then $\beta_r(X)$, $\beta^N_r(X;\xi)$, and $\dim(V_r)$ are finite.   

The first result we prove in this paper is Theorem~\ref{T1} below.

\begin{theorem}[Homotopy invariants]\label{T1}
If $f\colon X\to \mathbb S^1$ is a tame map and $\xi_f\in H^1(X;\mathbb Z)$ is the integral cohomology class represented by $f$ then: 
\begin{enumerate}[(a)]
\item\label{T1:a}
$\sharp\mathcal B^c_r(f)+\sharp\mathcal B^o_{r-1}(f)$ is a homotopy invariant of the pair $(X,\xi_f)$, 
more precisely is equal to the Novikov--Betti number $\beta^N_r(X;\xi_f)$. 
\item\label{T1:b}
The collection $\mathcal J_r(f)$ is a homotopy invariant of the pair $(X,\xi_f)$. 
More precisely, $\bigoplus_{J\in \mathcal J_r(f)}J := \bigoplus_{J\in \mathcal J_r(f)}(V_J,T_J) $ 
is the $r$-monodromy $(V_r, T_r)$ of $(X;\xi_f)$. 
\item\label{T1:c}
$\sharp \mathcal B^c_r(f)+\sharp \mathcal B^o_{r-1}(f)+\sharp \mathcal J_{r,1}(f)+\sharp \mathcal J_{r-1,1}(f)$ is equal to the Betti number $\beta_r(X)$.
\end{enumerate}
\end{theorem}

The definition of Novikov--Betti numbers  and of the monodromy are given in  Section~\ref{S4} and  ``$\sharp$'' denotes the cardinality of a multiset.

Item \itemref{T1:c} has been already established in \cite[Theorem~3.2]{BD11} and is included in Theorem~\ref{T1} only for the completeness of the topological information derived from bar codes and Jordan blocks.
 
 \vskip .1in

In view of Theorem~\ref{T1} it is natural to put together $\mathcal B_r^c(f)$ and $\mathcal B^o_{r-1}(f)$.
For this purpose consider $\mathbb T=\mathbb R^2/\mathbb Z$ and $\Delta_{\mathbb T}=\Delta/\mathbb Z$ where the $\mathbb Z$-action on $\mathbb R^2$ is given by $(n, (a, b))\mapsto(a + 2\pi n, b + 2\pi n)$ and $\Delta:=\{(a, b)\in \mathbb R^2 \mid a=b\}$.
One denotes the $\mathbb Z$-orbit of $(a,b)\in \mathbb R^2$ by $\langle a,b\rangle\in\mathbb T$. 
Note that $\mathbb T$ can be identified to $\mathbb C\setminus 0$ via the map $\langle a,b\rangle\mapsto z:=e^{(a-b)/2+i(a+b)/2}$.
Via this identification, $\Delta_\mathbb T$ corresponds to the unit circle $\mathbb S^1=\{z\in\mathbb C:|z|=1\}$.

We will record the collections $\mathcal B_r^c(f)\sqcup \mathcal B_{r-1}^o(f)$ as a finite configuration of points in $\mathbb T= \mathbb C\setminus 0$, denoted by $C_r(f)$,
and the collection $\mathcal B_r^{co}(f)\sqcup \mathcal B_{r}^{oc}(f)$ as a finite configuration of points in $\mathbb T\setminus \Delta_{\mathbb T}= \mathbb C\setminus (0 \sqcup \mathbb S^1)$, denoted  by $C_r^m(f)$.

More precisely, in the first case a closed $r$-bar code $[a,b]$ will be written as $\langle a, b\rangle \in \mathbb T$ or the complex number $z=e^{(a-b)/2+i(a+b)/2}\in\mathbb C\setminus0$ and an open
$(r-1)$-bar code $(\alpha,\beta)$ as $\langle\beta,\alpha\rangle\in\mathbb T$ or the complex number $z=e^{(\beta-\alpha)/2+i(\beta+\alpha)/2}\in\mathbb C\setminus 0$.
Similarly, in the second case, a closed-open $r$-bar code $[a,b)$ will be written as $\langle a,b\rangle\in\mathbb T\setminus\Delta_{\mathbb T}$ or the complex number $e^{(a-b)/2+i(a+b)/2}\in\mathbb C\setminus(0\sqcup\mathbb S^1)$ and an open-closed $r$-bar code $(\alpha,\beta]$ as $\langle\beta,\alpha\rangle\in\mathbb T\setminus\Delta_{\mathbb T}$ or the complex number $e^{(\beta-\alpha)/2+i(\beta+\alpha)/2}\in\mathbb C\setminus(0\sqcup \mathbb S^1)$.

In view of Theorem~\ref{T1}\itemref{T1:a}, if $f$ is in the homotopy class defined by $\xi\in H^1(X;\mathbb Z)$, then the configuration
$C_r(f)$ has the total cardinality of the support
\footnote{The total cardinality of the support of a configuration is the sum of the multiplicities of its points.}
exactly $\beta^N_r(X;\xi)$ and can be regarded as a point in the $n$-fold symmetric product $S^n(\mathbb T)$ of $\mathbb T$ where $n=\beta^N_r(X;\xi)$.
Identifying $\mathbb T$ with $\mathbb C\setminus0$ as above, the space $S^n(\mathbb T)$ identifies to the space of monic polynomials of degree $n$ with non-vanishing free coefficient, that is, $\mathbb C^{n-1}\times(\mathbb C\setminus 0)$, by assigning to a complex polynomial its configuration of zeros with multiplicities.
Hence, each $C_r(f)$ can be regarded as a monic polynomial $P_r^f(z)$ of degree $n$ with non-vanishing free coefficient.
We equip $S^n(\mathbb T)$ with the topology of the symmetric product or equivalently with the topology of $\mathbb C^{n-1}\times (\mathbb C\setminus 0)$.

Let $C(X,\mathbb S^1)$ denote the space of all continuous maps equipped with the compact open topology and let $C_{\xi}(X,\mathbb S^1)$ be the connected component corresponding to $\xi$. 
Let $C_{\xi, t}(X,\mathbb S^1)$ be the subspace of tame maps in $C_{\xi}(X,\mathbb S^1)$.  
Our next result 
is the following theorem which will be referred to below as Strong Stability Theorem.

\begin{theorem}[Stability]\label{T2} 
Suppose $X$ is a compact ANR. 
Then the assignment 
$$
C_{\xi, t}(X,\mathbb S^1)\ni f \mapsto C_r(f)\in S^{n}(\mathbb T),
$$
equivalently, 
$$
C_{\xi, t}(X,\mathbb S^1)\ni f \mapsto P^f_r(z)\in \mathbb C^{n-1}\times (\mathbb C\setminus 0),
$$ 
is continuous, where $n=\beta^N_r(X,\xi)$.

Moreover, if  $X$ is homeomorphic to a simplicial complex, this extends to a continuous map, $C_{\xi}(X,\mathbb S^1)\to S^{n}(\mathbb T)$,
equivalently, $C_{\xi}(X,\mathbb S^1)\to \mathbb C^{n-1}\times (\mathbb C\setminus 0)$.
\end{theorem}

In particular, if $X$ is triangulable, then the configuration $C_r(f)$ and therefore the closed and open bar codes, can be defined for any continuous map. 
It is expected that the triangulability hypothesis can be removed.
\footnote{Results on Hilbert cube manifolds permit to remove the triangulability hypothesis, cf.~\cite {B14}.}

The configuration $C_r(f)$, equivalently the polynomial $P_r^f(z)$, can be viewed as a refinement of the Novikov--Betti number in dimension $r$. 
The Poincar\'e duality for closed manifolds extends from Novikov--Betti numbers to these refinements and we have the following theorem. 

\begin{theorem}[Poincar\'e duality]\label{T3}
If $M^n$ is a closed $\kappa$-orientable
\footnote{If $\kappa$ has characteristic $2$ any manifold is $\kappa$-orientable if not the manifold should be orientable.} 
topological manifold and $f\colon M\to \mathbb S^1$ a tame map, then 
$$
C_r(f)(\langle a,b\rangle)=C_{n-r}(f)(\langle b, a\rangle).
$$
Equivalently, $C_r(f)(z)=C_{n-r}(f)(\tau(z))$ where $\tau(z):=1/\bar z=z/|z|^2$ denotes the inversion across the unit circle, $z\in\mathbb C\setminus0$.
\end{theorem}

The proofs of Theorems~\ref{T2} and \ref{T3} use an alternative definition of the configuration $C_r(f)$.
One defines the function $\delta^f_r$ on $\mathbb T$ with values in $\mathbb N_0$,  with no reference to ``bar codes'' or to graph representations, and one verifies that it is equal to the configuration $C_r(f)$.
One verifies Theorems~\ref{T2} and \ref{T3} for $\delta^f_r$ instead of $C_r(f)$.

\vskip .1in

The Jordan blocks introduced in \cite{BD11} via graph representations, can be also recovered in a different manner, 
more precisely, as the \emph{regular part} of a linear relation. This makes their computations achievable by an algorithm less expensive 
than the one presented in \cite {BD11}, cf.~\cite {B16}.

A linear relation, $R\colon V\leadsto V$, is a concept generalizing a linear map, $V\to V$. 
To every linear relation $R$ on a $\kappa$-vector space $V$ one can associated canonically a pair, $R_\reg= (V_\reg, T_\reg)$, where $V_\reg$ is a $\kappa$-vector space and $T_\reg\colon V_\reg\to V_\reg$ is a linear isomorphism.
This construction will be discussed in Section~\ref{SS:linrel}.

To a tame map $f\colon X\to \mathbb S^1$ one associates linear relations $R^\theta_r\colon H_r(f^{-1}(\theta))\leadsto H_r(f^{-1}(\theta))$ described as follows.
Let $\tilde f\colon\tilde X\to \mathbb R$ be the infinite cyclic covering defined by the pullback diagram 
$$
\xymatrix{
\tilde X\ar[d]\ar[r]^-{\tilde f}&\R\ar[d]^p\\X\ar[r]^-f&S^1.
}
$$
For $t$ with $p(t)= \theta$  the linear relation $R^{\theta}_r$ is obtained by passing to homology in the sequence, see Section~\ref{S8} for more details,
$$
f^{-1}(\theta)=\tilde f^{-1}(t)\hookrightarrow\tilde f^{-1}\bigl([t,t+2\pi]\bigr)
\hookleftarrow\tilde f^{-1}(t+2\pi)=f^{-1}(\theta).
$$

We have the following result.

\begin{theorem}[Monodromy theorem]\label{T4} 
If $f$ is a tame map and $r$ a  non-negative integer, then for any angle $\theta$ 
the pair $(R^\theta_r)_\reg$ is isomorphic to $$
(\bigoplus_{J\in\mathcal J_r(f)} V_J, \bigoplus_{J\in\mathcal J_r(f)} T_J),$$ 
with $J= (V_J, T_J)$ and is a homotopy invariant of the pair $(X,\xi_f).$
\end{theorem}  
\vskip .1in 

The next results refers to the collections $\mathcal B^{co}_r(f)$ and $\mathcal B^{oc}_r(f)$ of mixed bar codes.
First we note that the collection $\mathcal B_r^{co}(f)$ can be identified to the collection of finite \emph{persistence intervals} considered in \cite{ELZ02} or \cite{CEH07} for the map $\tilde f\colon\tilde X\to \mathbb R$ made equivalent modulo $2\pi$-translation.
Similarly, the collection $\mathcal B_r^{oc}(f)$, after changing $(a,b]$ into $[-b,-a)$, can be identified to the collection of finite \emph{persistence intervals} of the map $-\tilde f$ modulo $2\pi$-translation.

The configurations $C^m_r(f)$ obtained by putting together $\mathcal B_r^{co}(f)$ and $\mathcal B_r^{oc}(f)$ also enjoy a  stability property and  
Poincar\'e duality, cf. Theorem \ref{T5} and Theorem \ref{T6} below, however with different quantitative and qualitative properties. 
Theorem \ref {T5}  is a reformulation of the famous stability result   
of \cite{CEH07} and is stated here only for comparison with Theorem \ref{T2}. 

Note that for tame angle valued maps in the same homotopy class the configurations $C^m_r(f)$ do not have the support of the same cardinality therefore  
a stability property 
will require  a new topology on the set of configuration; in such topology the definition of ``proximity''  ignores the points near the diagonal $\Delta_{\mathbb T}.$ 
This topology on the space of configurations of points in $\mathbb T\setminus \Delta_{\mathbb T},$ called the {\it bottleneck topology}, can be   
derived from a metric proposed in \cite {CEH07}, the {\it bottleneck metric}. 

Here is an alternative definition of the ``bottleneck topology'' on the set $\Confg(X\setminus K)$ of configurations of points in 
$X\setminus K$, $X$ locally compact space and $K$ a closed subset of $X$ without involving metric.
Recall that a configuration is a map with finite support, $\delta\colon X\setminus K\to \mathbb N_0$. 
A base for the bottleneck topology  is given by the collection of sets $\mathcal U(S)$ indexed by systems 
$S=\{(U_1, k_1),\dotsc,(U_r,k_r),V\}$ satisfying 
\begin{enumerate}[(i)]
\item
$U_i$, $i=1,\dotsc,r$ open subsets of $X\setminus K$, $V$ open neighborhood of $K$, 
\item
$k_1,k_2,\dotsc,k_r$ positive integers.
\end{enumerate}
and defined by 
$$
\mathcal U(S):=\bigl\{\delta\in\mathcal\Confg(X\setminus K) \bigm| \supp(\delta)\subset U_1\cup\cdots\cup U_r\cup V,\ \textstyle\sum_{x\in U_i}\delta(x)=k_i\bigr\}.
$$
A set $\mathcal U\subset \Confg(X\setminus K)$ is open in the bottleneck topology if for any $\delta\in \mathcal U$ there exists $S$ such that
$$
\delta\in \mathcal U(S)\subset \mathcal U.
$$ 

When $X$ is a complete locally compact metric space and $K$ is a closed subspace the bottleneck metric of $\Confg(X\setminus K)$ given by the formulae 
 in \cite {CEH07} induces the bottleneck topology described above.

The ``main theorem''  in \cite{CEH07} implies:

\begin{theorem}[CEH stability] \label {T5}
The assignment $f\mapsto C^m_r(f)$ is a continuous map from the space $C_t(X,\mathbb S^1)$ of tame maps to $\mathcal\Confg(\mathbb T\setminus \Delta_{\mathbb T})$ 
when the first space is equipped with the compact open topology and the second with the topology  described above in case $(X,K)= (\mathbb T, \Delta_{\mathbb T})$. 
\end{theorem}

To better realize the differences between Theorem \ref {T2} and \ref {T5} we point out that: 
\begin{enumerate}[(a)]
\item\label{C:a} Arbitrary small perturbations of a tame map can introduce arbitrary many mixed bar codes, see Example~1 in Appendix~\ref{App2}.
\item\label{C:b} Arbitrary small perturbations of some tame maps (which have closed $r$-bar codes of the form $[a,a]$) can decrease the number of closed $r$-barcodes and increase the number of open $(r-1)$-bar codes, see Example~2 in Appendix~\ref{App2}.
\item\label{C:c} Continuous deformations of tame maps can make the mixed $r$-barcodes, i.e., closed $r$-bar codes and open $r$-bar codes, appear or disappear. The decrease or the increase  in the number of closed $r$-bar codes is at the expense of the increase or the decrease in the number of open $(r-1)$-barcodes, see Examples~2 and 3 in Appendix~\ref{App2}.
\item\label{C:d} The assignment $f\mapsto C^m_r(f)$,  as opposed to the assignment $f\mapsto C_r(f)$, can not be extended continuously to the entire space $C_\xi(X,\mathbb S^1)$ and this because of the lack of completeness of the bottleneck metric.
\end{enumerate}

For the reader familiar with Morse theory we point out that the disappearance of a closed-open $r$-bar code by a continuous deformation of tame map is similar to  the cancellation of a pair of two critical points one of index $r$ one of index $(r-1)$ in Morse theory as described in  \cite {Mi4}. 

\vskip .1in

For a closed topological manifold $M^n$  the configurations $C^m_r(f)$ satisfy Poincar\'e duality but in analogy to the 
the Poincar\'e duality  for the torsion subgroups of the integral homology groups for closed orientable manifolds. 
Precisely, we have the following result.

\begin{theorem}[Poincar\'e duality]\label{T6} 
If $M^n$ is a closed $\kappa$-orientable topological manifold, $f\colon M\to \mathbb S^1$ a tame map, then
$$
C^m_r(f)(\langle a,b\rangle)=C^m_{n-1-r}(f)(\langle b, a\rangle).
$$
Equivalently, $C^m_r(f)(z)=C^m_{n-1-r}(f)(\tau(z))$, where $\tau(z):=1/\bar z=z/|z|^2$ denotes the inversion across the unit circle, $z\in\mathbb C\setminus0$.
\end{theorem}

\vskip .1in

It is interesting to regard the elements \itemref{I:i}, \itemref{I:ii}, \itemref{I:iii}, that is, the critical values, bar codes and Jordan blocks associated to a tame angle valued map $f\colon X\to \mathbb S^1$,
as parallels  to the rest points, the isolated trajectories between rest points and the closed trajectories 
(actually Poincar\'e return maps for closed trajectories) of a vector field which has a Morse angle-valued map $f\colon M\to \mathbb S^1$ as Lyapunov map. 
These last ones are the concepts which enter the Morse--Novikov theory, cf.~\cite{Novikov,P}, and are related to the topology of $(X,\xi_f)$, where $\xi_f$ denotes the integral cohomology class defined by $f$, 
in a  similar way as the elements described in  \itemref{I:i}, \itemref{I:ii} and \itemref{I:iii} are. 

The last result, Theorem~\ref{T7} below, improves on Morse inequalities for real-valued maps resp.\ Morse--Novikov inequalities for angle-valued maps, the simplest and most familiar applications of Morse resp.\ Morse--Novikov theory, cf.\ \cite {Mi3,P}. 

Recall that for a smooth closed manifold $M^n$ a smooth real or angle valued map is Morse if all critical points $x$ are non degenerate, hence have a Morse index, $\operatorname{ind}_f(x)\in\{0,1,\dotsc,n\}$.  
Recall that a  point $x\in M$ is critical if, with respect to any local coordinates $(t_1,t_2,\dotsc,t_n)$ with $x$ given by $t_1=t_2=\cdots=t_n=0$, all partial derivatives $\partial f/\partial t_i(0)$ vanish. 
A critical point is non-degenerate if in addition the Hessian, i.e.\ the symmetric matrix $\frac{\partial^2f}{\partial t_i\partial t_j}(0)$, has all eigenvalues non-zero.
These eigenvalues are all real, and the Morse index of $f$ at $x$ coincides with the number of negative eigenvalues.
The concepts critical points, non-degenerate critical points, and index of a non-degenerate critical points are independent of the local coordinates $(t_1, t_2, \dotsc, t_n)$.

Let $c_i(f)$ be the number of critical points of index $i.$ 
The Morse inequalites claim that for any $r\in\mathbb N_0$, and with respect to any field $\kappa$, one has 
$$
\sum_{i=0}^r(-1)^{r-i}c_i(f) \geq \sum_{i=0}^r(-1)^{r-i}b_i
$$
where $b_i=\beta_i(M)$ if $f$ is real-valued Morse map, and $b_i=\beta^N_i(M;\xi_f)$ if $f$ is an angle-valued Morse map, see \cite[Equation~($4_\lambda$)]{Mi3}.

The following result refines the  Morse inequalities:

\begin{theorem}\label{T7}
Let $M^n$ be closed smooth manifold,
\footnote{The result remains true for compact manifolds with boundary satisfying an appropriate hypothesis on the behavior of $f$ along the boundary.} 
consider a field $\kappa$, and suppose $r$ is a non-negative integer.
If $f\colon M\to\mathbb R$ is a real-valued Morse map, then
$$
c_r(f)=\beta_r(M)+\sharp\mathcal B^{co}_r(f)+\sharp\mathcal B^{co}_{r-1}(f).
$$
Moreover, if $f\colon M\to\mathbb S^1$ is an angle-valued Morse map, then
$$
c_r(f)=\beta^N_r(M;\xi_f)+\sharp\mathcal B^{co}_r(f)+\sharp\mathcal B^{co}_{r-1}(f).
$$
\end{theorem}

Note that the right side of the above equalities make sense for an arbitrary compact ANR equipped with a tame map rather than compact manifolds equipped with a Morse function, an attractive feature in comparison with the classical Morse--Novikov theory.

\vskip .1in

The paper contains, in addition to the present section which summarizes the results, eight more sections which describe the concepts involved in and provide the proofs of the results and three appendices. 
In Section~\ref{S2} we review  simple results about graph representations of the two graphs relevant for this paper, $G_{2m}$ and $\mathcal Z.$  
In Section~\ref{SS3} we define the sets $\mathcal B^{\dots}_r(f)$ and $\mathcal J_r(f)$ and  provide the preliminaries for the proof of Theorem~\ref{T1}.
In Section~\ref{S4} we prove Theorem~\ref{T1}. 
In Section~\ref{SS5} we define the function $\delta^f_r$ and prove Theorem~\ref{T2}. 
In Sections~\ref{SS6} and \ref{SS7} we discuss the Poincar\'e duality for the configurations $C_r(f)$ and $C^m_r(f)$ and prove Theorems~\ref{T3} and \ref{T6}. 
In Section~\ref{SS2} we discuss some linear algebra of linear relations and prove Theorem~\ref{T4}. 
In Section~\ref{S9} we verify Theorem~\ref{T7}.
Appendix~\ref{App1} provides an example of tame map and describes its bar codes and Jordan cells.
Appendix~\ref{App2} illustrates the behavior of bar codes with respect to a continuous deformation of the map.
Appendix~\ref{SS10} provides a few observations about $\kappa[t^{-1},t]$-modules.


Note that if $f\colon X\to \mathbb S^1$ is not surjective, then the set $\mathcal J_r(f)$ vanishes for all $r$.
Note also that a real-valued $f$ can be viewed as a non-surjective angle-valued map, and the bar codes are essentially the same as the zigzag persistence barcodes cf.\ \cite{CSD09}. 
In this case there is no need to consider $\mathbb T$ and $\mathbb T\setminus \Delta_\mathbb T$; the natural place for the support of the configuration $C_r(f)$, consisting of closed $r$-barcodes and open $(r-1)$-barcodes, is $\mathbb C$; and the natural place for the support of $C^m_f(f)$, consisting of the closed-open $r$-barcodes and open-closed $r$-barcodes, is $\mathbb C\setminus \Delta$. 

\subsubsection*{Prior work}

Relating the topology of a space to the homological behavior of the sublevel sets of a real or angle-valued map represents what ``persistence theory'' introduced in \cite{ELZ02} intends to do. 
Prior efforts to extend Morse theory to all continuous real-valued functions (fonctionelles) can be found in the papers of M.~Morse \cite{MO} and R.~Deheuvels \cite{RD55}, which preceded persistence theory.
The work of R.~Deheuvels \cite{RD55}, permits to derive the barcodes (the support of persistence diagrams as considered in \cite{ELZ02} or \cite{EH}) from the differentials of Leray spectral sequence of a real valued tame map. 
The same observation holds about the barcodes in the zigzag persistence and persistence for circle valued maps but we can not find this in this existing in literature. 

A stability phenomena for the persistence diagrams associated to a real valued map in classical persistence theory was first established in \cite{CEH07}. 

The first use of graph representations in connection with persistence appears first in \cite{CSD09} under the name of zigzag persistence. 
The graphs considered are all linear finite graphs whose collection of indecomposable representations is finite and not hard to describe and interpret as bar codes (of four types).

The definition of bar codes and of Jordan cells for $\mathbb S^1$-valued tame maps was first provided in \cite{BD11} based on graph representations of the cyclic graph $G_{2m}$ 
whose indecomposable representations are more complex and led in addition to bar codes to Jordan cells. 

The referee points out a that a number of the results in this paper are reminiscent of  behavior  of the bar codes in zigzag persistence cf.\ \cite{CSD09} and this deserves to be mentioned. 
We are happy to do so. 
Reminiscences of the work of \cite{CEH09} in the Poincar\'e  duality Theorems~\ref{T6} should be also acknowledged.

\subsubsection*{Some more recent work}

Using results from topology of Hilbert cube manifolds, it was  recently observed that the hypothesis ``$X$ homeomorphic to a simplicial complex'' in Theorem~\ref{T2} can be weaken to 
``$X$ compact ANR'', and the hypothesis ``tame map'' in Theorems~\ref{T1}, \ref{T3}, and \ref{T4} can be weaken to ``continuous map'' cf \cite {B14} and \cite {B15}. 

In case of a real valued map and in the presence of a scalar product on $H_r(X)$ (the field $\kappa$ being $\mathbb R$ or $\mathbb C$) the configuration $C_r(f)$ can be implemented as a configuration
$\hat \delta^f_r$ of subspaces 
$\hat \delta^f_r(z)\subseteq H_r(X)$, $z$ in the support of $C_r(f)$, 
which are mutually orthogonal and have $\dim \hat \delta^f_r(z)$ equal to the multiplicity of $z$.
The assignment $f\rightsquigarrow \hat \delta^f_r$ remains continuous w.r.\ to the obvious topologies and in case of closed manifolds Poincar\'e duality between configurations $C_r(f)$  extends to the configurations $\hat \delta^f_r$ of vector spaces.
This is the case when $X$ is the underlying space of a closed Riemannian manifold $M^n$ and $\kappa=\mathbb R$ with the scalar product on $H_r(M)$ 
provided by the identification with the space of harmonic forms in complementary dimension $(n-r)$. 
This will be discussed in details in \cite{B14}.

A similar fact remains true for angle valued maps. If $\kappa=\mathbb C$ the Novikov homology $H^N_r(X;\xi_f)$ can be replaced by the $L_2$-homology $H^{L_2}_r(\tilde M)$ of the infinite cyclic cover $\tilde X$ defined by the map $f$. 
When regarded as a Hilbert module over the von Neumann algebra $L^{\infty}(\mathbb S^1)$  this Hilbert module has the von~Neumann dimension equal to the Novikov--Betti number $\beta^N_r(X;\xi_f)$. 
The mutually orthogonal subspaces are in this case mutually orthogonal Hilbert submodules. This will be discussed in \cite {B15}. 

If $f\colon M^n\to\mathbb R$ or $f\colon M^n\to \mathbb S^1$ is a Morse function, Lyapunov for a smooth vector field $X$ on a closed manifold $M,$ 
the  Morse complex resp.\ the Novikov complex tensored by a field $\kappa$ derived geometrically from the critical points of $f$ and the isolated trajectories  of $X$ between critical points, can be recovered up to isomorphism from the closed, open and closed-open bar codes of $f$ via the results discussed in this paper.  
Actually the closed-open barcodes  determine the rank of the boundary maps  in these complexes.
More about can be found in the forthcoming book \cite {BUR}.
The precise relation between Reidemeister torsion, closed trajectories for a vector field with an angle valued map $f$ as Lyapunov and the Jordan cells and barcodes is the topic of \emph{work in preparation} under the name \emph{Alternative to Morse--Novikov theory}.

\subsection*{Applications}


The angle valued maps are as interesting and frequent as the real valued maps.
Observing/sampling an environment/shape from a central point in each direction should be as interesting and natural as observing the sublevel sets with respect to a real valued function of a shape.  

So far there are pleasant mathematical applications of the results in this paper and of the subsequent work, cf.\ \cite{B14,B15,B16,BUR}, in Computational Topology, Geometric Analysis and Dynamics.
\begin{itemize}
\item\emph{Computational Topology.}
Theorems~\ref{T1} and \ref{T4} imply precise relations between Betti numbers of $X$, Novikov Betti numbers and the monodromy, i.e., Jordan cells of a pair $(X;\xi_f)$.
They lead to computer implementable algorithms for the calculation of the last two without involving the infinite cyclic cover of $\xi_f$, a computer unfriendly object (being infinite even when $X$ is a finite simplicial complex), cf.~\cite{BD11} and \cite{B16}.
In particular, they lead to alternative methods to calculate the Alexander polynomial of knots and some Reidemeister torsions, to recognition of when $f\colon X\to\mathbb S^1$ is homotopic to a fibration with compact fiber, and in this case to the calculation of the Betti numbers of the fiber.
A paper on these type of results is in preparation.  
\item{\it Algebraic Topology of complements of complex hyper surfaces.}
The complement of a complex hyper surface in $\mathbb C^n$ comes equipped with a natural angle valued map.
The relevant algebraic topology invariants of this space are quite important in algebraic geometry. 
They can be express in therms of bar codes and Jordan cells and then are in principle computable.
Results in this direction are available in \cite{B15,B16,BUR}.
\item\emph{Geometric Analysis.}
The implementation of $C_r(f)$ to a configuration of mutually orthogonal spaces provides an orthogonal decomposition of the space of complex harmonic $(n-r)$-differential forms in subspaces, i.e., components, each subspace corresponding to the complex number represented by a closed $r$-barcode or an open $(r-1)$-barcode. 
For generic $f$ each component has dimension one. 
A pleasant consequences of this additional structure is the existence for a generic pair $(g,f)$ $g$-Riemannian metric, $f$ smooth map of a canonical base in the space of $r$-differential forms, analogous of the base provided by the trigonometric functions in the space of smooth functions on $\mathbb S^1$.
\item\emph{Dynamics.} 
The presence of Jordan cells (i.e.\ non-trivial monodromy) for a map $f\colon X\to \mathbb S^1$  implies the existence of closed trajectories for flows on $X$ for which $f\colon X\to \mathbb S^1$ is Lyapunov.
The non-triviality of $C^m_r(f)$ implies existence of instantons between rest points. 
More precisely, Theorem~\ref{T7} above permits to describe the rank of boundary map $\partial_r$ in the Morse complex or the Novikov complex relevant quantity in the counting of instantons. 
More can be found in Chapter~8 of the book in preparation \cite{BUR}. 
\end{itemize}

The authors thank the referee for a careful reading and useful observations.
The present version ows much to his critics and suggestions.

\section{Graph representations}\label{S2}

Fix a field $\kappa$.
Let $\Gamma$ be an oriented graph, possibly with infinitely many vertices.
A $\Gamma$-representation $\rho$ assigns to each vertex $x$ of $\Gamma$ a finite dimensional vector space $V_x$ and to each arrow $a\colon x\to y$ between two vertices $x$ and $y$ a linear map $\phi_a\colon V_x\to V_y$. 
Suppose $\rho'$ is another $\Gamma$-representation with vector spaces $V_x'$ and linear maps $\phi_a'\colon V_x'\to V_y'$.
A morphism from $\rho$ to $\rho'$ is a collection of linear maps $\psi_x\colon  V_x\to V_x'$ such that $\phi_a'\psi_x=\psi_y\phi_a$ for all arrows $a\colon x\to y$ between any two vertices $x$ and $y$.
More succinctly, a $\Gamma$-representation may be defined as a covariant functor from the (small) category generated by the graph $\Gamma$ to the (abelian) category of finite dimensional vector spaces.
A morphism of $\Gamma$-representations is just a natural transformation between two such functors.
Consequently, $\Gamma$-representations and morphisms between $\Gamma$-representations form an abelian category, see \cite{BD68,P73}.
In particular, the concepts of isomorphism (equivalence), direct sum, kernel, image, and short exact sequence are well defined for (morphisms between) $\Gamma$-representations.

Suppose $\rho_\alpha$, $\alpha\in\mathcal A$, is a family of $\Gamma$-representations with vector spaces $V_x^\alpha$ and linear maps $\phi^\alpha_a\colon V_x^\alpha\to V_y^\alpha$.
If, for every vertex $x$, all but finitely many of the vector spaces $V^\alpha_x$ are trivial, then one considers the $\Gamma$-representation $\bigoplus_{\alpha\in \mathcal A}\rho_\alpha$ which assigns to a vertex $x$ the vector space $\bigoplus_\alpha V^\alpha_x$ and to an arrow $a\colon x\to y$ the linear map $\bigoplus_\alpha\phi_a^\alpha\colon\bigoplus_\alpha V^\alpha_x\to\bigoplus_\alpha V^\alpha_y$.

A $\Gamma$-representation $\rho$ is called:
\emph{regular}, if all the linear maps $\phi_a$ are isomorphisms;
\emph{with finite support}, if $V_x=0$ for all but finitely many vertices; and
\emph{indecomposable}, if it is not isomorphic to the sum of two non-trivial representations. 

A standard result in abelian categories, see \cite[Theorem~1]{A56}, \cite[Chapter~5]{P73} or \cite[Theorem~6.45]{BD68}, formulated for $\Gamma$-representations with finite support, reads:

\begin{theorem}[Krull--Remak--Schmidt]
Any $\Gamma$-representation with finite support is isomorphic to a direct sum $\rho_1\oplus\cdots\oplus\rho_n$ with indecomposable summands $\rho_i$. 
Moreover, the components $\rho_i$ are unique up to isomorphism and reordering.
\end{theorem}

In this paper the oriented graph $\Gamma$ of primary concern will be $G_{2m}$ and for technical reasons we will need the infinite oriented graph $\mathcal Z$. 
The graph $\Gamma=G_{2m}$ has vertices $x_1,x_2,\dotsc,x_{2m}$ and arrows $a_i\colon x_{2i-1}\to x_{2i}$, $1\leq i\leq m$, and $b_i\colon x_{2i+1}\to x_{2i}$, $1\leq i\leq m-1$ and $b_m\colon x_1\to x_{2m}$,
see Figure~\ref{F:G2m}.
The graph $\Gamma=\mathcal Z$ has vertices $x_i$, $i\in\mathbb Z$, and arrows $a_i\colon x_{2i-1}\to x_{2i}$ and $b_i\colon x_{2i+1}\to x_{2i}$, see Figure~\ref{F:Z}.

\begin{figure}
$$
\xymatrix{
 &x_2\\
x_3\ar[ur]_{b_1}\ar[d]^{a_2}& & x_1\ar[ul]^{a_1}\ar[d]_{b_m}\\
x_4& & x_{2m}\\
x_{2m-3} \ar@{<.>}[u]\ar[dr]^{a_{m-1}}& & x_{2m-1}\ar[u]^{a_m}\ar[dl]_{b_{m-1}}\\
& x_{2m-2}
}
$$
\caption{The graph $G_{2m}$.}
\label{F:G2m}
\end{figure}
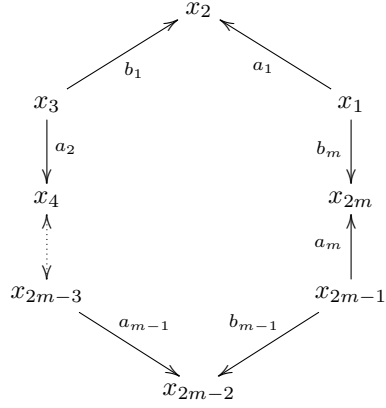

\begin{figure} 
$$
\xymatrix{
\cdots & x_{2i-1}\ar[l]_-{b_{i-1}} \ar[r]^-{a_i} & x_{2i} & x_{2i+1}\ar[l]_-{b_{i}} \ar[r]^-{a_{i+1}} & x_{2i+2} & \cdots \ar[l]_-{b_{i+1}}
}
$$
\caption{The graph $\mathcal Z$.}
\label{F:Z}
\end{figure}
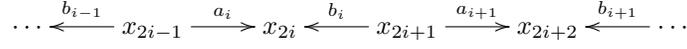

Both $G_{2m}$ and $\mathcal Z$-representations $\rho$ will be recorded as 
\begin{equation*}
\rho=\bigl\{V_r,\  \alpha_i\colon V_{2i-1}\to V_{2i},\  \beta_i\colon V_{2i+1}\to V_{2i} \bigr\}
\end{equation*}  
in the first case with $1\leq r\leq 2m$, $1\leq i \leq m$,  with the convention that $V_{2m+1}=V_1$, in the second case with $r,i\in \mathbb Z$.

Any regular $G_{2m}$-representation $\rho=\{V_r,\alpha_i,\beta_i\}$, not necessarily indecomposable, is equivalent i.e.\ isomorphic to the representation 
\begin{equation*}
\rho(V,T) := \bigl\{V'_r=V, \alpha'_1= T, \alpha'_i={\rm Id}\   i\ne 1, \  \beta'_i={\rm Id}\bigr\}
\end{equation*}
with $T=\beta^{-1}_m\cdot \alpha_m^{-1}\cdots \beta_1^{-1}\cdot\alpha_1$.
The isomorphism i.e.\ conjugacy class of the pair $(V,T)$ is called \emph{monodromy}.

According to the Krull--Remak--Schmidt theorem, every $G_{2m}$-representation $\rho$ decomposes as sum, $\rho\cong\rho'\oplus\rho''$, where $\rho''$ is regular and $\rho'$ has no non-trivial regular summand.
Moreover, both parts $\rho'$ and $\rho''$ are unique up to isomorphisms. 
The regular part  $\rho''$ provides  the \emph{monodromy} of $\rho$ which as pointed out above is determined by an isomorphism class of pairs $(V,T).$ 

The $\mathcal Z$-representations we consider are either with finite support or periodic. 
The representation is periodic if for some integer $N$, $V_r=V_{r+2N}$, $\alpha_i=\alpha_{i+N}$, $\beta_i=\beta_{i+N}$. 
Both type of $\mathcal Z$-representations, periodic and with finite support, as well as a finite direct sum of of such 
representations will be referred to as \emph{good} $\mathcal Z$-representations.

\subsection{The indecomposable $G_{2m}$-representations and  the indecomposable good $\mathcal Z$-representations.}

The indecomposable $G_{2m}$-representations are of two types, cf.~\cite[Section~4]{BD11}.  
In a slightly different formulation the identification below was first established in \cite{N73} and \cite{DF73}. 

\subsubsection*{Type I (bar codes)} 

These representations are labeled by the four types of intervals with integer valued ends $r$ and $s$, $r\leq s$, $1\leq r\leq m$, namely $[r,s]$ with $r\leq s$, and $(r,s)$, $[r,s)$, $(r, s]$ with $r<s$.
If $I$ is an interval of this form, then the corresponding representation will be denoted by $\rho^G(I)$.
More explicitly, they are denoted by $\rho^G (\{r, s\})$ 
with ``$\{$'' notation for either ``$[$'' or ``$($'' and ``$\}$'' for either ``$ ]$'' or ``$ )$'' and graphically described as follows.\footnote{A simpler  labeling 
 is possible but the one proposed is consistent with the geometric situation the representations are derived from.}  

Suppose the vertices $x_1,x_2,\dotsc,x_{2m-1},x_{2m}$ are located counter-clockwise on the unit circle, say at the the angles 
$t_1<\theta_1< t_2 < \theta_2 <\cdots <t_m < \theta_{m}$, with $0<t_1$ and $\theta_m \leq 2\pi$.

To describe the representation $\rho^G(\{i,j+m k\})$, $1\leq i,j\leq m$, draw the counterclockwise spiral curve from $a=\theta_i$ to $b=\theta_j+2\pi k$ with the ends a 
black or an empty circle to indicate ``closed'' or ``open'' interval. 
Black circle indicates that the end is on the spiral, empty circle that is not. 

The vector space $V_i$ is generated by the intersection points of the spiral with the radius corresponding to the vertex $x_i$ and $\alpha_i$ and $\beta_i$ are defined on generators as follows:
A generator $e$ of $V_{2i\pm1}$ is sent to the generator $e'$ of $V_{2i}$ if connected by a piece of spiral or to $0$ if not.
The spiral in Figure~\ref{barcode} below corresponds to $k=2$,  and defines the representation $\rho^G([i,j+2m))$.

\begin{figure}
\includegraphics[height=6cm]{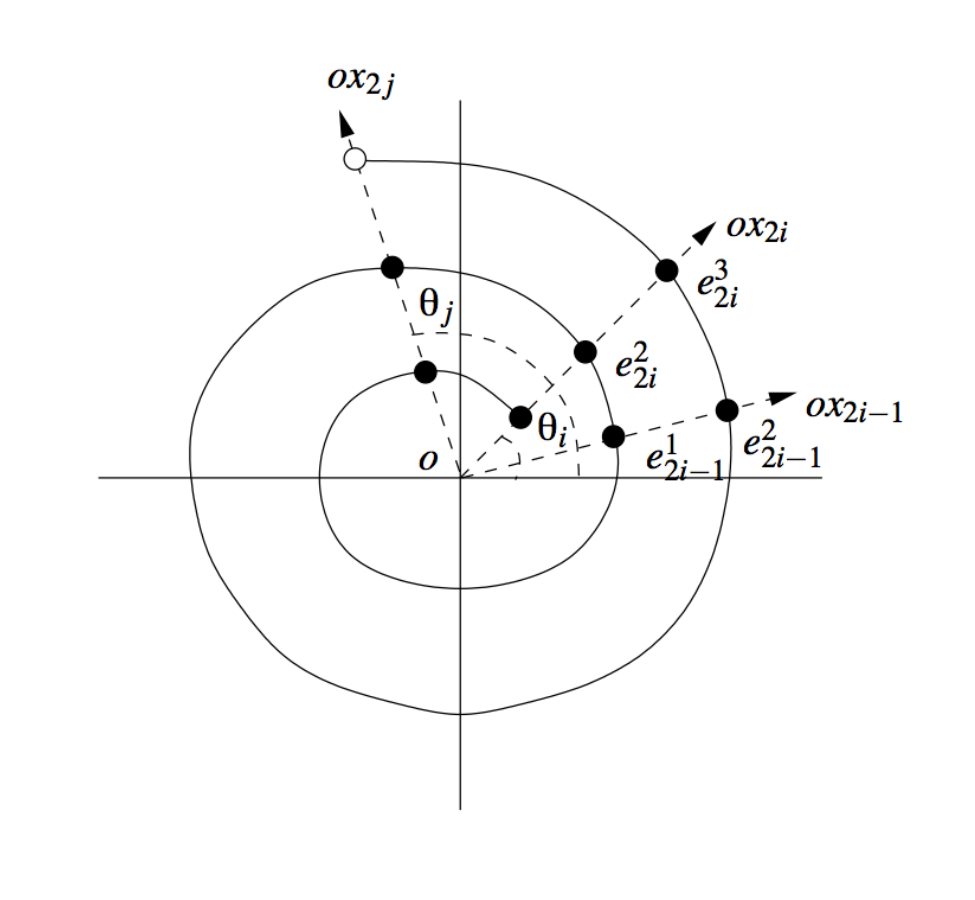}
\caption{The spiral for $[i, j+2m)$.}
\label{barcode}
\end{figure}

\subsubsection*{Type II (Jordan blocks / cells)}

They are labeled by Jordan blocks $J=(V,T)$ and denoted by $\rho^G(J)$. 
Recall that a Jordan block is an isomorphism class of indecomposable pairs $(V,T)$, $V$ a vector space $T\colon V\to V$ an isomorphism.  
The representation $\rho^G(J)$ has all vector spaces equal to $V$, $\alpha_1=T$ and $\beta_1=\alpha_i=\beta_i=\textrm{Id}$ for $2\leq i\leq m$. 
If $J=(\kappa^k, T(\lambda, k))$ we also write $\rho^G(J):=\rho^G(\lambda, k)$.

One refers to both the labeling interval $\{r,s\}$ and the representation $\rho^G(\{r,s\})$ as \emph{bar code} and to the indecomposable pair $J$ and the representation $\rho^G(J)$ as \emph{Jordan block}.

By the Krull--Remak--Schmidt theorem and the classification of indecomposables, any $G_{2m}$-representation $\rho$ can be decomposed as a sum of indecomposables,
\begin{equation}\label{E4}
\begin{aligned}
\rho\cong\bigoplus_{I\in \mathcal B(\rho)} \rho^G(I)\oplus \bigoplus_{J\in\mathcal J(\rho)} \rho^G(J).
\end{aligned}
\end{equation}
Here $\mathcal B(\rho)$ denotes the collection of all bar codes (with proper multiplicity) appearing in the decomposition of $\rho$, and $\mathcal J(\rho)$ denotes the collection of all Jordan blocks (with proper multiplicity) appearing in the decomposition of $\rho$.

We further decompose,
$$
\mathcal B(\rho)={\mathcal B}^c(\rho)\sqcup\mathcal B^o(\rho)\sqcup\mathcal B^{co}(\rho)\sqcup\mathcal B^{oc}(\rho)
$$
where ${\mathcal B}^c(\rho)$, $\mathcal B^o(\rho)$, $\mathcal B^{co}(\rho)$ and $\mathcal B^{oc}(\rho)$ denote the subcollections (with multiplicities) of barcodes with both ends closed, open, closed-open and open-closed, respectively.
For $\lambda\in\kappa\setminus 0$ one denotes by $\mathcal J_\lambda(\rho)$ the collection (with multiplicities) of Jordan blocks with eigenvalue $\lambda$.\footnote{If the 
linear map $T$ in the Jordan block $J=(V,T)$ has an eigenvalue $\lambda\in\kappa$ then this is the only eigenvalue, and $J$ is similar to $(\kappa^k,T(\lambda,k))$, see~\eqref{E1}.}

The indecomposable $\mathcal Z$-representations with finite support are all \emph{bar codes} indexed by four type of intervals $I$ with ends $i$ and $j$, $[i,j]$ with $i\leq j$,  
or $[i,j)$, $(i,j]$, $(i,j)$ with $i<j$ and denoted by $\rho^{\mathcal Z}(I)$. 
The only periodic indecomposable representation is denoted by $\rho^{\mathcal Z}_\infty$.
The representation denoted by $\rho^{\mathcal Z}(I)$ has all vector spaces equal to either $\kappa$ or $0$, 
the linear maps $\alpha_i$ and $\beta_j$ are equal to the identity if both, the source and the target, are non-trivial and zero otherwise.
Precisely, 
\begin{enumerate}[(i)]
\item 
$\rho^{\mathcal Z}([i,j])$, $i\leq j$ has $V_r=\kappa$ for $r=2i,\dotsc,2j$, and $V_r=0$ otherwise,
\item
$\rho^{\mathcal Z}([i,j))$, $i<j$ has $V_r=\kappa$ for $r=2i,\dotsc,2j-1$, and $V_r=0$ otherwise,
\item 
$\rho^{\mathcal Z}((i,j])$, $i<j$ has $V_r=\kappa$ for $r=2i+1,\dotsc,2j$, and $V_r=0$ otherwise,
\item 
$\rho^{\mathcal Z}((i,j))$, $i<j$ has $V_r=\kappa$ for $r=2i+1,\dotsc,2j-1$, and $V_r=0$ otherwise.
\end{enumerate}
Both, the labeling interval $I$ and the representation $\rho^{\mathcal Z}(I)$, will be referred to as \emph{bar code}.

The indecomposable representation $\rho^{\mathcal Z}_\infty$,  
has all vector spaces $V_r=\kappa$ and all linear maps $\alpha_i=\beta_i=\textrm{Id}$.

The Krull--Remak--Schmidt decomposition for representations with finite support extends to all good $\mathcal Z$-representations.  For the reader's convenience an argument is presented at the end of the next section since it involves the definition of {\it truncation}.

Precisely, any such (good) representation $\rho$ is a sum (in the sense described above) of possibly infinitely many indecomposables 
with finite support and finitely many copies of $\rho^{\mathcal Z}_\infty$,  
\begin{equation}\label{E5}
\rho\cong\bigoplus_{I\in\mathcal B(\rho)} \rho^{\mathcal Z}(I) \oplus \bigoplus_n\rho^\mathcal Z_\infty, 
\end{equation}
with indecomposable factors and their multiplicity unique up to isomorphism. 
Here $\mathcal B(\rho)$ the collection of all bar codes (with multiplicity) appearing in the decomposition, and $\bigoplus _n\rho^{\mathcal Z}_\infty$ denotes the sum of $n$ copies of $\rho^{\mathcal Z}_\infty$.
Each indecomposable $\rho^{\mathcal Z}(I)$ or $\rho^{\mathcal Z}_\infty$ appears with finite multiplicity. 
We let $\mathcal B^c(\rho)$, $\mathcal B^o(\rho)$, $\mathcal B^{co}(\rho)$ and $\mathcal B^{oc}(\rho)$ denote the subcollections (with multiplicities) of closed, open, closed-open and open-closed bar codes in $\mathcal B(\rho)$.
Moreover, $\mathcal J^{\mathcal Z}(\rho)$ denotes the collection of all copies of $\rho^{\mathcal Z}_\infty$ which appear as independent direct summands in $\rho$. The decomposition (\ref {E5}) is discussed in the next subsection.

In view of the above comments, statements about $G_{2m}$-representations or about good $\mathcal Z$-representations, formulated in this paper, will be verified first
for the indecomposable representations described above and if hold true, by the Krull--Remak--Schmidt decomposition theorem, concluded for arbitrary representations.

\subsection{Two basic constructions.}\label{SS:basic}

The \emph{infinite cyclic covering} of a $G_{2m}$-re\-pre\-sen\-ta\-tion $\rho=\{V_r, a_i, b_i, 1\leq r\leq 2m, 1\leq i\leq m\}$ is the periodic $\mathcal Z$-representation 
$\tilde \rho:=\{\tilde V_r, \tilde a_i, \tilde b_i, r, i\in \mathbb Z\}$ defined by $\tilde V_{r +2mk}=V_r$, $\tilde a_{i+km}= a_i$, and $\tilde b_{i+km}= b_i$. 
When applied to indecomposable $\rho^G(I)$ or $\rho^G(J)$, where $I$ denotes an interval and $J=(V,T)$ is a Jordan block, one obtains: 
\begin{equation}\label{E:qwerty1}
\begin{aligned}
\widetilde{\rho^G(I)}&=\bigoplus_{k\in \mathbb Z} \rho^{\mathcal Z}(I+m k)  \\
\widetilde{\rho^G(J)}&=\bigoplus_n \rho^{\mathcal Z}_\infty,  \qquad n=\dim V,\quad J=(V,T).
\end{aligned}
\end{equation}
Here $I+r$, $r\in \mathbb Z$ denotes the translate of the interval $I$, by $r$ units. 

The \emph{truncation} $T_{k,l}(\rho)$ of a $\mathcal Z$-representation $\rho$ is defined for any pair of integers $k,l$ with $k\leq l$.
If $\rho$ is a $G_{2m}$-representation, then the truncation $T_{k,l}(\rho)$ is defined for any pair of integers $k,l$ with $1\leq k\leq l\leq m$.
In either case, if $\rho=\{V_r, \alpha_i,\beta_i\}$ is a representation then the truncation is defined by $T_{k,l}(\rho):=\{V'_r, \alpha'_i, \beta'_i\}$ where:
\begin{equation}
\begin{aligned}
V'_r&=
\begin{cases} V_r & \textrm{if $2k\leq r\leq 2l$, and}\\ 0 & \textrm{otherwise.}\end{cases}
\\
\alpha'_r&=
\begin{cases} \alpha_r & \textrm{if $k+1\leq r\leq l$, and}\\ 0 & \textrm{otherwise.}\end{cases}
\\
\beta'_r&=
\begin{cases} \beta_r & \textrm{if $k\leq r\leq l-1$, and}\\  0 & \textrm{otherwise.}
\end{cases}
\end{aligned}
\end{equation} 
When applied to indecomposable $\mathcal Z$-representations one obtains
\begin{equation}\label{E:qwerty3}
\begin{aligned}
T_{k,l}(\rho^{\mathcal Z}_\infty)&= \rho^{\mathcal Z}([k,l]),\\
T_{k,l}(\rho^{\mathcal Z}(I))&=\rho^{\mathcal Z}(I\cap [k,l]),
\end{aligned}
\end{equation}
and when applied to indecomposable $G_{2m}$-representations one obtains 
\begin{equation}\label{E:qwerty2}
\begin{aligned}
T_{k,l}(\rho^G(I))&=\bigoplus_{r\in \mathbb Z}\rho^G(I_r), & I_r&=(I + rm)\cap [k,l],
\\
T_{k,l}(\rho^G(J))&=\bigoplus_n {\rho^G([k,l])}, & n&=\dim V,\quad J=(V,T).
\end{aligned}
\end{equation}
Here $I+rm$ denotes the translate of the interval $I$ to the right by $rm$ units.

 
Given a $G_{2m}$-representation $\rho$ one writes:  
$\tilde {\mathcal J} (\rho)$ for the collection which contains with any Jordan block $J=(V,T)\in \mathcal J(\rho)$, a number of $n(J)=\dim (V)$ copies of $\rho^{\mathcal Z}_\infty$ hence a total of 
$\sum_{J=(V,T)\in \mathcal J(\rho)} \dim V$ copies of $\rho^{\mathcal Z}_\infty$, 
and $\tilde{\mathcal B}^{-}(\rho):= \{  I+2\pi k \mid I\in \mathcal B^{-}(\rho), k\in \mathbb Z\}$ with $\tilde{\mathcal B}^{-}$  any of 
$\tilde{\mathcal B}$, $\tilde{\mathcal B}^c$, $\tilde{\mathcal B}^o$, $\tilde{\mathcal B}^{co}$, $\tilde{\mathcal B}^{oc}$.

In terms of this notation is convenient to keep in mind the following book-keeping. 


\begin{lemma}\label{O41}
(a) If $\rho$ is a $G_{2m}$-representation then 
$$
\mathcal B(\tilde \rho)=\tilde{\mathcal B}(\rho),\qquad\mathcal J(\tilde \rho)=\tilde{\mathcal J}(\rho),
$$
$$
\mathcal B^c(\tilde \rho)=\tilde{\mathcal B}^c(\rho),\quad
\mathcal B^o(\tilde \rho)=\tilde{\mathcal B}^o(\rho),\quad
\mathcal B^{co}(\tilde \rho)=\tilde{\mathcal B}^{co}(\rho),\quad
\mathcal B^{oc}(\tilde \rho)=\tilde{\mathcal B}^{oc}(\rho),
$$
and:
\begin{equation*}
\begin{aligned}
\mathcal B^c(T_{k,l}(\rho))&=\{I\cap[k,l]:\text{$I\in\tilde{\mathcal B}(\rho)$ such that $I\cap [k, l]$ is non-empty and closed}\}
\\&\qquad \sqcup \{\text{$[k,l]$ with multiplicity $\sharp\tilde{\mathcal J}(\rho)$}\},
\\
\mathcal B^o(T_{k,l}(\rho))&=\{I\in\tilde{\mathcal B}^o(\rho): I\subset [k,l] \}
\\
\mathcal J(T_{k,l}(\rho))&=\emptyset.
\end{aligned}
\end{equation*}

(b) If $\rho$ is a good $\mathcal Z$-representation then:
\begin{equation*}
\begin{aligned}
\mathcal B^c(T_{k,l}(\rho))&=\{I\cap[k,l]:\text{$I\in \mathcal B(\rho)$ such that $I\cap [k,l]$ is non-empty and closed}\}
\\&\qquad \sqcup \{\textrm{$[k,l]$ with multiplicity $\sharp{\mathcal J}(\rho)$}\},
\\
\mathcal B^o(T_{k,l}(\rho))&=\{I\in \mathcal B^o(\rho) : I\subset [k,l] \},
\\
\mathcal J(T_{k,l}(\rho))&=\emptyset.
\end{aligned}
\end{equation*}
\end{lemma} 

\begin{proof}
To see (a) observe that infinite cyclic covering and truncation are both additive constructions, that is to say,
$\widetilde{\rho_1\oplus\rho_2}=\tilde\rho_1\oplus\tilde\rho_2$ and $T_{k,l}(\rho_1\oplus\rho_2)=T_{k,l}(\rho_1)\oplus T_{k,l}(\rho_2)$ for any two $G_{2m}$-representations $\rho_1$ and $\rho_2$.
The expressions for $\mathcal B(\rho)$, $\mathcal B^c(\rho)$, $\mathcal B^o(\rho)$, $\mathcal B^{co}(\rho)$, $\mathcal B^{oc}(\rho)$, and $\mathcal J(\rho)$ thus follow immediately from \eqref{E:qwerty1}.
Similarly, the expressions for $\mathcal B^c(T_{k,l}(\rho))$, $\mathcal B^o(T_{k,l}(\rho))$, and $\mathcal J(T_{k,l}(\rho))$ follow from \eqref{E:qwerty2}.
Since the truncation is also additive for good $\mathcal Z$-representations, part (b) follows from \eqref{E:qwerty3}.
\end{proof}

{\it  Krull-Remak-Schmidt decomposition for good $\mathcal Z-$ representations:}

If $\rho$ has finite support this is the standard Krull-Remak-Schmidt decomposition theorem in an abelian category.

If $\rho$ is periodic it is isomorphic to some $\tilde \rho'$ with $\rho'$ a $G_{2m}-$representation. neither $m$ nor $\rho'$ is unique.  Clearly a decomposition of $\rho'$ as sum of the barcode-representations $I'_1$  with multiplicity $r'_1$, $I_2'$ with multiplicity $r'_2$ $\dots$  $I'_{N'}$ with multiplicity $r'_{N'_1}$ and Jordan blocks  whose total dimension of the underlying vector space $n'_1$ provides a decomposition of $\rho$ as an infinite sum  of $I'_1 +mk$  with multiplicity $r'_1$, $I_2'+mk$ with multiplicity $r'_2$ $\dots$  $I'_{N'}+mk$ with multiplicity $r'_{N'_1}$ for any $k\in \mathbb Z$ and $n'_1$ copies of $\rho^{match Z}_\infty.$  This implies the existence of decomposition as stated in (\ref{E5}).

Note that 
for any decomposition of type (\ref {E5}) the following holds:

- there are only finitely many barcodes up to translation by multiples of $m$  which makes the length of bar codes bounded from above,

- each barcode appears with finite multiplicity and  

- there are finitely many components $\rho^{\mathcal Z}_\infty$. 

\noindent Since a truncation $T_{i,j}$ converts a barcode into a barcode (possibly empty) and $\rho^{\mathcal Z_\infty}$ into a closed barcode $[i,j],$  comparing the outcome of enough many truncation $T_{i,j}$ (with $(j-i)$ larger than the length of the barcodes of the  two representations) and in view of the validity of the Krull-Remak-Schmidt theorem for finite graphs, on obtains the equality in the number of each type of barcodes  and of the number of components $\rho^{\mathcal Z}_\infty$ in any two decompositions (\ref {E5}).

\subsection{The matrix $M(\rho)$ and the representation $\rho_u$}\label{SS:rhou}

For every $G_{2m}$-representation $\rho=\{V_r, \alpha_i, \beta_i\}$, $1\leq r\leq 2m$, $1\leq i\leq m$, we introduce a linear map,
$M(\rho)\colon\bigoplus_{1\leq i\leq m}V_{2i- 1}\to\bigoplus_{1\leq i\leq m}V_{2i}$, defined by the block matrix:
\begin{equation*}
\begin{pmatrix}
\alpha_1 & -\beta_1 & 0        & \dots        & 0            \\
    0    & \alpha_2 & -\beta_2 & \ddots       & \vdots       \\
\vdots   & \ddots   & \ddots   & \ddots       & 0            \vspace{0.5ex} \\
0        & \dots    & 0        & \alpha_{m-1} & -\beta_{m-1} \vspace{1ex}   \\
-\beta_m & 0        & \dots    & 0            & \alpha_m
\end{pmatrix}.
\end{equation*}
Moreover, for $u\in\kappa\setminus0$ we let $\rho_u=\{V'_r,\alpha'_i,\beta'_i\}$ denote the $G_{2m}$-representation 
where $V'_r=V_r$, $\alpha'_1=u \alpha_1$, $\alpha'_i=\alpha_i$ for $i\ne1$ and $\beta'_i=\beta_i$. 

For a $\mathcal Z$-representation $\rho=\{V_r,\alpha_i,\beta_i\}$ the linear map $M(\rho)\colon\bigoplus_{i\in \mathbb Z}V_{2i-1}\to\bigoplus_{i\in \mathbb Z} V_{2i}$, 
is defined by the infinite block matrix  with entries:
\begin{equation*}
M(\rho)_{2r-1, 2s}=
\begin{cases} 
\alpha_r    & \textrm{if $s=r$,}\\
\beta_{r-1} & \textrm{if $s=r-1$, and}\\
0           & \textrm{otherwise.}
\end{cases}
\end{equation*}

If the $\mathcal Z$-representation $\tilde \rho$ is the infinite cyclic covering of a $G_{2m}$-representation $\rho$, then the shift $i\mapsto i+m$ 
defines isomorphisms $t_{\textrm{odd}}\colon\bigoplus_{i\in \mathbb Z}V_{2i-1}\to \bigoplus_{i\in \mathbb Z}V_{2(i+m)-1}$ and $t_{\textrm{even}}\colon\bigoplus_{i\in \mathbb Z}V_{2i}\to \bigoplus_{i\in \mathbb Z}V_{2(i+m)}$ such that $t_\textrm{even}\circ M(\tilde\rho)=M(\tilde\rho)\circ t_\textrm{odd}$.
The induced automorphisms on $\ker M(\tilde\rho)$ and $\coker M(\tilde\rho)$ will be denoted by:
\begin{equation}\label {EQ8}
t\colon\ker M(\tilde \rho)\to \ker M(\tilde \rho)\qquad\text{and}\qquad
t\colon\coker M(\tilde \rho)\to \coker M(\tilde \rho).
\end{equation}

For every $\Gamma$-representation $\rho$ introduce an $\mathbb N_0$-valued function $\dim\rho$ on the set of vertices of $\Gamma$, defined by $\dim\rho(x):=\dim V_x$, where $V_x$ is the vector space assigned to the vertex $x$ by $\rho$.
Moreover, for every representation $\rho$ of $G_{2m}$ or $\mathcal Z$ we put $\dim \ker (\rho):=\dim\ker M(\rho)$ and $\dim \coker(\rho):=\dim\coker M(\rho)$.

As noticed in \cite{BD11} one has:

\begin{lemma}[\cite{BD11}] 
Suppose $\rho$, $\rho_1$ and $\rho_2$ are representations of $G_{2m}$ or $\mathcal Z$.
Then the following hold true for $\lambda,u\in\kappa\setminus0$, $k\in\mathbb N$, and all intervals $I$ with integral endpoints:
\begin{enumerate}[(a)]
\item $\dim(\rho_u)= \dim(\rho)$.
\item $(\rho_1 \oplus \rho_2)_u= (\rho_1)_u\oplus (\rho_2)_u$.
\item $\rho^G(\lambda,k)_u= \rho^G(u\lambda, k)$.
\item $\rho^G(I)_u=\rho^G(I)$.
\item $\dim(\rho_1\oplus \rho_2)= \dim(\rho_1) + \dim(\rho_2)$.  
\item $\dim \ker (\rho_1\oplus \rho_2)= \dim \ker (\rho_1) + \dim \ker (\rho_2)$.
\item $\dim \coker(\rho_1\oplus \rho_2)= \dim \coker(\rho_1) + \dim\coker(\rho_2)$.
\end{enumerate}
\end{lemma}

\begin{proof}
The statements in parts (a) and (b) are trivial.
Part (c) follows from the fact that the Jordan block $T(u\lambda,k)$, see \eqref{E1}, is conjugate to $uT(\lambda,k)$ via the diagonal matrix $\operatorname{diag}(1,u,u^2,\dotsc,u^{k-1})$.
Part (d) readily follows from the classification of indecomposable $G_{2m}$-representations.
The statement in (e) is obvious.
To see (f) and (g) note that $M(\rho_1\oplus\rho_2)=M(\rho_1)\oplus M(\rho_2)$, hence $\ker(\rho_1\oplus\rho_2)=\ker(\rho_1)\oplus\ker(\rho_2)$ and $\coker(\rho_1\oplus\rho_2)=\coker(\rho_1)\oplus\coker(\rho_2)$.
For $G_{2m}$-representations the latter can be found in \cite[Proposition~4.1]{BD11}.
\end{proof}

Moreover: 

\begin{proposition}[\cite{BD11}]\label {P23}\

\begin{enumerate}[(a)]
\item
For indecomposable $G_{2m}$-representations of type I we have
\begin{enumerate}[({a}1)]
\item $\dim \ker \rho^G([i,j])=0$, $\dim \coker \rho^G([i,j])=1$,
\item $\dim \ker \rho^G([i,j))=0$, $\dim \coker \rho^I([i,j))=0$,
\item $\dim \ker \rho^G((i,j])=0$, $\dim \coker \rho^G((i,j])=0$,
\item $\dim \ker \rho^G((i,j))=1$, $\dim \coker \rho^G((i,j))=0$,
\end{enumerate}
and for indecomposable $\mathcal Z$-representations with finite support:
\begin{enumerate}[({a}1)]\setcounter{enumii}{4}
\item $\dim \ker\rho^{\mathcal Z}([i,j])=0$, $\dim \coker\rho^{\mathcal Z}([i,j])=1$,
\item $\dim \ker\rho^{\mathcal Z}([i,j))=0$, $\dim \coker\rho^{\mathcal Z}([i,j))=0$,
\item $\dim \ker\rho^{\mathcal Z}((i,j])=0$, $\dim \coker\rho^{\mathcal Z}((i,j])=0$,
\item $\dim \ker\rho^{\mathcal Z}((i,j))=1$, $\dim \coker\rho^{\mathcal Z}((i,j))=0$.
\end{enumerate}

\item 
For indecomposable $G_{2m}$-representations of type II we have
\begin{enumerate}[({b}1)]
\item $\dim \ker\rho^G(J)=0$ if $J\ne (\kappa^k, T(1,k))$;  $\dim \ker\rho^G(\kappa^k, T(1,k))=1$ 
\item $\dim \coker\rho^G(J)=0$ if $J\ne (\kappa^k, T(1,k))$;  $\dim \coker\rho^G(\kappa^k, T(1,k))=1$
\end{enumerate}
and for the $\mathcal Z$-representation $\rho^{\mathcal Z}_\infty$:
\begin{enumerate}[({b}1)]\setcounter{enumii}{2}
\item $\dim \ker (\rho^{\mathcal Z}_\infty)=0$,
\item $\dim \coker (\rho^{\mathcal Z}_\infty)=1$.
\end{enumerate}
\end{enumerate}
\end{proposition}

\begin{proof}
The statements in (a1), (a2), (a3), (a4), (b1), and (b2) can be found in \cite[Proposition~4.3]{BD11}.
Parts (a5), (a6), (a7), (a8), (b3), and (b4) can be proved analogously.
Indeed, the calculation of the kernel of $M(\rho)$ reduces to the description of the space of solutions of the linear system:
\begin{equation*}
\begin{aligned}
\alpha_1(v_1)&=\beta_1(v_3)\\
\alpha_2(v_3)&=\beta_2(v_5)\\
&\ \, \vdots\\
\alpha_m(v_{2m-1})&=\beta_m(v_1)
\end{aligned}
\end{equation*}
wit $v_{2i-1}\in V_{2i-1}.$  This is straight forward for indecomposable representations.
\end{proof}


\begin{lemma}\label{OO} 
If $\rho=\{V_i,\alpha_i,\beta_i\}$ is a regular $\mathcal Z$-representation, i.e.\ all $\alpha_i$ and $\beta_i$ are isomorphisms,
then $\ker M(\rho)=0$, and for every $i$ the canonical inclusion 
$V_{2i}\to\bigoplus_{r\in \mathbb Z} V_{2r}$ followed by the projection onto $\coker M(\rho)$ provides an isomorphism $V_{2i}\cong\coker M(\rho)$.
\end{lemma}

\begin{proof}
By regularity, the system of equations $\alpha_r(v_{2r-1})=\beta_r(v_{2r+1})$ does not have a non-trivial solution for which only finitely many of the $v_{2r-1}\in V_{2r-1}$ are non-trivial, whence $\ker M(\rho)=0$.
To see that $V_{2i}$ intersects the image of $M(\rho)$ trivially, suppose $w\in V_{2i}\cap\img M(\rho)$.
Then there exist $v_{2r-1}\in V_{2r-1}$, almost all zero, such that $w=\alpha_i(v_{2i-1})-\beta_i(v_{2i+1})$ and $\alpha_r(v_{2r-1})=\beta_r(v_{2r+1})$ for all $r\neq i$.
As $v_{2r+1}=0$ for sufficiently large $r$, the latter equations imply $v_{2i+1}=0$.
Similarly, since $v_{2r-1}=0$ for sufficiently small $r$, we obtain $v_{2i-1}=0$, whence $w=0$.
Finally, if $u\in V_{2r}$, then the corresponding element in $\coker M(\rho)$ can also be represented by by an element in $V_{2r+2}$, namely $\alpha_{r+1}\beta_r^{-1}(u)$.
Consequently, each element in $\coker M(\rho)$ can be represented by an element in $V_{2i}$.
\end{proof}

To formulate a refinement of Proposition~\ref{P23} we introduce additional notation:

\begin{definition}\label {D26}\
For a set $S$ denote by $\kappa[S]$  the vector space generated by $S$, i.e.\ 
the vector space of $\kappa$-valued maps on $S$ with finite support, and by $\kappa[[S]]$  
the vector space of all $\kappa$-valued maps on $S$.  
If $S$ is finite, then $\kappa[S]=\kappa[[S]]$. 

For two subsets $S_1$ and $S_2$ of $S$ the canonical linear maps $\kappa[S_1]\to \kappa[S_2]$, $\kappa[S_1]\to \kappa[[S_2]]$,
or $\kappa[[S_1]]\to \kappa[[S_2]]$ are the unique linear maps which restrict to the identity on $S_1\cap S_2$ and to zero on $S_1\setminus S_2$. 
\end{definition}

We warn the reader of the ``unfortunate notational similarity'' between $\kappa[S]$ and $\kappa[T^{-1},T]$ with the last one denoting the ring of Laurent polynomials of variable $T$.  
Fortunately they appear below in contexts which exclude confusion.

If $\rho=\{V_r,\alpha_i,\beta_i\}$ is a $G_{2m}$-representation, then the diagram
\begin{equation}\label{E:VVV}
\vcenter{
\xymatrix{
\bigoplus_{k<i\leq l} V_{2i-1} \ar[r] \ar[d]^{M(T_{k,l}(\rho))} & \bigoplus _{k'<i\leq l'}V_{2i-1} \ar[r] \ar[d]^{M(T_{k',l'}(\rho))} & \bigoplus_i V_{2i-1}\ar[d]^{M(\rho)}&
\\
\bigoplus_{k\leq i\leq l} V_{2i}   \ar[r] & \bigoplus_{k'\leq i\leq l'} V_{2i} \ar[r] & \bigoplus_i V_{2i}&
}}
\end{equation}
commutes, for all integers $1\leq k'\leq k\leq l\leq l'\leq m$, and we obtain induced linear maps
\begin{equation}\label{E:iii}
\xymatrix{
\ker M(T_{k,l}(\rho))\ar[r]^-{i}      &\ker M(T_{k',l'}(\rho))\ar[r]^-{i'}     &\ker M(\rho)&}
\end{equation}
as well as
\begin{equation}\label{E:jjj}
\xymatrix{\coker M(T_{k,l}(\rho))    \ar[r]^-{j}      &\coker M(T_{k',l'}(\rho))     \ar[r]^-{j'}     &\coker M(\rho).&}
\end{equation}
The same holds true if $\rho$ is a good $\mathcal Z$-representation and $k'\leq k\leq l\leq l'$.
For either representation, the linear maps $i$ and $i'$ in \eqref{E:iii} are injective since the horizontal inclusions in diagram \eqref{E:VVV} are injective.
The maps $j$ and $j'$ in \eqref{E:jjj} need not be injective in general.
Correspondingly, see Definition~\ref{D26}, we consider the linear maps
\begin{equation}\label{E19a}
\xymatrix{
\kappa[\mathcal B^o(T_{k,l}(\rho))] \ar[r] &\kappa[\mathcal B^o(T_{k',l'}(\rho))] \ar[r] &\kappa[\mathcal B^o(\rho)\sqcup \mathcal J]}
\end{equation}
and
\begin{equation}\label{E19}
\xymatrix{
\kappa[\mathcal B^c(T_{k,l}(\rho))
] \ar[r] &\kappa[\mathcal B^c(T_{k',l'}(\rho))
] \ar[r] &\kappa[\mathcal B^c(\rho)\sqcup \mathcal J]}
\end{equation} 
where $\mathcal J$ is defined as follows.
In \eqref{E19a} $\mathcal J=\emptyset$ if $\rho$ is a good $\mathcal Z$-representation and $\mathcal J= \mathcal J_1(\rho)$ if $\rho$ is a $G_{2m}$-representation. 
In \eqref{E19} $\mathcal J=\mathcal J(\rho)$ if $\rho$ is a good $\mathcal Z$-representation and  $\mathcal J= \mathcal J_1(\rho) $ if $\rho$ is a $G_{2m}$-representation. 
The linear maps in \eqref{E19a} are injective since, according to Lemma~\ref{O41}, we have inclusions 
$\mathcal B^o(T_{k,l}(\rho))\subseteq \mathcal B^o(T_{k',l'}(\rho))\subseteq\mathcal B^o(\rho)\subseteq \mathcal B^o(\rho)\sqcup \mathcal J$.
Lemma~\ref{O41} also shows that the linear maps in \eqref{E19} will not be injective in general.
Using decompositions as in \eqref{E4} and \eqref{E5}, additivity of the constructions, and Proposition~\ref{P23}, we see that there are linear isomorphisms
\begin{equation*}
\begin{aligned}
\kappa [\mathcal B^o(T_{k,l}(\rho))]&\cong\ker M(T_{k,l}(\rho)),&
\kappa [\mathcal B^o(\rho)\sqcup \mathcal J]&\cong\ker M(\rho),
\\
\kappa [\mathcal B^c(T_{k,l}(\rho))]&\cong\coker M(T_{k,l}(\rho)),&
\kappa [\mathcal B^c(\rho)\sqcup \mathcal J]&\cong\coker M(\rho),
\end{aligned}
\end{equation*} 
and that the ranks of the linear maps in \eqref{E:iii} and \eqref{E:jjj} coincide with the ranks of the corresponding linear maps in \eqref{E19a} and \eqref{E19}, respectively.
The following refinement of Proposition~\ref{P23} asserts that these isomorphisms may even be chosen to be compatible with truncation.

\begin{proposition}\label{O45}
(a) Let $\rho$ be a $G_{2m}$-representation. Then every decomposition $\rho=\bigoplus_{I\in\mathcal B(\rho)}\rho^G(I)\oplus\bigoplus_{J\in\mathcal J(\rho)}\rho^G(J)$  
induces isomorphisms $\Psi^o$, $\Psi^o_{k,l}$, $\Psi^c$, and $\Psi^c_{k,l}$ such that the diagrams
\begin{equation}\label{E13G}
\vcenter{
\xymatrix{
\ker M(T_{k,l}(\rho))\ar[r]^-{i}     &\ker M(T_{k',l'}(\rho))\ar[r]^-{i'}     &\ker M(\rho)
\\
\kappa[\mathcal B^o(T_{k,l}(\rho))]\ar[u]_\cong^{\Psi^o_{k,l}} \ar[r] &\kappa[\mathcal B^o(T_{k',l'}(\rho))] \ar[u]_\cong^{\Psi^o_{k',l'}}\ar[r] &\kappa[\mathcal B^o(\rho)\sqcup \mathcal J_1(\rho)]\ar[u]_\cong^{\Psi^o}
}} 
\end{equation} 
and
\begin{equation}\label{E14G}
\vcenter{
\xymatrix{
\coker M(T_{k,l}(\rho))\ar[r]^-{j}     &\coker M(T_{k',l'}(\rho))\ar[r]^-{j'}     &\coker M(\rho)
\\
\kappa[\mathcal B^c(T_{k,l}(\rho))]\ar[u]_\cong^{\Psi^c_{k,l}} \ar[r] &\kappa[\mathcal B^c(T_{k',l'}(\rho))] \ar[u]_\cong^{\Psi^c_{k',l'}}\ar[r] &\kappa[\mathcal B^c(\rho)\sqcup \mathcal J_1(\rho)].\ar[u]_\cong^{\Psi^c} 
}} 
\end{equation}
commute for all integers $1\leq k'\leq k\leq l\leq l'\leq m$.
 
(b) Let $\rho$ be a good $\mathcal Z$-representation. Then every decomposition $\rho=\bigoplus_{I\in\mathcal B(\rho)}\rho(I)\oplus\bigoplus_n\rho^{\mathcal Z}_\infty$, 
where $n=\sharp J(\rho)$, induces isomorphisms $\Psi^o$, $\Psi^o_{k,l}$, $\Psi^c$, and $\Psi^c_{k,l}$ such that the diagrams
\begin{equation}\label{E13Z}
\vcenter{
\xymatrix{
\ker M(T_{k,l}(\rho))\ar[r]^-{i}     &\ker M(T_{k',l'}(\rho))\ar[r]^-{i'}     &\ker M(\rho)
\\
\kappa[\mathcal B^o(T_{k,l}(\rho))]\ar[u]_\cong^{\Psi^o_{k,l}} \ar[r] &\kappa[\mathcal B^o(T_{k',l'}(\rho))] \ar[u]_\cong^{\Psi^o_{k',l'}}\ar[r] &\kappa[\mathcal B^o(\rho)]\ar[u]_\cong^{\Psi^o}
}} 
\end{equation} 
and
\begin{equation}\label{E14Z}
\vcenter{
\xymatrix{
\coker M(T_{k,l}(\rho))\ar[r]^-{j}     &\coker M(T_{k',l'}(\rho))\ar[r]^-{j'}     &\coker M(\rho)
\\
\kappa[\mathcal B^c(T_{k,l}(\rho))]\ar[u]_\cong^{\Psi^c_{k,l}} \ar[r] &\kappa[\mathcal B^c(T_{k',l'}(\rho))] \ar[u]_\cong^{\Psi^c_{k',l'}}\ar[r] &\kappa[\mathcal B^c(\rho)\sqcup \mathcal J(\rho)].\ar[u]_\cong^{\Psi^c} 
}} 
\end{equation}
commute for all integers $k'\leq k\leq l\leq l'$.
\end{proposition}

\begin{proof}
Since the involved constructions are all additive, that is to say, compatible with direct sums of representations, it suffices to construct the isomorphisms $\Psi^o$, $\Psi^o_{k,l}$, $\Psi^c$, and $\Psi^c_{k,l}$for indecomposable representations.
For indecomposable representations, however, the constructions are tautological in view of Proposition \ref{P23}. 
\end{proof}

We close this section with an observation about the infinite cyclic covering associated with a $G_{2m}$-representation.
Let $\rho=\{V_r,\alpha_i,\beta_i\}$ be a $G_{2m}$-representation and let $\tilde\rho=\{\tilde V_r,\tilde\alpha_i,\tilde\beta_i\}$ denote the associated infinite cyclic covering $\mathcal Z$-representation, cf.\ the beginning of Section~\ref{SS:basic}.
Recall that the shift by $m$ induces automorphisms denoted by $t$ on $\ker M(\tilde\rho)$ and $\coker M(\tilde\rho)$, see \eqref{EQ8}.
These automorphisms turn $\ker M(\tilde\rho)$ and $\coker M(\tilde\rho)$ into $\kappa[T^{-1},T]$-modules such that $T$ acts by $t$ and $T^{-1}$ acts by $t^{-1}$.
Appendix~\ref{SS10} contains some basic facts on $\kappa[T,T^{-1}]$-modules.

Correspondingly, the translation of intervals, $I\mapsto I+m$, induces bijections on $\mathcal B^o(\tilde\rho)$ and $\mathcal B^c(\tilde\rho)$, see \eqref{E:qwerty1}.
The induced automorphisms on $\kappa[\mathcal B^o(\tilde\rho)]$ and $\kappa[\mathcal B^c(\tilde\rho)]$ turn these two vector spaces into $\kappa[T^{-1},T]$-modules.
Moreover, identifying $\kappa[\mathcal J(\tilde\rho)]=\bigoplus_{(V,T)\in\mathcal J(\rho)}V$, we obtain an automorphism $\bigoplus_{(V,T)\in\mathcal J(\rho)}T$ on $\kappa[\mathcal J(\tilde\rho)]$ which we use to turn this vector space into a $\kappa[T^{-1},T]$-module. 
Via $\kappa[\mathcal B^c(\tilde\rho)\sqcup\mathcal J(\tilde\rho)]=\kappa[\mathcal B^c(\tilde\rho)]\oplus\kappa[\mathcal J(\tilde\rho)]$, we obtain a $\kappa[T^{-1},T]$-module structure on $\kappa[\mathcal B^c(\tilde\rho)\sqcup\mathcal J(\tilde\rho)]$.

\begin{lemma}\label{OOO}
Let $\rho$ be a $G_{2m}$-representation and let $\tilde\rho$ denote the associated infinite cyclic covering $\mathcal Z$-representation.
Then the following hold true:

a) The linear isomorphisms 
$$
\Psi^o\colon\kappa[\mathcal B^o(\tilde\rho)]\to\ker M(\tilde\rho)\quad\text{and}\quad\Psi^c\colon\kappa[\mathcal B^c(\tilde\rho)\sqcup\mathcal J(\tilde\rho)]\to\coker M(\tilde\rho)
$$ 
in Proposition~\ref{O45}(b) 
are isomorphisms of $\kappa[T^{-1},T]$-modules (since by construction compatible with the $m-$periodicity).   

b) The modules $\kappa[\mathcal B^o(\tilde\rho)]$ and $\kappa[\mathcal B^c(\tilde\rho)]$ are free.
More precisely, we have isomorphisms of $\kappa[T^{-1},T]$-modules
$$
\kappa[\mathcal B^o(\tilde\rho)]\cong\kappa[T^{-1},T][\mathcal B^o(\rho)]\quad\text{and}\quad\kappa[\mathcal B^c(\tilde\rho)]\cong\kappa[T^{-1},T][\mathcal B^c(\rho)].
$$

c) The torsion part of the $\kappa[T^{-1},T]$-module $\coker M(\tilde\rho)$ equipped with the automorphism induced by $T$ is isomorphic to the monodromy of the representation $\rho$, that is, $\kappa[\mathcal J(\tilde\rho)]\cong\bigoplus_{J\in\mathcal J(\rho)}J$.
\end{lemma}

\begin{proof}
Since the statement is additive in the $G_{2m}$-representation $\rho$, it suffices to consider indecomposable $G_{2m}$-representations $\rho$.
Part c) follows from Lemma~\ref{OO}. Indeed if $\rho$ is a barcode representation the result follows from Proposition \ref {P23}. In this case both $\ker(M(\tilde \rho))$ and $\coker(M(\tilde \rho))$ are free of rank $1$ or $0.$ and there is no regular part of $\rho.$ If $\rho$  is a Jordan block, hence $\rho$ is regular, the result follows from Lemma \ref{OO}. 
\end{proof}

\section{Bar codes and Jordan blocks via graph representations}\label{SS3}

In this section we will describe graph representations associated with a tame circle valued map.
Furthermore, we will establish fundamental exact sequences that permit to compute the (twisted) homology of the underlying space in terms of the corresponding barcodes and Jordan blocks.

Let $f\colon X\to S^1$ be a tame map and $0<\theta_1<\theta_2<\cdots<\theta_m\leq 2\pi$ be the critical angles (the angles of the set $\Sigma$ in the definition of tameness). 
Choose the regular values $t_1<t_2<\cdots<t_m$ with $\theta_{i-1}<t_i<\theta_i$ and $0<t_1<\theta_1$.  
In order to differentiate between regular and singular fibers we write $R_i:=f^{-1}(t_i)$ and $X_i:= f^{-1}(\theta_i)$.

The tameness of $f$ induces the maps $a_i\colon R_i\to X_i$ for $1\leq i\leq m$,  $b_i\colon R_{i+1}\to X_i$ for $i\leq m-1$ and $b_m\colon R_1\to X_m$ which are unique up to homotopy; 
this means that different choices of the regular values, say $t'_i$ instead of $t_i$, lead to homotopy equivalences $\omega_i\colon R_i\to R'_i$ s.t.\ 
$a'_i\cdot \omega_i$ is homotopic to $a_i$ and $b'_i\cdot \omega_i$ is homotopic to $b_i$.

Indeed the fiber $R_i$ identifies up to homotopy to regular fibers $f^{-1}(t)$ and $f^{-1}(t')$, $\theta_{i-1}< t <t' <\theta_i$ since $f^{-1}(\theta_{i-1},\theta_i)\to (\theta_{i-1},\theta_i)$ is a fibration. 
One chooses $t$ and $t'$ to make sure that $f^{-1}(t)$ and $f^{-1}(t')$ are contained in open sets which retract to $X_i$ resp.\ $X_{i-1}$. 
The maps $b_{i-1}$ and $a_i$ are the composition of such identifications with the retractions to $X_{i-1}$ resp.\ $X_{i}$.  
We leave the reader to do the tedious verification that the homotopy classes of $a_i$ and $b_{i-1}$ are independent of the choices made.

Passing to $r$-homology one obtains the $G_{2m}$-representation $\rho_r=\rho_r(f)$ whose vector spaces are $V_{2s}=H_r(X_s)$ and $V_{2s-1}= H_r(R_s)$ 
and the linear maps $\alpha_i$ and $\beta_i$ are induced by the continuous maps $a_i$ and $b_i$.

The representation $\rho_r(f)$ has  bar codes whose ends are $i,j+km$, $1\leq i,j \leq m$. 
Denote by $\mathcal B_r(f)$, the collections of intervals defined by the bar codes of $\rho_r(f)$ but with the ends $i$ and $j+km$ replaced by $\theta_i$ and $\theta_j+2\pi k$.  
Denote by $\mathcal J_r(f)$ the collection of Jordan blocks of the representation $\rho_r(f)$.  

If $\tilde f\colon\tilde X \to \mathbb R$ is the infinite cyclic covering of $f$ then the real numbers $\theta_i+2\pi k$ are the critical values and $t_i+2\pi k$ are regular values 
(between consecutive critical values) and the tameness of $\tilde f$ gives the maps $a_{i+km}\colon\tilde X_{t_{i+1}+2\pi k}\to\tilde X_{\theta_{i}+2\pi k}$ 
and $b_{i+km}\colon\tilde X_{t_i+2\pi k}\to \tilde X_{\theta_i+2\pi k}$. 
By passing to homology in dimension $r$ one obtains a good $\mathcal Z$-representation $\rho_r(\tilde f)$ which is exactly the infinite cyclic covering $\widetilde {\rho_r(f)}$.

The collections $\mathcal B_r(\tilde f)$, $\mathcal B^c_r(\tilde f)$, $\mathcal B^o_r(\tilde f)$, $\mathcal B^{co}_r(\tilde f)$, $\mathcal B^{oc}_r(\tilde f)$ also denoted by $\tilde {\mathcal B}_r(f)$, $\tilde {\mathcal B}^c_r(f)$, $\tilde {\mathcal B}^o_r(f)$, $\tilde {\mathcal B}^{co}_r( f)$, $\tilde {\mathcal B}^{oc}_r(f)$ are the bar codes of the representation $\widetilde {\rho_r(f)}.$  They are invariant w.r.\ 
to the $2\pi$ translation and the collections $\mathcal B_r(f)$, $\mathcal B^c_r(f)$, $\mathcal B^o_r(f)$, $\mathcal B^{co}_r(f)$, $\mathcal B^{oc}_r(f)$ can be viewed as 
equivalence classes (modulo the $2\pi$ translation) of elements of $\mathcal B^c_r(\tilde f)$, $\mathcal B^o_r(\tilde f)$, $\mathcal B^{co}_r(\tilde f)$, $\mathcal B^{oc}_r(\tilde f)$.   For $X$ compact and $f$ tame the sets $\mathcal B_r(f)$ are finite while $\tilde {\mathcal B}_r(f)$, if nonempty, are infinite.

Given $\xi\in H^1(X;\mathbb Z)$ and $u\in \kappa\setminus 0$, the pair $(\xi,u)$ denotes  the rank one representation $H_1(X;\mathbb Z)\to \mathbb Z\to \kappa\setminus 0$, 
where the first arrow is given by $\xi$ and the second by the homomorphism  $\langle u\rangle\colon\mathbb Z\to\kappa\setminus 0$ defined by $\langle u\rangle(n)=u^n$.
One denotes by $H_r(X;(\xi,u))$ the homology of $X$ with coefficients in the local system defined by the representation $(\xi,u)$, see \cite[Section~3.H]{H02}.
To describe the latter homology group, recall that the singular chain complex of the infinite cyclic covering, $C_*(\tilde X)$, can be regarded as a chain complex of $\kappa[T^{-1},T]$-modules where the action of $T$ is induced by the fundamental deck transformation.
The homology $H_r(X;(\xi,u))$ is canonically isomorphic to the $r$-th homology of the $\kappa$-cochain complex $C_*(\tilde X)\otimes_u\kappa$ obtained by tensorizing with the representation $\kappa[T^{-1},T]\to\kappa$ determined by $T\mapsto u$.
If $u=1$, then we have a canonical isomorphism $C_*(\tilde X)\otimes_u\kappa=C_*(X)$ and thus $H_r(X;(\xi, 1))=H_r(X)$.

Replacing homology by homology with coefficients in the local system $(\xi,u)$ leads also to the replacement of the representations $\rho_f(f)$ with the representations $\rho_r(f)_u.$ as explained below. 
Since the local system becomes trivial over $R_i$ and $X_i$, we have isomorphisms $H_r(R_i;(\xi,u))\cong H_r(R_i)$ and $H_r(X_i;(\xi,u))\cong H_r(X_i)$.
The maps induced by $a_i$ and $b_i$, however, will not all coincide with the maps in the representation $\rho_r(f)$ but with the ones for $\rho_r(f)_u.$
More precisely, every trivialization of the infinite cyclic covering over $[\theta_1,t_1+2\pi]$, induces isomorphisms $\phi_i$ and $\phi_i'$, $1\leq i\leq m$,
such that the diagram
$$
\xymatrix{
H_r(X_i;(\xi,u))&H_r(R_{i+1};(\xi,u))\ar[r]^-{(a_{i+1})_*}\ar[l]_-{(b_i)_*}&H_r(X_{i+1};(\xi,u))
\\
H_r(X_i)\ar[u]^-{\phi_i}_-\cong&H_r(R_{i+1})\ar[u]^-{\phi'_{i+1}}_-\cong\ar[l]_-{\beta_i}\ar[r]^-{\alpha_{i+1}}&H_r(X_{i+1})\ar[u]^-{\phi_{i+1}}_-\cong
}
$$
commutes for all $1\leq i<m$, and the diagram
$$
\xymatrix{
H_r(X_m;(\xi,u))&H_r(R_1;(\xi,u))\ar[r]^-{(a_1)_*}\ar[l]_-{(b_m)_*}&H_r(X_1;(\xi,u))
\\
H_r(X_m)\ar[u]^-{\phi_m}_-\cong&H_r(R_1)\ar[u]^-{\phi_1'}_-\cong\ar[l]_-{\beta_m}\ar[r]^-{u\alpha_1}&H_r(X_1)\ar[u]^-{\phi_1}_-\cong
}
$$
commutes.
The $G_{2m}$-representation obtained by using homology with coefficients in $(\xi,u)$ will thus be isomorphic to $(\rho_r(f))_u$, see Section~\ref{SS:rhou}.

\subsection{The relevant exact sequences, cf.\ \cite{BD11}} 

The tool which permits the calculation of the homology of $X$, $\tilde X$ and various pieces of these spaces is provided by Proposition~\ref{P31} below.
The sequence in \eqref{EE13} has been established in \cite[Section~5]{BD11}.

\begin{proposition}\label{P31}
Let $f\colon X\to \mathbb S^1$ be a tame map and $\tilde f\colon\tilde X\to \mathbb R$ its infinite cyclic covering.  
Let $\rho_r= \rho_r(f)$ and $\tilde \rho_r=\rho_r(\tilde f)= \widetilde\rho_r(f)$ be the representations associated with $f$ and $\tilde f$.
One has the following short exact sequences
\begin{equation}\label{EE16}
0\to\coker M((\rho_r)_u)\to H_r(X; (\xi_f,u))\to\ker M((\rho_{r-1})_u)\to0,
\end{equation}
which for $u=1$ becomes 
\begin{equation}\label{EE13}
0\to\coker M(\rho_r)\to H_r(X)\to\ker M(\rho_{r-1})\to0.
\end{equation}
Moreover, one has a short exact sequence of $\kappa[T^{-1},T]$-modules
\begin{equation}\label{F1}
0\to\coker M(\tilde \rho_r)\to H_r(\tilde X)\to\ker M(\tilde \rho_{r-1})\to0.
\end{equation}
These sequences are all compatible with truncations as explained below, see Diagrams~\eqref{DE1} and \eqref{DE30}.
\end{proposition}

Recall that the $\kappa[T^{-1},T]$-module structure on $H_r(\tilde X)$ is induced by the fundamental deck transformation.
The $\kappa[T^{-1},T]$-module structures on $\ker(\tilde\rho_r)$ and $\coker(\tilde\rho_r)$ have been described at the end of Section~\ref{S2}.


Observe that for $\theta_i\leq \theta_j$ critical angles of $f$, 
if $f_{[\theta_i, \theta_j]}$ denotes the restriction of $f$ to $X_{[\theta_i, \theta_j]}=f^{-1}[\theta_i, \theta_j]$, then 
\begin{equation*}
\rho_r(f_{[\theta_i, \theta_j]})=T_{i,j}(\rho_r(f)). 
\end{equation*}
Similarly, for $c_i\leq c_j$ critical values of $\tilde f$, if $\tilde f_{[c_i,c_j]}$ denotes the restriction of $\tilde f$ to $\tilde X_{[c_i, c_j]}=\tilde f^{-1}[c_i, c_j],$ then 
\begin{equation*}
\rho_r(\tilde f_{[c_i, c_j]})= T_{i,j}(\tilde \rho_r(f)). 
\end{equation*}
Since $f$ and therefore $\tilde f$ is tame one also has: 

\noindent for any $\theta'$ with $\theta_{i-1}< \theta'\leq \theta_i$ and  $\theta''$ with $\theta_j\leq  \theta'' < \theta_{j+1}$  
$$\rho_r(f_{[\theta', \theta'']})= \rho_r(f_{[\theta_i, \theta_j]})$$
and for any $c'$ with $c_{i-1}< c'\leq c_i$ and  $c''$ with $c_j\leq  c'' < c_{j+1}$ 
$$\rho_r(f_{[c', c'']})= \rho_r(f_{[c_i, c_j]}).$$

In the case of the $G_{2m}$-representation $\rho_r(f)$ ``compatibility with truncation'' means that for any pairs of critical angles $(\theta_i,\theta_j)$ and $(\theta_{i'}, \theta_{j'})$, 
$0<\theta_i\leq \theta_{i'}\leq\theta_{j'}\leq\theta_j\leq 2\pi$ the following diagram is commutative:
\begin{equation}\label{DE1}
\xymatrix{
0\ar[r] & \coker M(T_{i',j'}( \rho_{r}))\ar[d]^{v_l} \ar[r] &H_r( X_{[\theta_{i'},\theta_{j'}]}) \ar[r]^-{\pi'}\ar[d]^v &\ker M(T_{i', j'}( \rho_{r-1}))\ar[r]\ar[d]^{v_r}& 0\\
0\ar[r] & \coker M(T_{i,j}( \rho_{r}))\ar[d]^{v'_l}\ar[r] &H_r( X_{[\theta_i,\theta_j]}) \ar[r]^-{\pi''} \ar[d]^{v'}&\ker M(T_{i,j}( \rho_{r-1}))\ar[r]\ar[d]^{v'_r}& 0\\
0\ar[r] & \coker M( (\rho_{r})_u)\ar[r] &H_r(X; (\xi_f,u)) \ar[r]^-\pi &\ker M(( \rho_{r-1})_u)\ar[r]& 0}
\end{equation}
In the case of the $\mathcal Z$-representation $\tilde\rho_r$, this means that for any pairs of critical values $(c_i, c_j)$ and $(c_{i'}, c_{j'})$ with $c_i\leq c_{i'}\leq c_{j'}\leq c_j$ the following diagram is commutative:
\begin{equation}\label{DE30}
\xymatrix{
0\ar[r] & \coker M(T_{{i'}, {j'}}(\tilde \rho_{r}))\ar[d]^{v_l} \ar[r] &H_r( \tilde X_{[c_{i'}, c_{j'}]}) \ar[r]^-{\pi'}\ar[d]^v &\ker M(T_{i', j'}( \tilde \rho_{r-1}))\ar[r]\ar[d]^{v_r}& 0\\
0\ar[r] & \coker M(T_{i, j}( \tilde \rho_{r}))\ar[d]^{v'_l}\ar[r] &H_r( \tilde X_{[c_i, c_j]}) \ar[r]^-{\pi''} \ar[d]^{v'}&\ker M(T_{i, j}( \tilde\rho_{r-1}))\ar[r]\ar[d]^{v'_r}& 0\\
0\ar[r] & \coker M (\tilde \rho_{r})\ar[r] &H_r(\tilde X) \ar[r]^-\pi &\ker M(\tilde \rho_{r-1})\ar[r]& 0.
}
\end{equation}
Note that diagrams (\ref {DE1}) and (\ref{DE30}) implies that any splitting (= right inverse) of $\pi'$ or $\pi''$ extend to a splitting of $\pi''$ or $\pi$ respectively.

\begin{proof}   
The reader should recognize in the matrices $M(T_{i, j}( \tilde \rho_{r})),$ $M( (\rho_{r})_u)$ and $M (\tilde \rho_{r})$ the linear maps induced from 
$\bigoplus_{i\leq k\leq j} H_r(R_k)$ to  $H_r( X_{[\theta_i,\theta_j]}),$   from $\bigoplus_{1\leq i\leq m} H_r(R_i)$ to $H_r(X; (\xi_f,u))$ and from $\bigoplus_{1\leq i\leq m} H_r(R_i)$ to $H_r(\tilde X)$ respectively.

Denote by $\mathcal R:=\bigsqcup_{1\leq i\leq m}R_i$, $\tilde{\mathcal R}:=\bigsqcup _{i \in \mathbb Z} R_i$, $\mathcal X:=\bigsqcup _{1\leq i\leq m} X_i$
and $\tilde{\mathcal X}:=\bigsqcup _{ i\in \mathbb Z} X_i$.
The short exact sequence \eqref{EE16} follows from the long exact sequence
\begin{equation}\label{MV1}
\to H_r(\mathcal R)\xrightarrow{M((\rho_r)_u)}H_r(\mathcal X)\to H_r(X;(\xi,u))\to H_{r-1}(\mathcal R)\xrightarrow{M((\rho_{r-1})_u)}H_{r-1}(\mathcal X)\to
\end{equation}
with $H_r(\mathcal R)= \bigoplus_{1\leq i\leq m} H_r(R_i)$ and $H_r(\mathcal X)= \bigoplus_{1\leq i\leq m} H_r(X_i)$,
and the short exact sequence \eqref{F1} follows from the long exact sequence 
\begin{equation}\label{MV2}
\cdots\to H_r(\tilde {\mathcal R})\xrightarrow{M(\rho_r)}H_r(\tilde {\mathcal X})\to H_r(\tilde X)\to H_{r-1}(\tilde {\mathcal R})\xrightarrow{M(\rho_{r-1})}H_{r-1}(\tilde{\mathcal X})\to\cdots.
\end{equation}
which remain to be established.
The sequence in \eqref{EE13} appears as a special case of the sequence in \eqref{EE16} for $u=1$.

Since both long exact sequences \eqref{MV1} and \eqref{MV2} are derived in the same way we will treat only \eqref{MV1} and for simplicity only the case $u=1$. 

First choose an $\epsilon>0$ small enough so that $2\epsilon<t_1$ and $\theta_{i-1}+2\epsilon<t_i<\theta_i-2\epsilon$.
To simplify the writing, since $i\leq m$, introduce $\theta_{m+1}=\theta_1+2\pi$, let 
$$
f^{-1}\bigl([\theta_m\pm\epsilon,\theta_{m+1}\pm \epsilon)\bigr):=\tilde f^{-1}\bigl([\theta_m\pm\epsilon,\theta_1+2\pi\pm\epsilon]\bigr),
$$ 
and define 
$$
\mathcal P':=\bigsqcup_{1\leq i\leq m}f^{-1}\bigl([\theta_i, \theta_{i+1}-\epsilon)\bigr),
\qquad
\mathcal P'':=\bigsqcup_{1\leq i\leq m}f^{-1}\bigl((\theta_i+\epsilon, \theta_{i+1}]\bigr).
$$
Observe that in view of the choice of $\epsilon$ and of the tameness of $f$ the inclusions 
$\mathcal X \subset \mathcal P'$,
$\mathcal X \subset \mathcal P''$, and
$\mathcal X \sqcup \mathcal R\subset \mathcal P'\cap \mathcal P''$
are homotopy equivalences.  
The Mayer--Vietoris long exact sequence for $X=\mathcal P'\cup \mathcal P''$ gives the commutative diagram 
\begin{equation}\label{EE23}
\xymatrix{
&                                                       & H_r(\mathcal R)\ar[r]^{M(\rho_r(f))}                                                  & H_r(\mathcal X)\ar@/^/[rd]                                               &&\\
\ar[r]&H_{r+1}(X)\ar@/^/[ur]
\ar[r]^-{\partial_{r+1}}&H_r(\mathcal R)\oplus H_r(\mathcal X)\ar[u]^{\operatorname{pr}_1}\ar[r]^N  &H_r(\mathcal X)\oplus H_r(\mathcal X)\ar[u]^{(\operatorname{Id},-\operatorname{Id})}{\ar[r]^-{(i_r, -i_r)}} &H_r(X)\ar[r] &\\
&                                                       &H_r(\mathcal X)\ar[u]^{\operatorname{in}_2}\ar[r]^{\rm Id}                                           &H_r(\mathcal X)\ar[u]^{\Delta}                                                 &&
}
\end{equation}
where $\Delta$ denotes the diagonal, $\operatorname{in}_2$ the inclusion on the second component, $\operatorname{pr}_1$ the projection on the first component, 
$i_r$ the linear map induced in homology by the inclusion $\mathcal X\subset X$. 
Recall that the matrix $M(\rho_r(f))$ is defined by 
\begin{equation*}
\begin{pmatrix}
\alpha^r_1 & -\beta^r_1 & 0          & \cdots         & 0\\
0          & \alpha^r_2 & -\beta^r_2 & \ddots         & \vdots\\
\vdots     & \ddots     & \ddots     & \ddots         & 0 \vspace{0.5ex}\\
0          & \cdots     & 0          & \alpha^r_{m-1} & -\beta^r_{m-1}\vspace{1ex}\\
-\beta^r_m & 0          & \cdots     & 0              & \alpha^r_m
\end{pmatrix} 
\end{equation*}
with $\alpha^r_i\colon H_r(R_i)\to H_r(X_i)$ and $\beta^r_i\colon H_r(R_{i+1})\to H_r(X_i)$ induced by the maps $a_i$ and $b_i$.  
The block matrix $N$ is defined by 
\begin{equation*}
N:=
\begin{pmatrix}
\alpha^r & \Id\\
-\beta^r & \Id
\end{pmatrix}
\end{equation*}
where $\alpha^r$ and $\beta^r$ are the matrices 
$$
\begin{pmatrix}
\alpha^r_1 & 0          & \cdots &    0 \\
0          & \alpha^r_2 & \ddots &\vdots \\
\vdots     & \ddots     & \ddots &0 \\
0          & \cdots     & 0      & \alpha^r_{m-1} 
\end{pmatrix}
\quad\text{and}\quad
\begin{pmatrix}
0         & \beta^r_1 & 0         & \dots  & 0 \\
0         & 0         & \beta^r_2 & \ddots & \vdots \\
\vdots    & \vdots    & \ddots    & \ddots & 0 \\
0         & 0         & \dots     & 0      & \beta^r_{m-1}\\
\beta^r_m & 0         & \dots     & 0      & 0
\end{pmatrix}.
$$
The long exact sequence \eqref{MV1} is the top sequence in the diagram \eqref{EE23}.

The long exact sequence \eqref{MV2} can be established analogously.
The naturality of the Mayer--Vietoris sequence w.r. to maps which preserve the {\it decomposition of a space in two pieces}  implies that the homomorphisms in the sequence \eqref{MV2} intertwine the automorphisms induced by the fundamental deck transformation on $H_r(\tilde{\mathcal R})$, $H_r(\tilde{\mathcal X})$, and $H_r(\tilde X)$, respectively.
Hence, \eqref{F1} is a short exact sequence of $\kappa[T^{-1},T]$-modules.
Compatibility with truncations follows from the naturality of the Mayer--Vietoris sequence too.
\qed
\end{proof}

\section{Proof of Theorem~\ref{T1} and some refinements.}\label{S4}

Let $f\colon X\to\mathbb S^1$ be a tame map on a compact ANR, and let $\xi=\xi_f\in H^1(X;\mathbb Z)$ denote the corresponding integral cohomology class.
Moreover, let $\pi\colon\tilde X\to X$ denote the associated infinite cyclic covering, that is, the pull back by $f$ of the universal covering $p\colon\mathbb R\to\mathbb S^1$.
There exists a tame map $\tilde f\colon\tilde X\to\mathbb R$ which is equivariant with respect to the (principal) $\mathbb Z$-actions on $\tilde X$ and $\mathbb R$ such that the
following diagram commutes:
$$
\xymatrix{
\tilde X\ar[d]_\pi\ar[r]^-{\tilde f}&\R\ar[d]^p \\X\ar[r]^-f&S^1.
}
$$

Recall that the vector space $H_r(\tilde X)$ is a $\kappa[T^{-1},T]$-module\footnote{$\kappa[T^{-1},T]$ denotes the ring of Laurent polynomials with coefficients in $\kappa$, i.e., the group algebra $\kappa[\mathbb Z]$.} 
where the multiplication by $T$ is the linear isomorphism induced by the fundamental deck transformation $\tau\colon\tilde X\to \tilde X$ corresponding to the action of $1\in\mathbb Z$.

Let $\kappa[T^{-1},T]]$ be the field of Laurent power series and define 
$$
H^N_r(X;\xi):=H_r(\tilde X)\otimes_{\kappa[T^{-1}, T]} \kappa[T^{-1}, T]].
$$ 
The $\kappa[T^{-1},T]]$-vector spaces $H^N_r(X;\xi)$ is called the $r$-th Novikov homology\footnote{Instead of $\kappa[T^{-1},T]]$ one can consider the field 
$\kappa[[T^{-1},T]$ of Laurent power series in $T^{-1}$, which is isomorphic to $\kappa[T^{-1},T]]$ by an isomorphism induced by $T\to T^{-1}$.  
The (Novikov) homology defined using this field has the same Novikov--Betti numbers as the one defined using $\kappa[T^{-1}, T]]$.}
and its dimension over the field $\kappa[T^{-1},T]]$, the \emph{Novikov--Betti number} $\beta^N_r(X;\xi)$.

Consider $H_r(\tilde X)\to H^N_r(X;\xi)$ the $\kappa[T^{-1},T]$-linear map induced by taking the tensor product 
with $\kappa[T^{-1},T]]$ over $\kappa[T^{-1},T]$. 
The $\kappa[T^{-1},T]$-module $V_r(\xi)$, 
$$
V_r(\xi):=\ker\bigl(H_r(\tilde X)\to H^N_r(X;\xi)\bigr),
$$
when regarded as a $\kappa$-vector space equipped with the linear isomorphism $T_r(\xi)$ provided by the multiplication by $T$, is referred to as the \emph{$r$-monodromy} of $(X,\xi)$. 
As a $\kappa[T^{-1},T]$-module $V_r(\xi)$ is exactly the torsion of the $\kappa[T^{-1},T]$-module $H_r(\tilde X)$.

A base for $V_r(\xi)$ provides a parametrization of the abstract set $\tilde {\mathcal J}_r(f)$ and therefore an identification of $V_r(\xi)$ to $\kappa [\tilde {\mathcal J}_r(f)]$.
We continue to call ``monodromy'' and denote by $T_r(\xi)$ the isomorphism $T_r(\xi)\colon\kappa[\tilde {\mathcal J}_r(f)]\to \kappa[ \tilde {\mathcal J}_r(f)]$ obtained by using  the above identification.

\begin{proof}[Proof of Theorem \ref{T1}]
Part~\itemref{T1:c} follows from Propositions~\ref{P31} and \ref{O45} which relate the Betti numbers to the bar codes and Jordan cells via the short exact sequence \eqref{EE13}.
Similarly one can compute $\dim H_r(X,(\xi_f,u))$ in terms of barcodes and Jordan cells using the short exact sequence~\eqref{EE16}.

To check parts \itemref {T1:a} and \itemref{T1:b} recall that according to Lemma~\ref{OOO} we have isomorphisms of $\kappa[T^{-1},T]$-modules:
\begin{align*}
\coker M(\tilde\rho_r(f))&\cong\kappa[T^{-1},T][\mathcal B^c_r(f)]\oplus\kappa[\tilde{\mathcal J}_r(f)],
\\
\ker M(\tilde\rho_{r-1}(f))&\cong\kappa[T^{-1},T][\mathcal B^o_r(f)].
\end{align*}
In particular, the $\kappa[T^{-1},T]$-module $\ker M(\tilde\rho_{r-1}(f))$ is free.
Hence, the short exact sequence~\eqref{F1} splits and we obtain an isomorphism of $\kappa[T^{-1},T]$-modules,
$$
H_r(\tilde X)\cong\kappa[T^{-1},T][\mathcal B^c_r(f)\sqcup\mathcal B^o_{r-1}(f)]\oplus\kappa[\tilde{\mathcal J}_r(f)].
$$
Tensorizing with $\kappa[T^{-1},T]]$, we obtain an isomorphisms of $\kappa[T^{-1},T]]$-vector spaces,
$$
H^N_r(X;\xi_f)\cong\kappa[T^{-1},T]][\mathcal B^c_r(f)\sqcup\mathcal B^o_{r-1}(f)].
$$
We conclude $\beta^N_r(X;\xi_f)=\sharp\mathcal B^c_r(f)+\sharp\mathcal B^o_{r-1}(f)$, whence part (a). 
Moreover,
$$
V_r(\xi_f)\cong\kappa[\tilde{\mathcal J}_r(f)]\cong\bigoplus_{J\in\mathcal J_r(f)}J,
$$
whence part (b).
\qed
\end{proof}

{\bf More calculations}

\vskip.1in
A nonempty subset $K$ of $\mathbb S^1$ or $\mathbb R$, will be called a \emph{closed multi-interval} if it is a finite union of disjoint closed intervals 
$[\theta_1,\theta_2]$ with $0\leq \theta_1 \leq \theta_2 <2\pi$ in the case of $\mathbb S^1$, and $[a,b]$ with $ a\leq b $ 
or $(-\infty, a]$ or $[b,\infty)$ 
in the case of $\mathbb R$. 
One denotes by $X_{K}:=f^{-1}(K)$ if $K\subset \mathbb S^1$ and by $\tilde X_{K}=f^{-1}(K)$ if $K\subset \mathbb R$.  

In case  $K\subset \mathbb S^1$ one considers
\begin{align*}
\mathcal B^c_{r,K}(f)&=\{I\in \mathcal B^c_r(f)  \mid I\cap K\ne \emptyset\},
\\
\mathcal B^o_{r,K}(f)&=\{I\in \mathcal B^o_{r}(f)  \mid I\subset K\},
\end{align*}
and for $u\in \kappa\setminus 0$ the sets 
\begin{align*}
S_{r, u} (f)&= \mathcal B^c_{r}(f)\sqcup  \mathcal B^o_{r-1}(f)\sqcup \mathcal J_{r,u}(f)\sqcup \mathcal J_{r-1,u}(f),\\
S_{r, K, u} (f)&= \mathcal B^c_{r,K}(f)\sqcup  \mathcal B^o_{r-1,K}(f)\sqcup \mathcal J_{r,u}(f).
\end{align*}
Recall that $\mathcal J_{r,u}(f)$ denotes the set of Jordan blocks $J= (V,T)\in \mathcal J_r(f)$ whose linear isomorphism $T$ has $u$ as eigenvalue. Since $u\in \kappa\setminus 0,$ $J$ is actually a Jordan cell.

In case $K\subset\mathbb R$ an  one considers the sets
\begin{align*}
\widetilde{\mathcal B}_{r}(f)&=\{I\in {\mathcal B}_r(\tilde f) \},
\\
\widetilde{\mathcal B}^c_{r,K}(f)&=\{I\in {\mathcal B}^c_r(\tilde f)  \mid I\cap K\ne \emptyset\},
\\
\widetilde{\mathcal B}^o_{r,K}(f)&=\{I\in {\mathcal B}^o_{r}(\tilde f)  \mid I\subset K\},
\end{align*}
and 
\begin{equation}\label {E23}
\begin{aligned} 
\widetilde S_{r, K} (f)&= \widetilde {\mathcal B}^c_{r,K}(f)\sqcup  \widetilde{\mathcal B}^o_{r-1,K}(f)\sqcup \widetilde{\mathcal J}_r(f),
\\
\widetilde S_{r} (f)&= \widetilde {\mathcal B}^c_{r}(f)\sqcup  \widetilde{\mathcal B}^o_{r-1}(f)\sqcup \widetilde{\mathcal J}_{r}(f). 
\end{aligned}
\end{equation}

These sets have the following properties: 
\begin{enumerate}[(i)]
\item 
If $K_1, K_2, K$ are closed multi-intervals in $\mathbb S^1$ or $\mathbb R$ with $K_1\cap K_2=\emptyset$ and $K= K_1\cup K_2$ then 
$S_{r, K,u}= S_{r, K_1,u} \cup S_{r, K_2,u}$ and $\tilde S_{r, K}= \tilde S_{r, K_1} \cup \tilde S_{r, K_2}$.
\item 
If $K_1, K_2, K$ are closed multi-intervals in $\mathbb S^1$ or $\mathbb R$ with $K_1\cap K_2=K$  then 
$S_{r, K,u}= S_{r, K_1,u} \cap S_{r, K_2,u}$ and $\tilde S_{r, K}= \tilde S_{r, K_1} \cap \tilde S_{r, K_2}$.
\item 
If $K_1, K_2 $ closed multi-intervals with $K_1\subset K_2$ then 
$S_{r, K_1,u} \subseteq  S_{r, K_2,u}$ and $ \tilde S_{r, K_1} \subseteq \tilde S_{r, K_2}$.
\end{enumerate}

For $K$ a multi-interval in $\mathbb S^1$ or $\mathbb R$ denote by: 
\begin{align*}
\mathbb I_r(f; K, u)&:= \img\bigl(H_r(X_{K}) \to H_r(X; (\xi, u)\bigr),
\\
\mathbb I_r(\tilde f; K)&:= \img\bigl(H_r(\tilde X_{K}) \to H_r(\tilde X)\bigr).  
\end{align*}

Let $f\colon X\to \mathbb S^1$ be a tame map with $m$ critical values $0<\theta_1 <\theta_2 \cdots <\theta_m\leq 2\pi$  and $\tilde f:\tilde X\to \mathbb R$ an infinite cyclic cover of $f,$ with critical values $c_i,  i \in \mathbb Z$ with $p(c_{i+mk})= \theta_i.$ 

Recall that for a surjective linear map $\pi: A\to B$ a linear map $s\colon B\to A$ such that $\pi\cdot s=id$ is called a splitting. 

For the projections 
\begin{equation}
\begin{aligned}
H_r(\tilde X)\to &\ker M(\tilde \rho_{r-1}(f))\\
H_r(\tilde X_{[c_i, c_j]})\to &\ker M(T_{i,j}(\tilde \rho_{r-1}(f)))
\end{aligned}
\end{equation}
one  considers collections of splittings  
\begin{equation}
\begin{aligned}
\tilde s_r\colon &\ker M(\tilde \rho_{r-1}(f)) \to H_r(\tilde X)\\
\tilde s_{r; i, j}\colon &\ker M(T_{i,j}(\tilde \rho_{r-1}(f)))\to H_r(\tilde X_{[c_i, c_j]})
\end{aligned}
\end{equation}
%
%
such that the diagram
\begin{equation}
\xymatrix{ker M(T_{i',j'}(\tilde \rho_{r-1}(f)))\ar[r]^-{\tilde s_{r-1;i',j'}}\ar[d] & H_r(\tilde X_{[c_{i'}, c_{j'}]})\ar[d]\\
\ker M(T_{i,j}(\tilde \rho_{r-1}(f)))\ar[r]^-{\tilde s_{r-1;i,j}}\ar[d] & H_r(\tilde X_{[c_i, c_j]})\ar[d]\\
\ker M(\tilde \rho_{r-1}(f))\ar[r]^-{\tilde s_{r-1}} & H_r(\tilde X)}
\end{equation}
commutes, for all $-\infty<i' \leq i \leq j \leq j'<\infty$.

A collection of splittings as above is called a \emph{collection of compatible splittings}.
In view of the fact that the splitting $\tilde s_{r-1; i,j}$ and $\tilde s_{r-1; j,k}$  can be extended to a splitting $\tilde s_{r-1; i,k}$, the existence of collections of compatible splittings is straightforward.  
The construction being realized inductively from $(i,j)$ to $(i,j+1)$ and from $(i,j)$ to $(i-1,j)$.
Moreover one can produce collections of compatible splittings which are $m$-periodic, which means that $\tilde s_r$ intertwines the isomorphism $t_r$ with $\tau _r$.
In other words, $\tilde s_r$ may be assumed to be a homomorphism of $\kappa[T^{-1},T]$-modules

Precisely, we consider the surjective maps $\pi_{r,i,j}: H_r(\tilde X_{[c_i, c_j]}) \to \kappa [\mathcal B^o(T_{i,j}(\tilde \rho_{r-1})]$, 
the composition of 
$H_r(X_{[c_i, c_j]}) \to \ker (M(T_{i,j} (\tilde \rho_{r-1}))$ with  the isomorphism $(\Psi^o_{i,j})^{-1} : \kappa[\mathcal B^o(T_{i,j}(\tilde \rho_{r-1})]\to \ker (M(T_{i,j} (\tilde \rho_{r-1})).$ Here $\tilde \rho_{r-1}$ abbreviates $\tilde \rho_{r-1}(f).$
For any bar code $I\in B^o(\tilde \rho_{r-1})$ with ends $c_i$ and $c_j$ call {\it lift of $I$} an element $v_I\in H_r(\tilde X_{[c_i, c_j]})$ s.t. $\pi_{r, i,j}(v_I)= I.$ 

It is straightforward to provide a family of lifts $v_I$ for any $I\in \mathcal B^o(\rho_{r-1})$  s.t. 
in view of surjectivity of $\pi_{r,i,j}$ all $v_I$ with $I$ of the same ends  linearly independent and in view of the isomorphism $t^\ast: H_r(\tilde X_{[c_i, c_j]})\to H_r(\tilde X_{[c_{i+m}, c_{j+m}]}),$ $c_{i+m}= c_i +2\pi,$ with $v_I$ satisfying  $v_{I+2\pi}= t^\ast(v_I).$
Define  
$\overline s_{r-1,i,j}: \kappa[ \mathcal B^o(T_{i,j}(\tilde \rho_{r-1})\to H_r(\tilde X_{[c_i, c_j]}$ and $\overline s_{r-1}: \kappa[ \mathcal B^o(\tilde \rho_{r-1})\to H_r(\tilde X)$ by taking 
$\overline s_{r-1, i,j} (I)$ to be  the image of $v_I$ in $H_r(\tilde X_{[c_i, c_j]}$
and $\tilde s_{r-1} (\Psi^o(I)$ to be  the image of $v_I$ in $H_r(\tilde X).$
Define $\tilde s_{r-1,i,j}= \overline s_{r-1,i,j}\cdot \Psi^o_{i,j}$ and $\tilde s_{r-1}= \overline s_{r-1}\cdot \Psi^o.$

A similar definition can be obtained by replacing $\tilde\rho_{r-1}(f)$ and $H_r(\tilde X)$ by $\rho_{r-1}(f)_u$ and $H_r(X;(\xi_f,u))$.

With the notations and the definitions above we have the following technical results which calculate $\mathbb I_r(f; K, u)$ and $\mathbb I_r(\tilde f; K)$  as well as homologies of $H_r(X,(\xi_f,u)),$ $H_r(\tilde X)$  already described.

\begin{proposition}\label{P35}\
Let $f\colon X\to \mathbb S^1$ be a tame map and suppose that for each $r$ a decomposition of the representation $\rho_r(f)$ as a sum of bar code representations and Jordan block and a collection of compatible splittings is provided. 
Then:

a) For $u\in\kappa\setminus 0$ the decompositions and  the  collections of compatible splittings  provide the 
isomorphisms 
\begin{equation*}
\omega_{r,u}\colon\kappa[S_{r,u}(f)] \to H_r(X; (\xi_f ,u))
\end{equation*}
and for any  closed multi interval $K\subset \mathbb S^1$  the 
isomorphisms 
\begin{equation*}
\omega_{r, K, u}\colon\kappa[S_{r,K, u} (f)]\to \mathbb I_r(f;K,u)
\end{equation*}
such that for any pair $K', K$ of closed multi-intervals  in $\mathbb S^1$ with $K'\subset K$,  
the diagram (\ref {E20}) is commutative.
\begin{equation}\label{E20}
\vcenter{
\xymatrix{
\mathbb I_r(f;K',u) \ar[r]^-{\subseteq} &\mathbb I_r(f;K,u)  \ar[r]^-{\subseteq} &H_r( X; (\xi_f, u))\\
\kappa[S_{r, K', u}(f)]\ar[u]_\cong^{\omega_{r, K', u}}\ar[r] & \kappa[S_{r, K, u}(f)]\ar[u]_\cong^{\omega_{r, K, u}} \ar[r]& \kappa[S_{r, u}(f)]\ar[u]_\cong^{\omega_{r, u}}.
}}
\end{equation}
The horizontal arrows in the bottom line are induced by the inclusions of the sets in brackets.

b) The decompositions and the collection of compatible splittings  provide the isomorphisms
\begin{equation*}
\widetilde\omega_{r}\colon\kappa[\widetilde S_{r}(f)] \to H_r(\tilde X)
\end{equation*}
and for any closed multi interval $K\subset \mathbb R$  the 
isomorphisms 
\begin{equation*}
\widetilde{\omega}_{r, K}\colon\kappa[\widetilde S_{r,K} (f)]\to \mathbb I_r(\tilde f;K)
\end{equation*}
such that for any pair $K', K$ of closed multi-intervals in $\mathbb R$ with $K'\subset K$, the diagram  (\ref {E20b}) is commutative.
\begin{equation}\label{E20b}
\vcenter{
\xymatrix{
\mathbb I_r(\tilde f;K') \ar[r]^{\subseteq} &\mathbb I_r(\tilde f;K)  \ar[r]^{\subseteq} &H_r( \tilde X)\\
\kappa[\widetilde S_{r, K'}(f)]\ar[u]_\cong^{\widetilde \omega_{r,K'}}\ar[r] & \kappa[\widetilde S_{r, K}(f)]\ar[u]_\cong^{\widetilde \omega_{r;K}} \ar[r]& \kappa[\widetilde S_{r}(f)]\ar[u]_\cong^{\widetilde \omega_{r}}.
}}
\end{equation}
The horizontal arrows in the bottom line are induced by the inclusions of the sets in brackets. 
The isomorphism $\tilde\omega_r$ is an isomorphism of $\kappa[T^{-1}, T]$-modules.

c) The  decompositions and the splittings provide the isomorphisms 
$$
\omega^N_r\colon\kappa[T^{-1}, T]] \ [\mathcal B^c_r(f)\sqcup \mathcal B^o_{r-1}(f)]\to H_r^N(X;\xi_f).
$$ 
\end{proposition}

{\bf Even more calculations}
\vskip .1in
It is also possible to calculate $H_r(X_K)$ for $K\subset \mathbb S^1$ and $H_r(\tilde X_K)$ for $K\subset \mathbb R$. 
In this case, in addition to closed and open bar codes and  to Jordan blocks, the mixed bar codes will appear.

For the purpose of definition below we treat a closed interval $K'=[\theta', \theta''] \subset \mathbb S^1,$ $0<\theta' \leq  \theta'' <2\pi$ as the closed  interval $K= [\theta', \theta''] \subset \mathbb R$

To formulate the result for $K$ a closed interval of  $\mathbb R$ we add to the previous definitions, see formulae \eqref{E23}, the sets:
$$
\widetilde{\mathcal B}^{co}_{r,K} (f)= \{I\in \widetilde{\mathcal B}^{co}_r(f) \mid\textrm{$I\cap K\ne\emptyset$ and closed}\},
$$ 
$$
\widetilde{\mathcal B}^{oc}_{r,K} (f)= \{I\in \widetilde{\mathcal B}^{oc}_r(f) \mid\textrm{$I\cap K\ne\emptyset$ and closed}\}
$$ 
$$
\widetilde{\mathcal B}^{oo}_{r,K} (f)= \{ I\in \widetilde{\mathcal B}^{o}_r(f) \mid  I \supset K\}
$$ 
and denote by $\widetilde S'_{r, K} (f)$ the set
\begin{equation}\label {E24}
\widetilde S'_{r, K} (f)=  \widetilde {\mathcal B}^{co}_{r,K}(f)\sqcup  \widetilde {\mathcal B}^{oc}_{r,K}(f) \sqcup \widetilde{\mathcal B}^{oo}_{r,K} (f)\sqcup \widetilde S_{r,K}.
\end{equation}
We have  $\widetilde{S}_{r,K} (f)\subseteq \widetilde{S}' _{r,K} (f).$ 

\begin{proposition}\label{P36} 

a) The decompositions and  the collections of compatible splittings  provide
for any pair of angles $\theta'$, $\theta''$,  $0< \theta'\leq \theta'' <2\pi$, the isomorphisms  
$$
\omega'_{r,[\theta', \theta'']}\colon\kappa[S'_{r, [\theta', \theta'']}(f)] \to H_r(X_{[\theta', \theta'']})
$$ 
such that  for $0<\theta_1\leq \theta_2 \leq \theta_3 \leq \theta_4 <2\pi$
the diagram \eqref{E20c} below is commutative:
\begin{equation}\label{E20c}
\vcenter{
\xymatrix{
H_r(X_{[\theta_2, \theta_3]}) \ar[r]^{v_r} &H_r(X_{[\theta_1, \theta_4]})  \ar[r]^{v'_r} &H_r( X; (\xi_f,u))
\\
\kappa[S'_{r,[\theta_2,\theta_3]}(f)]\ar[u]^{\omega'_{r,[\theta_2,\theta_3]}}\ar[r] & \kappa[S'_{r,[\theta_1,\theta_4]}(f)]\ar[u]^{\omega'_{r,[\theta_1,\theta_4]}} \ar[r]& \kappa[S_{r, u}(f)]\ar[u]^{\omega_{r,u}}.
}}
\end{equation}

b) The decompositions and  the compatible splittings provide for any $a\leq b$, $a$, $b$ real numbers or $\pm\infty$
the isomorphisms 
$$
\tilde \omega'_{r, [a, b]}\colon\kappa[\tilde S'_{r, [a,b]}(f)] \to H_r(\tilde X_{[a,b]})
$$  
such that for $a\leq b\leq c\leq d$ the diagram~\eqref{E20d} below is commutative. 
\begin{equation}\label{E20d}
\vcenter{
\xymatrix{
H_r(\tilde X_{[b,c]})\ar[r] ^{v_r}         &H_r(\tilde X_{[a,d]}) \ar[r]^{v'_r}               &H_r( \tilde X) \\
\kappa[\tilde S'_{r,[b,c]}(f)]\ar[u]^{\tilde\omega_{r,[b,c]}} \ar[r] & \kappa[\tilde S'_{r,[a,d]}(f)]\ar[u]^{\tilde\omega_{r,[a,d]}} \ar[r]& \kappa[\tilde S _r(f)] \ar[u]^{\tilde \omega_{r}}.
}}
\end{equation}
In both diagrams the horizontal arrows in the top line are linear maps induced by the obvious inclusions, while in the bottom line are the canonical linear maps provided by the sets in brackets subsets of a larger set of all bar codes and all $\tilde {\mathcal J}(f)'$s, cf.\ Definition~\ref{D26}. 
A  bar code in the set $S'_{r, \cdots}$ or in ${\tilde S'}_{r,\cdots}$ is sent to itself if  belongs to the next set  and if not, to zero in the next vector space. 
\end{proposition}

Proposition \ref{P36} permits to express the vector spaces $H_r(\tilde X_{[a,b]}), H_r(\tilde X_{[c,d]}\setminus \tilde X_{(a,b)})$ and the linear maps 
$H_r(\tilde X_{[a,b]}) \to H_r(\tilde X_{[c,d]})$ and $H_r(\tilde X_{[c,d]}\setminus \tilde X_{(a,b)})\to H_r(\tilde X_{[c,d]})$ in terms of the bar codes 
$\tilde {\mathcal B}^{-}_{-}(f)$ and  $\tilde {\mathcal J}_{-}(f).$  This will be used in section (\ref{SS6}).

\begin{proof}[Proof of Propositions~\ref{P35} and \ref{P36}]
In view of the properties of the sets $S_{K,-}$ and $\tilde S_{K,-}$, it suffices to prove the statements for $K$ consisting of one single interval and
in view the  tameness of $f$ one can suppose that $\theta_1$, $\theta_2$ are critical angles and $a$, $b$ critical values.
We treat first the part (a) in both Propositions  (\ref{P35}) and  (\ref{P36}).

The compatible splittings lead to the  commutative diagram  (\ref {E31})
with horizontal arrows isomorphisms.
\begin{equation}\label{E31}
\vcenter{
\xymatrix{
{\coker M(T_{\theta_2,\theta_3}( \rho_{r}))\oplus \ker M(T_{\theta_2,\theta_3}( \rho_{r-1}))}                   \ar[d]^{v_l\oplus v_r}   \ar[r]     &H_r( X_{[\theta_2,\theta_3]})\ar[d]^v \\
{\coker M(T_{\theta_1,\theta_4}( \rho_{r})) \oplus  \ker M(T_{\theta_1,\theta_4}( \rho_{r-1}))}                   \ar[d]^{v'_l\oplus v'_r} \ar[r]     &H_r( X_{[\theta_1,\theta_4]})\ar[d]^{v'} \\
\quad  {\coker M( (\rho_{r})_u) \oplus \ker M(( \rho_{r-1})_u)}                                                                                            \ar[r]    &H_r(X; (\xi_f,u))
}}
\end{equation}
Proposition~\ref{O45} combined with Lemma~\ref{O41} gives the commutative diagram (\ref {E32}) with horizontal lines isomorphisms.
\begin{equation}\label{E32}
\vcenter{
\xymatrix{
\kappa[\tilde S'_{r,[\theta_1, \theta_4]}(f)] \ar[d] \ar[r]    &  \coker M(T_{\theta_2,\theta_3}( \rho_{r}))\oplus \ker M(T_{\theta_2,\theta_3}( \rho_{r-1}))\ar[d]^{v_l\oplus v_r}\\
\kappa[\tilde S'_{r,[\theta_2, \theta_3]}(f)] \ar[d]\ar[r]      & \coker M(T_{\theta_1,\theta_4}( \rho_{r})) \oplus  \ker M(T_{\theta_1,\theta_4}( \rho_{r-1})) \ar[d]^{v'_l\oplus v'_r}\\                     
\kappa[S_{r,u}] \ar[r]                                              & \coker M( (\rho_r)_u) \oplus \ker M(( \rho_{r-1})_u)
}}
\end{equation}

The isomorphism $\omega_{r,u}$ (in Proposition~\ref{P35}) 
is the composition of horizontal arrows in the last line of diagrams \eqref{E31} and \eqref{E32} 
and the isomorphisms $\omega_{r,[\theta_2, \theta_3],u}$ and $\omega_{r, [\theta_1, \theta_4], u}$ are restrictions of $\omega_{r,u}$. 
Similarly he isomorphism $\omega'_{r,[\theta_2,\theta_3]}$ and $\omega'_{r,[\theta_1,\theta_4]}$ (in Proposition~\ref{P36}) are the compositions of the horizontal arrows in the first 
and second lines of the same diagrams.  
The commutativity of the diagrams (\ref {E20}) and (\ref {E20c})
is the consequence of the commutativity of the diagrams \eqref{E31} and \eqref{E32}. 
This establishes part (a) in both Propositions~\ref{P35} and \ref{P36}.

Parts (b) are verified essentially in the same way. 
More precisely, the decompositions of the representations $\rho_r$ imply decompositions of $\tilde \rho_r$ and $T_{k,l}(\tilde \rho_r)$.
Observe that the commutative diagrams \eqref{E31} and \eqref{E32} 
remain valid when one replaces $X$ by $\tilde X$, the representation $\rho_r$ by $\tilde \rho_r$, and $\theta_1$, $\theta_2$, $\theta_3$, $\theta_4$ by $a$, $b$, $c$, $d$.   
In this case $\tilde\omega$ is defined in the same way as $\omega_u$, namely as the composition of the horizontal arrows in the last lines of the diagrams which replace 
diagrams \eqref{E31} and \eqref{E32} derived considering $\tilde \omega$ instead of $\omega_u$. 

To check part (c) in Proposition~\ref{P35}, observe first that 
$\kappa [\tilde S_r(f)]=\kappa[\tilde {\mathcal B}^o_{r-1}(f)]\oplus \kappa[\tilde {\mathcal B}^c_{r}(f)\sqcup \tilde{\mathcal J}_r( f)]$ 
and as pointed out by Lemma~\ref{OOO} at the end of Section~\ref{S2}, both linear maps $\Psi^o$ and $\Psi^c$ are actually isomorphisms of 
$\kappa[T^{-1}, T]$ modules; therefore so is $\tilde{\omega}_r$.  
Then one takes $\omega^N_r=\tilde{\omega}_r\otimes_{\kappa[T^{-1}, T]}\kappa [T^{-1}, T]]$. 
Clearly $\kappa [\tilde S_r(f)]\otimes_{\kappa[T^{-1}, T]}\kappa [T^{-1}, T]]=\kappa [T^{-1}, T]] [\mathcal B^c_r(f)\sqcup \mathcal B^o_{r-1}(f)]$ 
since $\kappa[ \tilde J( f)]$ as a  $\kappa[T^{-1}, T]$-module is a torsion module, cf.\ Lemma~\ref{OOO}. 
\qed
\end{proof}

\section{Stability for configurations $C_r(f)$. Proof of Theorem~\ref{T2}}\label{SS5}

The proof of Theorems~\ref{T2} and \ref{T3} will require an alternative definition of the configurations $C_r(f)$. 
This will be provided by the integer valued functions $\delta_r^f$ which will be defined for a proper real-valued tame map and then, 
via the infinite cyclic covering  for a tame angle-valued map. Ultimately they are defined for any continuous map. 

\subsection{Real valued maps}

For $f\colon X\to\mathbb R$ a map and $a,b\in\mathbb R$, introduce the notation
$X^f(a)=f^{-1}(a)$, $X^f_a= f^{-1}((-\infty, a])$, $X^b_f= f^{-1}([b,\infty))$, $^fX^b_a=X^a_f\cap X^f_b.$ 
Let $i^f_a\colon X^f_a\to X$ and $i^b_f\colon X^b_f\to X$ be the obvious inclusions.
Denote by 
\begin{align*}
\mathbb I^f_a(r)&:=\img\bigl(i^f_a (r)\colon H_r(X^f_a)\to H_r(X)\bigr),
\\
\mathbb I^b_f(r)&:=\img\bigl(i^b_f(r)\colon H_r( X^b_f)\to H_r(X)\bigr),
\end{align*} 
and let $F^f_r(a,b):=\dim (\mathbb I^f_a(r)\cap \mathbb I^b_f(r))$ and $G^f_r(a,b):=\dim H_r(X)/(\mathbb I^f_a(r) + \mathbb I^b_f(r))$.

Observe that: 

 \begin{lemma}\label{O3}\ 
\begin{enumerate}[(a)]
\item\label{O3:a}
For $a\leq a'$ and $b'\leq b$, we have $F^f_r(a,b)\leq F^f_r(a',b')$ and $G^f_r(a,b)\geq G^f_r(a',b')$.    
\item\label{O3:b}
If $|f-g|<\epsilon$ and $a\leq b$ then $F^f_r(a-\epsilon,b+\epsilon)\leq F^g_r(a, b)$ and $G^f_r(a,b) \leq G^g_r(a-\epsilon, b+\epsilon)$.
\item\label{O3:c}
$F^f_r(a,b)=F^{-f}_r(-b,-a)$ and $ G^f_r(a,b)= G^{-f}_r(-b,-a)$.
\end{enumerate}
\end{lemma}

\begin{proof}
To check \itemref{O3:a}, notice that $X^f_a\subseteq X^f_{a'}$ and $X^{b'}_f\supseteq X^b_f$ imply $\mathbb I^f_a\subseteq \mathbb I^f_{a'}$ and $\mathbb I^{b'}_f\subseteq \mathbb I^{b}_f$ and $\mathbb I^{b'}_f\subseteq \mathbb I^{b}_f$,
hence $\mathbb I^f_a\cap \mathbb I^b_f \subseteq \mathbb I^f_{a'} \cap \mathbb I_f^{b'}$ and $\mathbb I^f_a +\mathbb I^b_f \subseteq \mathbb I^f_{a'} + \mathbb I_f^{b'},$ then the statement. 
To check \itemref{O3:b}, notice that $|f-g|<\epsilon$ implies $f-\epsilon <g<f+\epsilon$ which implies $X^f_{a-\epsilon}\subseteq X^g_a$ and $X^f_{b+\epsilon} \subseteq X^g_b$.
These inclusions imply $\mathbb I^f_{a-\epsilon}\subseteq \mathbb I^g_{a}$ and $\mathbb I^{b+\epsilon}_f\subseteq \mathbb I_g^{b}$, 
hence $F^f(a-\epsilon, b+\epsilon)\leq F^g(a,b)$. The arguments for $G$ are similar.
To check \itemref{O3:c}, one uses the fact that $f^{-1}((-\infty,a])=(-f)^{-1}([-a,\infty))$.
\qed
\end{proof}

If $X$ is a compact ANR it is immediate that both $F^f_r(a,b)$ and $G^f_r(a,b)$ are finite since $\dim H_r(X)$ is finite. 
The same remains true for $f\colon X\to \mathbb R$ a tame map (hence proper) with $X$ not compact  despite the fact that $\dim H_r(X)$ is not necessarily finite. 

\begin{proposition}\label{P51} 
If $f\colon X\to\mathbb R$ is a tame map, then:
\begin{enumerate}[(a)]
\item\label{P51:a}
$F_r^f(a,b)<\infty$.
\item\label{P51:b}
$G^f_r(a,b)<\infty$.
\end{enumerate}
\end{proposition}

\begin{proof} To ease the writing, we  (sometimes) drop $f$  from notation. 
We start with \itemref{P51:a}.
In view of Observation~\ref{O3} it suffices to check the statements for $a>b$. 

Consider 
$$
i_a(r)-i^b(r)\colon H_r(X_a)\oplus H_r(X^b) \to H_r(X),
$$ 
let $p_1\colon H_r(X_a)\oplus H_r(X^b) \to H_r(X_a)$ be the first factor projection 
and observe that $\mathbb I^f_a(r)\cap\mathbb I^b_f(r)=(i_a(r)\cdot p_1)(\ker(i_a(r)-i^b(r))$.

Then  
$$
\dim\bigl(\mathbb I^f_a(r)\cap \mathbb I^b_f(r)\bigr)\leq\dim\ker\bigl(i_a(r)-i^b(r)\bigr).
$$
Since $a\geq b$ we have $X= X_a\cup X^b$.
In view of the Mayer--Vietoris long exact sequence associated with $X=X_a\cup X^b$ 
$$
\ker\bigl(i_a(r)-i^b(r)\bigr)=\img\bigl(H_r(X^b_a)\to H_r(X_a)\oplus H_r(X^b)\bigr)
$$  
has finite dimension since $\dim H_r(X^b_a)$ is finite. 

Next we prove \itemref{P51:b}.
If $a<b$ one uses the long exact sequence of the pair $(X,X_a\sqcup X^b)$ to conclude that  
$H_r(X)/(\mathbb I^f_a(r) + \mathbb I^b_f(r))$ is isomorphic to a subspace of $H_r(X, X_a\sqcup X^b)= H_r(X^b_a,X(a) \sqcup X(b))$ which is of finite dimension. 
Indeed, $f$ tame implies that $X(a)$, $X(b)$, and $X_a^b$ are compact ANRs, hence with finite dimensional homology.

If $a\geq b$ one uses the Mayer--Vietoris exact sequence associated with $X_a$, $X^b$ to conclude that  
$H_r(X)/(\mathbb I^f_a(r) + \mathbb I^b_f(r))$ is isomorphic to a subspace of $H_r(X^b_a)$ which is of finite dimension.
\qed
\end{proof}

Let $a<b$ and $c<d$. 
We refer to the set 
$$
B(a,b: c,d)=(a,b]\times[c,d)\subset\mathbb R^2,\qquad\textrm{$a<b$, $c<d$,}
$$  
as a ``box'', and define:
\begin{equation}\label{EQQ1}
\begin{aligned}
\mu_r^{F,f}(B)&=F_r^f(a,d) + F_r^f(b,c)- F_r^f(a,c) - F_r^f(b,d),
\\
\mu_r^{G,f}(B)&=-G_r^f(a,d) - G_r^f(b,c)+ G_r^f(a,c) + G_r^f(b,d).
\end{aligned}
\end{equation}

One has:

\begin{proposition}\label{P52}
If $X$ is compact or $f$ is a tame map, then:
\begin{enumerate}[(a)]
\item\label{P52a}
$\mu_r^{F,f}(B)=\mu_r^{G,f}(B)$.  
\item\label{P52b}
Putting $\mu^f_r(B):=\mu_r^{F,f}(B)=\mu_r^{G,f}(B)$, we have $\mu^f_r(B)\geq0$.
\item\label{P52c}
If $B=B_1\cup B_2$, $B_1\cap B_2=\emptyset$ with $B_1$, $B_2$ boxes, then $\mu^f (B)=\mu^f(B_1)+\mu^f(B_2)$.
\item\label{P52d} If if $B'$ and $B''$ are boxes with $B'\subseteq B''$ one has $\mu^f(B')\leq \mu^f(B'')$. 
\end{enumerate}
\end{proposition}

\begin{proof}  
To ease the writing, we drop $f$ and $r$ from notation and introduce:
\begin{align*}
I_1&:=\dim\bigl(\mathbb I_a\cap\mathbb I^d\bigr),
\\
I_2&:=\dim\bigl((\mathbb I_a\cap\mathbb I^c)/(\mathbb I_a\cap\mathbb I^d)\bigr),
\\
I_3&:=\dim\bigl((\mathbb I_b\cap\mathbb I^d)/(\mathbb I_a\cap \mathbb I^d)\bigr),
\\
I_4&:=\dim\bigl((\mathbb I_b\cap \mathbb I^c)/(\mathbb I_a\cap\mathbb I^c+\mathbb I_b\cap\mathbb I^d)\bigr),
\\
I_5&:=\dim\bigl(\mathbb I_b/(\mathbb I_a+\mathbb I_b\cap\mathbb I^c)\bigr),
\\
I_6&:=\dim\bigl(\mathbb I^c/(\mathbb I_a\cap\mathbb I^c+\mathbb I^d)\bigr), 
\\
I_7&:=\dim\bigl(H/(\mathbb I_b+\mathbb I^c)\bigr),\qquad\textrm{with $H=H_r(X)$.}
\end{align*}
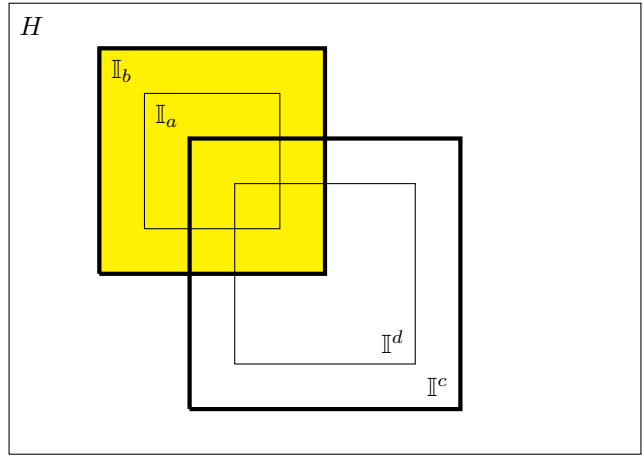
\begin{figure}
\begin{tikzpicture}[scale=0.3]
\draw(0,0) -- (0,20) -- (28,20) -- (28, 0) -- (0,0) node at (1,19) {$H$};
\draw[ultra thick, fill=yellow](4,8) -- (4,18) -- (14,18) -- (14, 8) -- (4,8) node at (5,17) {$\mathbb I_b$};
\draw(6,10) -- (6,16) -- (12,16) -- (12, 10) -- (6,10) node at (7,15) {$\mathbb I_a$};
\draw[ultra thick](8,2) -- (8,14) -- (20,14) -- (20, 2) -- (8,2) node at (19,3) {$\mathbb I^c$};
\draw(10,4) -- (10,12) -- (18,12) -- (18, 4) -- (10,4) node at (17,5) {$\mathbb I^d$};
\end{tikzpicture}
\caption{An illustration for the proof of Proposition~\ref{P52}.}
\label{F:boxes}
\end{figure}
Using Figure~\ref{F:boxes}, it is not hard to notice that: 
\begin{align*}
F(a,d)&=I_1, & G(a,d)&=I_7 +I_6 +I_5 + I_4,
\\
F(b,c)&=I_1+ I_2 + I_3 + I_4, & G(b,c)&=I_7,
\\
F(a,c)&=I_1 + I_2, & G(a,c)&=I_7 +I_5,
\\
F(b,d)&=I_1 +I_3, & G(b,d)&=I_7 + I_6.
\end{align*}
Then we have: 
\begin{multline*}
F(a,d)+F(b,c)-F(a,c)-F(b,d)
\\=I_1+(I_1+I_2+I_3+I_4)-(I_1+I_2)-(I_1+I_3)=I_4
\end{multline*}
and
\begin{multline*}
G(a,d)+G(b,c)-G(a,c)-G(b,d)
\\=(I_7+I_6+I_5+I_4)+I_7-(I_7+I_5)-(I_7+I_6)=I_4.
\end{multline*}
These equalities establish \itemref{P52a} and \itemref{P52b}.
Part \itemref{P52c} follows from (\ref {EQQ1}) by inspecting the all relative positions of $B_1$ and $B_2$ as disjoint subsets of $B.$
To check part \itemref{P52d} one tiles $B''$ as a disjoint union of boxes $B''= B''_1\sqcup B''_2\cdots \sqcup B''_{r-1}\sqcup B''_r$ s.t. $B'$ is one of these boxes and use  item \itemref{P52c} to inductively derive that $\mu(B'')= \sum _{1\leq i\leq r} \mu(B''_u).$ This implies the result  

\qed
\end{proof}

Note that  both Propositions (\ref{P51}) and (\ref{P52}) remain valid for $f$ a proper continuous  map  but with more elaborated arguments and  this is not 
not needed  in this paper.

Define the \emph{jump function}, $\delta^f_r\colon\R^2\to\mathbb \mathbb Z_{\geq 0}$, by 
\begin{equation}\label{EQQ2}
\delta^f_r(a,b):=\lim_{\epsilon \to 0}\mu^f\bigl((a-\epsilon, a+\epsilon]\times [b-\epsilon, b+\epsilon)\bigr).  
\end{equation}
The limit exists since, by Proposition~\ref{P52}\itemref{P52c}, the right side decreases when $\epsilon$ decreases.
This function has values in $\mathbb Z_{\geq 0}.$ Since the critical values of a tame map are discrete, $\delta^f_r$ has discrete support and satisfies the following proposition.

\begin{proposition}\label{P53}
If $X$ is compact or $f$ is a tame map then:
\begin{enumerate}[(a)]
\item\label{P53a}
For $a<b$ and $c<d$ one has $\mu^f_r\bigl((a,b]\times [c,d)\bigr)=\sum_{a<x\leq b, c\leq y<d}\delta^f_r (x,y)$.
\item\label{P53b}
$F^f_r(b,c)=\sum_{x\leq b, c\leq y}\delta^f_r(x,y)$.  
\item\label{P53c}
$G^f_r(a,d)=\sum_{a\leq x, y\leq d}\delta^f_r(x,y)$. 
\end{enumerate}
\end{proposition}

\begin{proof}
Item~\itemref{P53a} follows from Proposition~\ref{P52}\itemref{P52c} as shown below.

First observe that in in view of (\ref {EQQ1}) if both $(b-a)$ and $(d-c)$ are small  enough then $\mu^f_r\bigl((a,b]\times [c,d)\bigr)= \delta^f_r(c,d).$ 
Then choose  subdivisions $a_0=a <a_1 <\cdots < a_r= b$ and  $c=c_0 <c_1 <\cdots < c_k= d$ such that all critical values between $a$ and $b$ rep. between $c$ and $d$ are among $a_i$'s and $c_j$'s respectively and both $a_{i}- a_{i-1}$ and $c_{j}- c_{j-1}$ are small enough s.t. $\mu^f_r((a_{i-1}, a_i]\times [ c_i, c_{i+1}))= \delta^f_r(a_i, c_i).$  Applying Proposition~\ref{P52}\itemref{P52c}  one obtains $\mu^f_r\bigl((a,b]\times [c,d)\bigr)= \sum_{i\geq 1, j \geq 1} \mu^f_r\bigl((a_{i-1},a_i]\times [c_{j-1}, c_j,d)\bigr)$ which implies the result as stated.

Item \itemref{P53b}  follows from \itemref{P53a} by letting  $a\to-\infty$ and $d\to\infty$.
Similarly, item \itemref{P53c} follows from item \itemref{P53a} by letting $b\to\infty$ and $c\to-\infty$.  
\qed
\end{proof}

Since for a tame map $f$ the set of critical values is discrete we write them as a sequence  $\cdots<c_{i-1} < c_i<c_{i+1}<\cdots$ and 
define  
$$
\epsilon (f)=\inf_{i\in \mathbb Z}(c_{i+1}-c_i).
$$
Clearly, if $f\colon X\to\mathbb R$ is tame with $X$ compact, then $\epsilon(f)>0$ 
and if $f\colon X\to \mathbb S^1$ is tame then the infinite cyclic covering $\tilde f\colon\tilde X\to \mathbb R$ is tame and $\epsilon (\tilde f)>0$.

\begin{proposition}\label{O12}
Let  $f\colon X\to \mathbb R$ be a tame map with $\epsilon (f)>0$.  
For any $\epsilon,\epsilon'<\epsilon(f)$ one has:
\begin{equation}\label{E:O12a}
F^f_r(c_i, c_j) = F^f_r(c_i+\epsilon, c_j-\epsilon') = F^f_r(c_{i+1}-\epsilon, c_{j-1}+\epsilon'),
\end{equation} 
and
\begin{equation}\label{E:O12b}
\delta^f _r(c_i, c_j)= F^f_r(c_{i-1}, c_{j+1}) + F^f_r(c_i, c_j)- F^f_r(c_{i-}, c_j) - F^f_r(c_i, c_{j+1}).
\end{equation}
\end{proposition}

\begin{proof}
The tameness of $f$ and the hypothesis $\epsilon,\epsilon'<\epsilon(f)$ imply that the inclusions $X^f_{c_i}\subseteq X^f_{c_i+\epsilon}$, $X^f_{c_i}\subseteq X^f_{c_{i+1}-\epsilon'}$  
and $X_f^{c_j-\epsilon}\supseteq X_f^{c_j}$, $X_f^{c_{j-1}+\epsilon'}\supseteq X_f^{c_j}$ induce isomorphisms in homology. 
These facts imply that $\mathbb I^f_{c_i}=\mathbb I^f_{c_i+\epsilon}=\mathbb I^f_{c_{i+1}-\epsilon'}$ and $\mathbb I_f^{c_{j-1}+\epsilon}=\mathbb I_f ^{c_j-\epsilon'}=\mathbb I_f ^{c_j}$
which imply \eqref{E:O12a}. 
To check \eqref{E:O12b}, recall that in view of (\ref {EQQ1}), (\ref{EQQ2}) and (\ref{E:O12a}), for $\epsilon$ very small, one has 
$\delta^f_r (c_i, c_j)= F_r(c_i-\epsilon, c_j+\epsilon) + F_r(c_i+\epsilon, c_j-\epsilon)- F_r(c_i-\epsilon, c_j-\epsilon) - F_r(c_i+\epsilon, c_j+\epsilon)$.
Thus \eqref{E:O12b} follows from \eqref{E:O12a} by taking $\epsilon<\epsilon (f)$. 
\qed
\end{proof}

For a pair $(a,b)\in \mathbb R^2$ and $\epsilon >0$ consider the box $B(a,b;2\epsilon)=(a-2\epsilon,a+2\epsilon]\times [b-2\epsilon, b+2\epsilon)$.

\begin{proposition}\label{P55}
Let $f\colon X\to\mathbb R$ be a tame map with $\epsilon (f)>0$ and 
$\epsilon<\epsilon(f)/4.$ For  any tame map $g$ with $|f-g|<\epsilon$ and any $(a,b)\in\supp\delta^f_r$ the following holds:
\begin{enumerate}[(a)]
\item\label{P55a}
$\supp(\delta^f_r)\cap B(a,b;2\epsilon)\equiv (a,b)$       
\item\label{P55b}
$\sharp\bigl(\supp(\delta^g_r)\cap\bigl(\bigsqcup_{(a,b)\in\supp\delta^f_r}B(a,b;2\epsilon)\bigr)\bigr)=\sharp\supp(\delta^f_r)$.
\end{enumerate}
In particular, if the cardinality of the supports\footnote{Recall that the cardinality of the support is the sum of multiplicity of the elements in the support.} 
of $\delta^f_r$ and $\delta^g_r$ are equal 
and $|g-f|<\epsilon$, then the support of $\delta^g_r$ lies in an $\epsilon$-neighborhood\footnote{Here $\epsilon$-neighborhood of $(a,b)$ 
means the domain $(a-\epsilon,a+\epsilon)\times (b-\epsilon,b+\epsilon)$.} of the support of $\delta^f_r$. 
\end{proposition}

This Proposition is closed to Box Lemma in \cite {CEH07} page 112.
\begin{proof}
Item~\itemref{P55a} follows from the definition of $\delta^f_r$.
To prove item~\itemref{P55b} observe that if $(a,b)\in\supp\delta^f$ both numbers have to be critical values, hence the $a=c_i$, $b=c_j$. 
In view of Proposition~\ref{O12}, for any $\epsilon',\epsilon''<\epsilon(f)$ one has:
\begin{equation}\label{E37}
\begin{aligned}
F^f_r(c_{i-1}, c_{j+1})&=F^f_r(a-\epsilon', b+\epsilon'')\\
F^f_r(c_i, c_j)&=F^f_r(a+\epsilon', b-\epsilon'')\\
F^f_r(c_i, c_{j+1})&=F^f_r(a+\epsilon', b+\epsilon'')\\
F^f_r(c_{i-1}, c_j)&=F^f_r(a-\epsilon', b-\epsilon'')
\end{aligned}
\end{equation}
Since $|f-g|<\epsilon$, in view of Observation~\ref{O3} one has:
\begin{equation}\label{E38}
\begin{array}{rcccl}
F^f_r(a-3\epsilon,b+3\epsilon)&\leq&F^g_r(a-2\epsilon,b+2\epsilon)&\leq&F^f_r(a- \epsilon,b+\epsilon) \\
F^f_r(a+ \epsilon,b- \epsilon)&\leq&F^g_r(a+2\epsilon,b-2\epsilon)&\leq&F^f_r(a+3\epsilon,b-3\epsilon)\\
F^f_r(a+ \epsilon,b+3\epsilon)&\leq&F^g_r(a+2\epsilon,b+2\epsilon)&\leq&F^f_r(a+3\epsilon,b+\epsilon) \\
F^f_r(a-3\epsilon,b- \epsilon)&\leq&F^g_r(a-2\epsilon,b-2\epsilon)&\leq&F^f_r(a- \epsilon,b-3\epsilon)
\end{array}
\end{equation}
Since $\epsilon<\epsilon(f)/4$, equations~\eqref{E37} and \eqref{E38} imply:
\begin{equation}\label{E39}
\begin{aligned}
F^g_r(a-2\epsilon, b+2\epsilon)&=F^f_r(c_{i-1}, c_{j+1})\\
F^g_r(a+2\epsilon, b-2\epsilon)&=F^f_r(c_i,c_j)\\
F^g_r(a+2\epsilon, b+2\epsilon)&=F^f_r(c_{i},c_{j+1})\\
F^g_r(a-2\epsilon, b-2\epsilon)&=F^f_r(c_{i-1}, c_{j}) 
\end{aligned}
\end{equation}
In view of Proposition~\ref{P53} we have
\begin{equation*}
\begin{aligned}
\sharp\bigl(\supp(\delta^g_r)\cap B(a,b;2\epsilon)\bigr)
&=\mu^g_r\bigl(B(a,b;2\epsilon)\bigr)
\\&=F^g_r(a-2\epsilon,b+2\epsilon)+F^g_r(a+2\epsilon,b-2\epsilon)
\\&\qquad-F^g_r(a-2\epsilon,b-2\epsilon)-F^g_r(a+2\epsilon,b+2\epsilon).
\end{aligned}
\end{equation*}
Using the equations~\eqref{E39} and the equation~\eqref{E:O12b} in Proposition~\ref{O12} one derives
$$
\sharp\bigl(\supp(\delta^g_r)\cap B(a,b;2\epsilon)\bigr)=\sharp\bigl(\supp(\delta^f_r)\cap B(a,b;2\epsilon)\bigr)=\delta^f_r(a,b),
$$ 
\qed
\end{proof}

\subsection{Angle valued maps} 

Let $f\colon X\to\mathbb S^1$ be a tame map and $\tilde f\colon\tilde X\to \mathbb R$ its infinite cyclic covering.  
Note that $\epsilon(\tilde f)>0.$ Observe that 
\begin{equation}\label{EEE1}
\delta^{\tilde f}_r(a, b)=\delta^{\tilde f}_r(a+2\pi, b+2\pi).
\end{equation}

Consider the projection $p\colon\mathbb R^2\to \mathbb T=\mathbb R^2/ \mathbb Z$, with $\mathbb T$ the quotient space of $\mathbb R^2$ by the action 
$\mathbb Z\times \mathbb R^2\rightarrow  \mathbb R^2$ given by 
$(n, (a,b))\rightarrow (a+2\pi n, b+2\pi n)$. Write $p(a,b)=\langle a, b\rangle.$
Define 
$$
\epsilon (f):=\epsilon (\tilde f)
$$ 
and  
\begin{equation}\label{EQQ}
\delta^f_r(\langle a, b \rangle):= \delta^{\tilde f}_r(a,b).
\end{equation}

In view of \eqref{EEE1}, $\delta^f_r\colon\mathbb T\to \mathbb Z_{\geq 0}$ is a well defined  function with finite support and Proposition~\ref{P55} holds 
true for $f\colon X\to \mathbb S^1.$ 

For  the proof of Theorem~\ref{T2} we also need to show that $\delta^f_r$ and $C_r(f),$ when viewed as functions on $\mathbb T,$ are equal.

\begin{proposition}\label{O78}
If $f$ is a tame real- or angle-valued map defined on $X$, a compact ANR, then $\delta^f_r$ and $C_r(f)$ are equal $\mathbb Z_{\geq 0}-$valued functions defined on $\mathbb R^2$ or $\mathbb T$. 
\end{proposition}

\begin{proof}
We check the case of an angle valued map $f\colon X\to \mathbb S^1$ only. 
The real valued case can be regarded as a particular case of the angle valued map. 
First note that $\epsilon (f)>0$. 
In view of the definition of $\delta^{\tilde f}_r$ it suffices to check that:
\begin{enumerate}[(i)]
\item\label{O78i}
If at least one, $a$ or $b,$ is not a critical value then  
we have $\delta^{\tilde f}_r(a,b)=0$.
\item\label{O78ii}
If $a=c_i$ and $b=c_j$ are critical values with $c_i\geq c_j$, then
\begin{equation*}
\delta^{\tilde f}_r(c_i,c_j)=\sharp\bigl\{I\in \tilde{\mathcal B}^c_r(f)\bigm|I=[c_j, c_i]\bigr\}.
\end{equation*}
\item\label{O78iii}
If $a=c_i$ and $b=c_j$ are critical values with $c_i<c_j$, then
\begin{equation*}
\delta^{\tilde f}_r(c_i,c_j)=\sharp\bigl\{I\in \tilde{\mathcal B}^o_{r-1}(f)\bigm|I=(c_j, c_j)\bigr\}.
\end{equation*}
\end{enumerate}

Recall that $\delta^{\tilde f}_r(a,b)=\lim_{\epsilon \to 0}(-F^{\tilde f}_r(a-\epsilon, b-\epsilon)- F^{\tilde f}_r(a+\epsilon, b+\epsilon)+ F^{\tilde f}_r(a-\epsilon, b+\epsilon) + F^{\tilde f}_r(a+\epsilon, b-\epsilon))$.
In view of Proposition~\ref{O12}, if $a$ is not a critical value, $y\in\mathbb R$ and $\epsilon>0$ is sufficiently small, then $F^{\tilde f}_r(a-\epsilon,y)=F^{\tilde f}_r(a+\epsilon,y)$ 
and thus $\delta^{\tilde f}_r(a,y)=0$. Similarly, if $b$ is not a critical value, $x\in\mathbb R$ and $\epsilon>0$ is sufficiently small, then
$F^{\tilde f}_r(x,b-\epsilon)=F^{\tilde f}_r(x,b+\epsilon)$ and thus $\delta^{\tilde f}_r(x,b)=0$. This establishes statement \itemref{O78i}.

Suppose that $a=c_i$ and $b=c_j$ critical values. In view of Proposition~\ref{O12} and of the definition of $\delta^{\tilde f}$ one has
\begin{equation}\label{EQ45}
\delta^{\tilde f}_r(c_i, c_j)= -F^{\tilde f}_r(c_{i-1},c_j) - F^{\tilde f}_r(c_i,c_{j+1}) +F^{\tilde f}_r(c_{i-1}, c_{j+1}) + F^{\tilde f}_r(c_i, c_j).
\end{equation}
By Propositions~\ref{P36} (b) when $c_i\geq c_j,$ one has
\item 
\begin{multline}\label{EQ46}
F^{\tilde f}_r(c_i, c_j)
=\sharp\bigl\{I\in\tilde{\mathcal B}_r^c(f)\bigm|I \cap [c_j,c_i]\ne\emptyset\bigr\}
\\+\sharp\bigl\{I\in\tilde{\mathcal B}^o_{r-1}(f)\bigm|I\subset(c_j,c_i)\bigr\}
+\sharp\tilde{\mathcal J}_r(f),
\end{multline}
and when $c_i>c_j$
one has
\begin{equation}\label{EQ47}
F^{\tilde f}_r(c_i,c_j)
=\sharp\bigl\{I\in\tilde{\mathcal B}_r^c(f)\bigm| I\supset[c_i,c_j]\bigr\}
+\sharp\tilde{\mathcal J}_r(f). 
\end{equation}

Indeed Proposition~\ref{P36} (b) calculates $\mathbb I^{\tilde f}_a(r)$ and $\mathbb I_{\tilde f}^b(r)$ as the $\kappa-$vector space generated by 
$$\{ I\in \widetilde {\mathcal B}_r^c \mid  I\cap (-\infty, a]\ne \emptyset\}\cup \{  I\in \widetilde {\mathcal B}_{r-1}^o \mid  I\subset (-\infty, a)\}$$
and $$\{ I\in \widetilde {\mathcal B}_r^c \mid  I\cap ([b,\infty )\ne \emptyset\}\cup \{  I\in \widetilde {\mathcal B}_{r-1}^o \mid  I\subset (b, \infty)\}$$
of which the above descriptions follow.

Comparing the collections of bar codes whose cardinality are given by $F^{\tilde f}_r(c_{i-1},c_j)$, $F^{\tilde f}_r(c_i,c_{j+1})$, $F^{\tilde f}_r(c_{i-1}, c_{j+1})$ 
and $F^{\tilde f}_r(c_i,c_j)$ and using \eqref{EQ45}, \eqref{EQ46} and \eqref{EQ47} one derives the statement 
\itemref{O78ii}, and 
\itemref{O78iii}.
\qed
\end{proof}

To prove Theorem ~\ref{T2} 
one  begins with a few observations. 
\begin{enumerate}[(i)]
\item\label{IT2i}
Consider the space of continuous maps $C(X,\mathbb S^1)$, $X$ a compact ANR, with the compact open topology.
This topology is induced from the metric $D(f,g):=\sup_{x\in X}d(f(x),g(x))$ with $d(\theta_1,\theta_2)$   
given by $d(\theta_1,\theta_2)=\inf(|\theta_1-\theta_2|,2\pi-|\theta_1-\theta_2|)$, $0\leq\theta_1,\theta_2<2\pi$.
Equipped with this metric $(C(X,\mathbb S^1),D)$ is a complete metric space.

Recall that the set of connected components of the space $C(X,\mathbb S^1)$ identifies to $H^1(X;\mathbb Z)$. 
Denote by $C_\xi(X,\mathbb S^1)$ the connected component corresponding to the class $\xi\in H^1(X;\mathbb Z)$ and by $C_{\xi,t}(X,\mathbb S^1)$ the subset of tame maps in this connected component equipped with the induced topology.

\item\label{IT2ii}  
Observe that if $f,g$ are in a connected component $C_\xi(X,\mathbb S^1)$ of $C(X,\mathbb S^1)$ and $D(f,g)<\pi$ then for any $t\in [0,1]$ 
the map $h_t:=h_t(f,g) \in C(X;\mathbb S^1)$  
defined below  lies in the connected component of $C_\xi(X,\mathbb S^1).$ Moreover for any $0=t_0<t_1<\cdots<t_{N-1}<t_N=1$ one has 
\begin{equation}\label{E0}
D(f,g)=\sum_{0\leq i<N} D(h_{t_{i+1}},h_{t_i}).
\end{equation} 
Considering the inclusion of $\mathbb S^1\subset \mathbb R^2$ as the unit circle centered at origin, if one regards $f$ and $g$ as $\mathbb R^2$-valued maps, the map $h_t$ is defined by  
\begin{equation*}
h_t(x)=\frac{t f(x) + (1-t) g(x)}{\|t f(x) + (1-t) g(x)\|}.
\end{equation*}

\item\label{IT2iii} 
Recall that $f$ is a p.l. map on $X$ if with respect to some subdivision is simplicial (i.e.\ the liftings to $\mathbb R$ of the restriction of $f$ to simplexes of the subdivision are linear) 
and for any two p.l.\ maps $f,g$ there exists a common subdivision of $X$ which makes $f$ and $g$ simultaneously simplicial, hence any $h_t$ is a simplicial map.

If $X$ is a simplicial complex and $\mathcal U\subset C_\xi(X,\mathbb S^1)$ denotes the subset of p.l. maps then: 
\begin{enumerate}[(1)]
\item\label{Ipl1}
$\mathcal U$ is a dense subset in $C_\xi(X,\mathbb S^1)$.
\item\label{Ipl2}
$f,g\in \mathcal U$ implies $h_t \in \mathcal U$ hence  $\epsilon(h_t)>0$ hence for any $t\in[0,1]$ there exists $o(t)>0$ so that $|t'-t|< o(t)$ implies $D(h_{t'}, h_t)<\epsilon (h_t)/6$.
\end{enumerate}

Indeed 
(1) follows from approximability of continuous maps by p.l.\ maps and 
(2) from the continuity in $t$ of the family $h_t$ and the compacity of $X$. 

\item\label{IT2iv}
For $k$ a positive integer consider $S^k\mathbb T=(\mathbb T\times\cdots\times\mathbb T)/\Sigma_k$, with $\Sigma_k$  the $k$-symmetric group acting on the $k$-fold cartesian product of $\mathbb T$ 
by permutations equipped  with the metric $\underline D$ induced from the complete metric on $\mathbb T=\mathbb R^2/\mathbb Z$. 
With this metric $(S^N(\mathbb T),\underline D)$ is a complete metric space.

\item\label{IT2v}
Proposition~\ref{P55} states that $f,g\in C(X,\mathbb S^1)_{t, \xi}$ and $D(f,g) <\epsilon(f)/6$ implies
\begin{equation}\label{E41}
\underline D(\delta^f_r,\delta^g_r)<2D(f,g).
\end{equation} 
\end{enumerate}
\vskip .1in

\noindent {\bf Proof of Theorem \ref{T2}:}

Observation~\itemref{IT2v} makes 
the assignment 
$C(X,\mathbb S^1)_{t, \xi}\ni f\mapsto\delta^f_r\in S^{\beta^N_r(X,\xi)} (\mathbb T)$ a continuous map.  
 
In order to conclude the existence of a continuous extension of $\delta^f_r$ to the entire $C_\xi(X,\mathbb S^1)$, in view of 
the completeness of the metrics $D,$ and $\underline D,$ stated in observations~\itemref {IT2i} and  \itemref{IT2iv} above,  
it suffices to show that for a Cauchy sequence $\{f_\alpha\}$, $f_\alpha\in \mathcal U$, the sequence $\delta_r^{f_\alpha}$ is a Cauchy sequence in $S^{\beta^N_r(X,\xi)}(\mathbb T).$ 
This will follow once we can show that \eqref{E41} holds for  any two $f,g\in\mathcal U$ with $d(f,g)<\pi$. 
 To establish this we proceed as in \cite[Section~3.3] {CEH07}.

Start with $f,g\in \mathcal U$ with $D(f,g)<\pi$ and consider $h_t,$ $t\in [0,1]$ defined in \itemref{IT2ii} above.
Choose a finite sequence $0=t_0<t_2<t_4<\cdots<t_{2L-2}<t_{2L}=1,$ $L$positive integer, so that the open intervals 
$I_{2i}=(t_{2i}-o(t_{2i}),t_{2i}+o(t_{2i}))$ cover $[0,1]$ with $o(t)$ from item~\itemref {IT2iii} (2). 
The compacity of $[0,1]$ makes such choice possible.

By possibly removing some of the points $t_{2i}$ and decreasing $o(t_{2i})$ if necessary one can make $I_{2i}\cap I_{2i+2}\ne\emptyset$ and $t_{2t_-2},t_{2i+2}\notin I_{2i}$.
Choose $t_1<t_3<\cdots<t_{2L-1}$ with $t_{2i}<t_{2i+1}<t_{2i}$ and $t_{2i+1}\in I_{2i}\cap I_{2i+2}$.
We have then $|t_{2i+1}-t_{2i}|<o(t_{2i})$ and $|t_{2i+2}-t_{2i+1}|<o(t_{2i+2})$.

In view of item~\itemref{IT2iii}, $|t_{2i+1}-t_{2i}|< o(t_{2i})$ implies 
$$D(h_{t_{2i}},h_{t_{2i+1}})<\epsilon(h_{t_{2i}})/3$$  
and $|t_{2i+2}-t_{2i+1}|<o(t_{2i+2})$ implies $$D(h_{t_{2i+2}},h_{t_{2i+1}})<\epsilon(h_{t_{2i+2}})/6.$$ 
In view of item~\itemref{IT2v} the last inequalities imply $$\underline D(\delta^{h_{t_{2i+1}}}_r,\delta^{h_{t_{2i}}}_r) < 2 D(h_{t_{2i}},h_{t_{2i+1}})$$ 
as well as
$$\underline D(\delta^{h_{t_{2i+2}}}_r, \delta^{h_{t_{2i+1}}}_r) < 2 D(h_{t_{2i+2}},h_{t_{2i+1}}).$$  
Therefore,  
for any $0\leq k\leq {2L-1}$ one has $\underline D(\delta^{h_{t_{k+1}}}_r, \delta^{h_{t_{k}}}_r) < 2 D(h_{t_{k+1}},h_{t_k})$.
Then by \eqref{E41} and  \eqref{E0} cf item (ii), one obtains  
$$
\underline D(\delta^f, \delta^g) \leq \sum_{0\leq i<2L-1} \underline D(\delta^{h_{t_{i+1}}}, \delta^{h_{t_i}})  \leq 2\sum_{0\leq i<2L-1}  D(h_{t_{i+1}}, h_{t_i}) = D(f,g).
$$ 
This finishes the proof of Theorem~\ref{T2}.

\section{Poincar\'e duality for configurations $C_r(f)$. Proof of Theorem~\ref{T3}}\label{SS6}

For an $n$-dimensional manifold, not necessarily compact, Poincar\'e duality can be better formulated using Borel--Moore homology, 
cf.~\cite{BM}, especially tailored for locally compact spaces $Y$ and pairs $(Y,K)$, $K$ closed subset of $Y$. 
Borel--Moore homology coincides with the standard homology when $Y$ is compact. 
In general, for a locally compact space $Y$, it can be described as the inverse limit of the homology vector spaces $H_r(Y,Y\setminus U)$ for all $U$ open sets with compact closure.

One denotes by $H^{\BM}_r(\cdots)$ the Borel--Moore homology in dimension $r.$ 
For $Y$ an $n$-dimensional topological $\kappa$-orientable manifold, $g\colon Y\to \mathbb R$ a tame map, hence a proper continuous map, 
and $a$ a regular value of $g$,\footnote{i.e.\ $f\colon f^{-1}(a-\epsilon, a+\epsilon)\to (a-\epsilon, a+\epsilon)$ is a fibration}  
Poincar\'e duality provides the commutative diagrams
\begin{equation}\label{OD2}
\vcenter{
\xymatrix{
H^{\BM}_r(Y_a)         \ar[d]\ar[r] & H^{\BM}_r(Y)      \ar[d]\ar[r] & H^{\BM}_r(Y, Y_a)   \ar[d] \\
H^{n-r}(Y,Y^a)         \ar[d]\ar[r] & H^{n-r}(Y)        \ar[d]\ar[r] & H^{n-r}(Y^a)        \ar[d] \\
(H_{n-r}(Y,Y^a))^\ast  \ar[r]       & (H_{n-r}(Y))^\ast \ar[r]       & (H_{n-r}(Y^a))^\ast 
}}
\end{equation}
and
\begin{equation}\label {OD2'}
\vcenter{
\xymatrix{
H^{\BM}_r(Y^a)\ar[d]  \ar[r] & H^{\BM}_r(Y)       \ar[d]\ar[r] & H^{\BM}_r(Y,Y^a)     \ar[d] \\
H^{n-r}(Y,Y_a)\ar[d]  \ar[r] & H^{n-r}(Y)         \ar[d]\ar[r] & H^{n-r}(Y_a)         \ar[d] \\
(H_{n-r}(Y,Y_a))^\ast \ar[r] & (H_{n-r}(Y))^\ast  \ar[r]       & (H_{n-r}(Y_a))^\ast.
}}
\end{equation}
Recall that $Y_a= f^{-1}((-\infty, a])$ and $Y^a= f^{-1} ([a, \infty)).$
The first vertical arrow in each column of the diagrams~\eqref{OD2} and \eqref{OD2'} is the Poincar\'e duality isomorphism, 
the second is the isomorphism between cohomology and the dual of homology with coefficients in a field.  
The horizontal arrows are induced by the inclusions of $Y_a$ or $Y^a$ in $Y$ and the inclusion of the pairs $(Y,\emptyset)$ in $(Y, Y_a)$ or $(Y,Y^a)$.

We apply diagrams~\eqref{OD2} and \eqref {OD2'} to $Y=\tilde M$ and $g=\tilde f$, where $\tilde M$ is an infinite cyclic cover of $M$ defined by $f\colon M\to \mathbb S^1,$ a tame map, $M$ a closed $\kappa$-orientable topological manifold, 
and $\tilde f\colon\tilde M\to \mathbb R$ the lift of $f$ to $\tilde M.$ One obtains  the commutative diagrams
\begin{equation}\label{PD1}
\vcenter{
\xymatrix{
H^{\BM}_r(\tilde M_a)               \ar[d]\ar[rr]^{i^{\BM}_a(r)} && H^{\BM}_r(\tilde M)      \ar[d]\ar[rr]^{j^{\BM}_a(r)} && H^{\BM}_r(\tilde M,\tilde M_a) \ar[d] \\
H^{n-r}(\tilde M,\tilde M^a)        \ar[d]\ar[rr]^{s^a(n-r)}     && H^{n-r}(\tilde M)        \ar[d]\ar[rr]^{r^a(n-r)}     && H^{n-r}(\tilde M^a)            \ar[d] \\
(H_{n-r}(\tilde M,\tilde M^a))^\ast \ar[rr]^{(j^a(n-r))^\ast}    && (H_{n-r}(\tilde M))^\ast \ar[rr]^{(i^a(n-r))^\ast}    && (H_{n-r}(\tilde M^a))^\ast
}}
\end{equation}
and 
\begin{equation}\label{PD1'}
\vcenter{
\xymatrix{
H^{\BM}_r(\tilde M^b)               \ar[d]\ar[rr]^{i^{\BM,b}(r)} && H^{\BM}_r(\tilde M)      \ar[d]\ar[rr]^{j^{\BM,b}(r)} && H^{\BM}_r(\tilde M,\tilde M^b)\ar[d] \\
H^{n-r}(\tilde M,\tilde M_b)        \ar[d]\ar[rr]^{s_b(n-r)}     && H^{n-r}(\tilde M)        \ar[d]\ar[rr]^{r_b(n-r)}     && H^{n-r}(\tilde M_b)\ar[d] \\
(H_{n-r}(\tilde M,\tilde M_b))^\ast \ar[rr]^{(j_b(n-r))^\ast}    && (H_{n-r}(\tilde M))^\ast \ar[rr]^{(i_b(n-r))^\ast}    && (H_{n-r}(\tilde M_b))^\ast.
}}
\end{equation}

For $\tilde M$, $\tilde M_a$, and $\tilde M^a$ the Borel--Moore homology can be described as the following inverse limits:
\begin{equation}\label{E50}
\begin{aligned}
H^{\BM}_r(\tilde M)            & = \varprojlim_{0< l\to \infty} H_r(\tilde M, \tilde M_{-l}\sqcup \tilde M^l),\\
H^{\BM}_r(\tilde M_a)          & = \varprojlim_{0< l\to \infty} H_r(\tilde M_a, \tilde M_{a-l}),\\
H^{\BM}_r(\tilde M^a)          & = \varprojlim_{0< l\to \infty} H_r(\tilde M^a, \tilde M^{a+l}),\\
\end{aligned}
\end{equation}
The inclusions of pairs $(\tilde M,\tilde M_{-l'}\sqcup \tilde M^{l'})\subseteq (\tilde M, \tilde M_{-l}\sqcup \tilde M^l)$ for $l'>l$ 
induce in homology an inverse system whose limit is $H^{\BM}_r(\tilde M)$.
Similar inclusions of pairs associated with $l'>l$ induce inverse systems whose limits are the remaining Borel--Moore homology vector spaces considered above.

The horizontal arrows in both diagrams are inclusion induced linear maps in Borel-Moore homology, cohomology and homology.

In view of the use of Borel--Moore homology, in addition to $\mathbb I^{\tilde f}_a(r)$ and $\mathbb I_{\tilde f}^a(r)$, one considers 
\begin{align*}
\mathbb I^{\BM,\tilde f}_a(r)   &= \img\bigl(H^{\BM}_r(\tilde M_a) \to H^{\BM}_r(\tilde M)\bigr),
\\
\mathbb I^{\BM,a}_{\tilde f}(r) &= \img\bigl(H^{\BM}_r(\tilde M^a) \to H^{\BM}_r(\tilde M)\bigr),
\end{align*} 
and $F^{\BM,f}_r(a,b)= \dim\bigl(\mathbb I^{\BM,\tilde f}_a(r) \cap \mathbb I^{\BM,b}_{\tilde f}(r)\bigr)$.
\vskip .2in

The first step in the proof of $Theorem \ref{T3}$ is the verification of the equality
$$F^{BM,\tilde f}_r(a,b)=  G^f_{n-r}(b,a),$$ the second the verification of the equality 
$$F^{BM,\tilde f}_r(a,b)+ \sharp \tilde {\mathcal J}_r(f)= F^{\tilde f}_r(a,b)$$ and the third step  the verification of the equality
$$\delta^{\tilde f}_{n-r} (b,a)= \delta^{\tilde f}_r(a,b).$$
In view of the definition of $\delta^f_r (\langle a, b\rangle)$, cf. (\ref{EQQ}), the last equality implies $\delta^f_r(\langle a, b\rangle)= \delta^f_{n-r}(\langle b, a\rangle)$ hence Theorem \ref{T3}.
 \vskip .1in
 
 {\bf STEP 1:}

Recall that if $\alpha'\colon A'\to B$ and $\alpha''\colon A''\to B$ are linear maps, one writes $\alpha'+\alpha''$ for the linear map  $\alpha'+\alpha''\colon A'\oplus A''\to B$ defined by
$$(\alpha'+\alpha'')(a',a''):= \alpha(a') + \alpha''(a''),
$$ 
and if $\beta'\colon A\to B'$ and $\beta''\colon A\to B''$ are linear maps, one writes $(\beta', \beta'')$ for the linear map 
$
(\beta', \beta'')\colon A\to B'\oplus B''$ defined by $$(\beta', \beta'')(a):=\bigl(\beta'(a),\beta''(a)\bigr).
$$

One has the canonical isomorphisms 
\begin{equation}\label {E58}
\begin{aligned}
\ker (\beta', \beta'')^\ast \simeq  \coker ((\beta')^\ast +(\beta'')^\ast)\\
\ker (\alpha'+\alpha'')^\ast\simeq \coker ((\alpha')^\ast , (\alpha'')^\ast).
\end{aligned}
\end{equation}

The exact sequences in Borel--Moore homology of the pairs $(\tilde M, \tilde M_a)$ and $(\tilde M, \tilde M^b)$, 
which are the top horizontal rows of the two diagrams \eqref{PD1} and \eqref{PD1'}, imply  
\begin{equation}\label {E59} 
F_r^{\BM,\tilde f}(a,b)
=\dim\bigl(\mathbb I^{\BM,\tilde f}_a(r) \cap \mathbb I^{\BM,b}_{\tilde f}(r)\bigr)
=\dim\ker\bigl(j^{\BM}_a(r), j^{\BM,b}(r)\bigr).
\end{equation} 
Looking to the right side corners of the diagrams \eqref{PD1} and \eqref{PD1'} one concludes
\begin{equation}\label {EQ2} 
\ker\bigl(j^{\BM}_a(r), j^{\BM,b}(r)\bigr)\equiv \ker\bigl(r^a(n-r), r_b(n-r)\bigr).
\end{equation}
In view of the canonical isomorphism between cohomology vector space  and the  dual of homology vector space one obtains: 
\begin{equation}\label {EQ3} 
\ker\bigl(r^a(n-r),r_b(n-r)\bigr)\equiv\bigl(\coker\bigl(i^a(n - r) + i_b(n -r)\bigr)\bigr)^\ast.
\end{equation}

 Observe that:
\begin{enumerate}
\item 
In view of the finite dimensionality of $G^{\tilde f}(a,b)$ one has

$\dim G^{\tilde f}_{n-r}(b,a)= \dim G^{\tilde f}_{n-r}(b,a)^\ast.$
\item  
In view of the definition of $G^f_{n-r}$ one has 

$\dim G^{\tilde f}_{n-r}(b,a)^\ast= 
\dim\coker\bigl(i_b(n-r)+i^a(n-r)\bigr)^*.$ 
\item 
In view of (\ref{E58}) one has 

$\coker\bigl(i_b(n-r)+i^a(n-r)\bigr)^*
=\dim\bigl(\ker\bigl(i_b(n-r)^\ast+i^a(n-r)*\bigr)\bigr)^\ast$ 
\item  In view of (\ref {EQ2}) and (\ref{EQ3}) 
one has 

$\dim \coker\bigl(i_b(n-r)+i^a(n-r)\bigr)^* = 
\dim\bigl(\ker\bigl(j^{BM,b}(r), j^{BM}_a(r)).$

\item In view of (\ref{E59}) one has 

$\dim \bigl(\ker\bigl(j^{BM,b}(r), j^{BM}_a(r))= F^{BM,f}_r(a,b).$ 
\end{enumerate}

Consequently, $F^{\BM, \tilde f}_r(a,b)= G_{n-r}^{\tilde f}(b,a)$. 
\vskip .1in
{\bf STEP 2:}

We first provide below the description of the Borel--Moore homologies 
considered above in terms of subsets of 
$\tilde{\mathcal B}_r(f)\sqcup \tilde{\mathcal B}_{r-1}(f)\sqcup \tilde {\mathcal J}_r(f)\sqcup \tilde {\mathcal J}_{r-1}(f).$ This is a little more than we need but is useful for future references.

For $\alpha<\beta,$ we use the notations
\begin{equation}\label {E00}
\begin{aligned}
i_{\alpha,\beta}(r): &H_r(\tilde M_\alpha)\to H_r(\tilde M_\beta)\\
i^{\alpha,\beta}(r): &H_r(\tilde M^\beta)\to H_r(\tilde M^\alpha)\\
i_\alpha ^\beta (r): &H_r(\tilde M_\alpha\sqcup \tilde M^\beta)\to H_r(\tilde M)
\end{aligned}
\end{equation}
for the linear maps induced by the inclusions 
$\tilde M_\alpha \subseteq \tilde M_\beta,$
$\tilde M^\beta \subseteq \tilde M^\alpha$ and
$\tilde M_\alpha \cup \tilde M^\beta \subseteq\tilde M.$

We begin by considering the commutative diagram (\ref {ED64}) below whose rows are the long exact sequences of the pairs 
$(\tilde M_a, \tilde M_{-l})$, $(\tilde M,\tilde M_{-l}\sqcup\tilde M^l)$, $(\tilde M^b,\tilde M^l)$ for $-l<a$ and $b<l$ and vertical arrows 
induced by the inclusions of pairs 
$$
(\tilde M_a, \tilde M_{-l})\subset  (\tilde M,\tilde M_{-l}\sqcup\tilde M^l)\supset (\tilde M^b,\tilde M^l).
$$

\begin{equation}\label{ED64}
{\small\xymatrix{
\cdots \ar[r]& H_r(\tilde M_{-l})\ar[r]^{i_{-l,a}(r)}\ar[d]&H_r(\tilde M_{a}\ar[r])\ar[d]&H_r(\tilde M_{a},\tilde M_{-l})\ar[r]\ar[d]&H_{r-1}(\tilde M_{-l})\ar[d]\ar[r] &\cdots\\
\cdots \ar[r]& H_r(\tilde M_{-l}\sqcup \tilde M^l)\ar[r]^-{i_{-l}^l(r)}&H_r(\tilde M)\ar[r]&H_r(\tilde M, \tilde M_{-l}\sqcup  \tilde M^l)\ar[r]&H_{r-1}(\tilde M_{-l}\sqcup\tilde M^l)\ar[r]&\cdots\\
\cdots  \ar[r]
&H_r(\tilde M^{l})\ar[r]^{i^{b,l}(r)}\ar[u]&H_r(\tilde M^b)\ar[r]\ar[u]&H_r(\tilde M^b,\tilde M^{l})\ar[r]\ar[u]&H_{r-1}(\tilde M^{l})\ar[u]\ar[r]&\cdots\\
}}.
\end{equation}


The diagram \eqref {ED64} leads to the following commutative diagram whose rows are short exact sequences.
\begin{equation}\label {ED65}
\xymatrix{
0\ar[r]&\coker (i_{ -l, a}(r))\ar[r]\ar[d]&H_r(\tilde M_a, \tilde M_{-l})\ar[r]\ar[d]&\ker(i_{-l,  a}(r-1))\ar[r]\ar[d]&0\\
0\ar[r]&\coker (i_{-l}^l(r))\ar[r]&H_r(\tilde M,\tilde M_{-l}\sqcup\tilde M^l)\ar[r]&\ker(i_{-l}^l(r-1))\ar[r]&0\\
0\ar[r]&\coker (i^{b,l}(r))\ar[r]\ar[u]&H_r(\tilde M^b, \tilde M^l)\ar[r]\ar[u]&\ker(i^{b,l}(r-1))\ar[r]\ar[u]&0\\
}.
\end{equation}

Note that there exists compatible linear maps induced by inclusions when passing from the diagram corresponding to $(l', a', b')$ to the diagram corresponding to $(l, a,b)$ when $l'\geq l$, $a'\geq a$, $b'\leq b$.  
Note also that for $M$ compact and $f$ tame the set of bar codes $\mathcal B_r(f)$ is finite and therefore there is a maximal length of all bar codes say ${L(f)}$.

Proposition  \ref{P36} item (b) implies on the nose that following calculations.   

\begin{proposition}\label{P62}
Let $a,b$ fixed and suppose $l$ satisfies $a>-l, b<l$.  Then 
\begin{enumerate}[(a)]
\item $\coker(i_{ -l, a}(r))=\kappa[\mathcal M_{-l,a}(r)]$ with 
\begin{align*}
\mathcal M_{-l,a}(r)
&:=\{[\alpha,\beta]\in \mathcal B^c_r \mid -l <\alpha \leq a\}
\\&\qquad\cup\{(\alpha,\beta)\in \mathcal B^o_{r-1} \mid -l <\beta  \leq  a\}
\\&\qquad\cup\{[\alpha,\beta)\in \mathcal B^{co}_r \mid   -l <\alpha \leq a <  \beta \}.
\end{align*}
\item $\ker(i_{ -l,a}(r))=\kappa[\mathcal N_{-l,a}(r)]$ with 
\begin{equation*}
\mathcal N_{-l,a}(r):=\{[\alpha,\beta)\in \mathcal B^{co}_r \mid  \alpha \leq  -l <\beta \leq a\}.
\end{equation*}
\item $\coker(i^{b,l}(r))=\kappa[\mathcal M^{b,l}(r)]$ with 
\begin{align*}
\mathcal M^{b,l}(r)
&:=\{[\alpha,\beta]\in \mathcal B^c_r \mid  b \leq \beta <l  \} 
\\&\qquad\cup\{(\alpha,\beta)\in \mathcal B^o_{r-1} \mid b \leq \alpha <l\}
\\&\qquad\cup\{(\alpha,\beta]\in \mathcal B^{oc}_r \mid  \alpha< b \leq \beta< l\}.
\end{align*}
\item $\ker(i^{b,l}(r))=\kappa[\mathcal N^{b,l}(r)]$ with 
\begin{equation*}
\mathcal N^{b,l}(r):=\{(\alpha,\beta] \in \mathcal B^{oc}_r \mid  b \leq \alpha < l \leq \beta\}.
\end{equation*}
\end{enumerate}
If $2l>{L(f)}$ then: 
\begin{enumerate}[(a)]
\setcounter{enumi}{4}
\item $\coker(i_{-l}^l(r))=\kappa[\mathcal M_{-l}^l(r)]$ with 
\begin{align*}
\mathcal M_{-l}^l(r)
&:=\{[\alpha,\beta]\in \mathcal B^c_r \mid  [\alpha, \beta]\subset  (-l,l)\} 
\\&\qquad\cup\{(\alpha,\beta)\in \mathcal B^o_{r-1} \mid \alpha <l, \beta> -l\}.
\end{align*}
\item $\ker(i_{-l}^l(r))=\kappa[\mathcal N_{-l}^l(r)\sqcup \tilde J_r(f)]$ with\footnote{In view of the hypothesis $(a,b)$ can not contain both $-l$ and $l$.}
\begin{align*}
\mathcal N_{-l}^l(r)
&:=\{[\alpha,\beta)\in \mathcal B^{co}_r \mid  (\alpha ,\beta)\ni -l\}
\\&\qquad\cup\{(\alpha,\beta]\in \mathcal B^{oc}_r \mid  (\alpha ,\beta)\ni l\}.
\end{align*}
\end{enumerate}
\end{proposition}

Clearly  for $l'>l$ and  $l'-l >{L(f)}$ one has:
\begin{align*}
\mathcal M_{ -l',a}(r)&\supseteq \mathcal M_{-l,a}(r),
&\mathcal N_{-l',a}(r)&\cap \mathcal N_{-l,a}(r)= \emptyset,
\\
\mathcal M^{b,l'}(r) &\supseteq \mathcal M^{b,l}(r),
&\mathcal N^{b,l'}(r) &\cap \mathcal N^{b,l}(r)=\emptyset,
\\
\mathcal M_{-l'}^{l'}(r)&\supseteq \mathcal M_{-l}^l(r),
&\mathcal N_{-l'}^{l'}(r) &\cap \mathcal N_{-l}^l(r)=\emptyset.
\end{align*}

Note that the sets $\mathcal M^-_-(r)$, $\mathcal N^-_-(r)$, $\tilde{\mathcal J}_r$, and $\tilde {\mathcal J}_{r-1}$ are all subsets of 
$S=\tilde {\mathcal B}_r \sqcup \tilde {\mathcal B}_{r-1} \sqcup \tilde {\mathcal J}_r \sqcup \tilde{\mathcal J}_{r-1}.$
Note also that  all inclusions induced linear maps between the homologies involved in the diagrams (\ref{ED64}) and (\ref{ED65}),  via the identifications of these homologies to 
vector spaces generated by subsets of $S,$ correspond to canonical linear maps.  

Recall that if $S_1, S_2\subseteq S$ then the canonical linear map $\kappa[S_1]\to \kappa[S_2]$ is the unique linear extension of the map which the identity on $S_1\cap S_2$ and zero on $S_1\setminus S_2,$ cf.~Definition~\ref{D26}. 

To finalize the verification of Step 2 
for $\alpha, \beta \in \mathbb R$  one denotes
\begin{align*}
\mathcal B^c_{r,\alpha}  &:=\tilde {\mathcal B}^c_{r,(-\infty, \alpha]}(f)=\{I\in \mathcal B^c_r(\tilde f) \mid\textrm{$I\cap (-\infty,\alpha]\ne\emptyset$}\},
\\
\mathcal B^{c,\beta}_{r} &:=\tilde {\mathcal B}^c_{ r, [\beta, \infty)}(f)=\{I\in \mathcal B^c_r(\tilde f) \mid\textrm{$I\cap [\beta,\infty)\ne\emptyset$}\},
\\
\mathcal B^o_{r,\alpha}  &:=\tilde {\mathcal B}^o_{r, (-\infty,\alpha]}(f)=\{I\in \mathcal B^o_r (\tilde f)\mid\textrm{$I\subset (-\infty,\alpha]$} \},
\\
\mathcal B^{o,\beta}_{r} &:=\tilde {\mathcal B}^{o}_{r, [\beta,\infty)}(f)=\{I\in \mathcal B^o_r (\tilde f)\mid\textrm{$I\subset [\beta,\infty)$}\}.
\end{align*}
 and one considers the commutative diagram~\eqref{DD1} below. 
\begin{equation}\label {DD1}
{\small\xymatrix{
\mathbb I^{\tilde f}_\alpha(r)\ar[d]& \kappa[\mathcal B^c_{r,\alpha}\sqcup \mathcal B^o_{r-1,\alpha}\sqcup\tilde{\mathcal J}_r(f)]\ar[d]^{v_\alpha}\ar[l]_-{\tilde\omega_{r,(-\infty, \alpha]}}\ar[r]^{\pi_\alpha} &\kappa[[\mathcal B^c_{r,\alpha}\sqcup \mathcal B^o_{r-1,\alpha}]]\ar[d]^{v_\alpha} \ar[r]^-{\omega^{\BM}_{r,(-\infty,\alpha]}} & \mathbb I^{\BM,\tilde f}_\alpha(r)\ar[d]\\
H_r(\tilde M)& \kappa[\mathcal B^c_r\sqcup \mathcal B^o_{r-1}\sqcup\tilde{\mathcal J}_r(f)]\ar[l]_-{\tilde \omega_r}\ar[r]^{\pi} &\kappa[[\mathcal B^c_r\sqcup \mathcal B^o_{r-1}
]]\ar[r]^-{\tilde \omega^{\BM}_r}  & H_r^{\BM}(\tilde M)\\
\mathbb I_{\tilde f}^\beta(r)\ar[u]& \kappa[\mathcal B^{c,\beta}_r\sqcup \mathcal B^{o,\beta}_{r-1}\sqcup\tilde{\mathcal J}_r(f)]\ar[u]_{v^\beta}\ar[l]_-{\tilde\omega_{r, [\beta, \infty)}}\ar[r]^{\pi^\beta} &\kappa[[\mathcal B^{c,\beta}_r\sqcup \mathcal B^{o,\beta}_{r-1}]]\ar[u]^{\pi_\beta}\ar[r]^{\tilde\omega^{\BM}_{r,[\beta, \infty)}}  & \mathbb I_{\tilde f}^{\BM,\beta}(r)\ar[u]\\
}}
\end{equation}

Propositions~\ref{P36} (b) 
provides   the left  side of the diagram 
with $\tilde\omega_{r,(-\infty, \alpha]}$, $\tilde\omega_{r}$, $\tilde\omega_{r, [\beta, \infty)}$ isomorphisms. 

By passing to the inverse limit when $l\to \infty,$ Proposition (\ref{P62}) provides the right side of the diagram 
with $\tilde\omega^\BM_{r,(-\infty, \alpha]}$, $\tilde\omega^\BM_{r}$, and $\tilde\omega^\BM_{r,[\beta, \infty)}$ isomorphisms.  

The canonical linear maps $\pi_\alpha$, $\pi$, $\pi^\beta$,$v_\alpha$ and $v^\beta$ defined by the sets in brackets when regarded as subsets in 
$\tilde {\mathcal B}_r(f)\sqcup \tilde {\mathcal B}_{r-1}(f) \sqcup \tilde {\mathcal J}_r(f) \sqcup  \tilde {\mathcal J}_{r-1}(f)$ (cf.\ Definition~\ref{D26})  provide the middle of the diagram. 

%
%
%

%

Diagram~\eqref{DD1} implies that  $\mathbb I^{\tilde f}_\alpha (r)\cap \mathbb I^{\beta}_{\tilde f}(r)$  identifies to 
$$
\kappa [ (\tilde {\mathcal B}^c_{r,\alpha} \cap \tilde {\mathcal B}^{c,\beta}_r) \sqcup (\tilde {\mathcal B}^o_{r-1,\alpha} \cap \tilde {\mathcal B}^{o,\beta}_{r-1})\sqcup \tilde {\mathcal J} _r(f)]
$$ 
and 
$\mathbb I^{\BM, \tilde f}_\alpha (r)\cap \mathbb I^{\BM, \beta}_{\tilde f}(r)$ identifies to 
$$
\kappa [[ (\tilde {\mathcal B}^c_{r,\alpha} \cap \tilde {\mathcal B}^{c,\beta}_r) \sqcup (\tilde {\mathcal B}^o_{r-1,\alpha} \cap \tilde {\mathcal B}^{o,\beta}_{r-1})]].
$$ 
which in view of the fact that the set in brackets is finite (hence there is no difference between $\kappa[[\cdots]]$ and $\kappa[\cdots]$) imply the exactness of the sequence  

$$
0 \to \kappa [\tilde {\mathcal J}_r(f)] \to \mathbb I^{\tilde f}_\alpha (r)\cap \mathbb I^{\beta}_{\tilde f} (r) \to \mathbb I^{\BM, \tilde f}_\alpha (r)\cap \mathbb I^{\BM, \beta}_{\tilde f}(r)\to 0
.$$  
 Note that $\mathbb I^{\tilde f}_\alpha (r)\cap \mathbb I^{\beta}_{\tilde f} (r) \to \mathbb I^{\BM, \tilde f}_\alpha (r)\cap \mathbb I^{\BM, \beta}_{\tilde f}(r)$ 
is 
the composition $\tilde \omega^{\BM}_r \cdot \pi \cdot  \tilde \omega^{-1}_r.$  The exact sequence above implies 
$$F^{BM,\tilde f}_r(a,b)+ \sharp \tilde {\mathcal J}_r(f)= F^{\tilde f}_r(a,b)$$
which  finalizes Step 2.

 As a side remark note that by passing to the inverse limit when $l\to \infty$ Proposition \ref {P62} and Diagram (\ref {ED64}) led to the following calculations used in the next section. 
\begin{proposition}\label{P633}\
\begin{enumerate}

\item  $H^{\BM}_r(\tilde X) = \kappa[[\tilde S_r (f)\sqcup \tilde {\mathcal J}_{r-1}(f)]]= \kappa[[\tilde {\mathcal B}^c_r(f)\sqcup \tilde{\mathcal B} ^o_{r-1}(f)\sqcup \tilde {\mathcal J}_{r-1}(f)]]$.
\item  $\mathbb I^{\BM, \tilde f}_a(r)=\kappa[[\tilde S_{r, (-\infty, a]}(f)]]= \kappa [[\tilde {\mathcal B}^c_{r, (-\infty,a]}(f)\sqcup \tilde {\mathcal B }^o_{r-1, (-\infty,a]}(f)]]$.
\item  $H^{\BM}_r(\tilde X_a) = \kappa[[\tilde S '_{r, (-\infty, a]}(f)]]= \kappa[[\tilde {\mathcal B}^{co}_{r, (-\infty, a]}\sqcup \tilde S_{r, (-\infty, a]}(f)]]$. 
\item  $\mathbb I^{\BM, b}_{\tilde f}(r)=\kappa[[\tilde S_{r, [b, \infty)}(f)]]= \kappa [[\tilde {\mathcal B}^c_{r, ([b,\infty)}(f)\sqcup\tilde {\mathcal B }^o_{r-1, [b,\infty)}(f)]]$.
\item  $H^{\BM}_r(\tilde X^b) = \kappa[[\tilde S'_{r, (b,\infty)}(f)]]= \kappa[[ \tilde {\mathcal B}^{oc}_{r, [b,\infty)}\sqcup \tilde S_{r, (-\infty, a]}(f)]]$. 
\end{enumerate} 
\end{proposition}
 
 The canonical linear map $H_r(\tilde X) \to H_r^{BM(}(\tilde X)$ 
 can be also read off from Diagram~\eqref{DD1} in terms of barcodes. It is exactly the unique linear map which is identity on $\mathcal B^c_r\sqcup \mathcal B_{r-1}^o$ and zero on $\widetilde {\mathcal J}_r(f).$ 
 \vskip .1in
  
{\bf STEP 3:} 

In view of the fact that by $\delta^{\tilde f}_r(a,b)= \mu^{\tilde f}_r((a-\epsilon, a+\epsilon]\times [b-\epsilon, b+\epsilon)$  for $\epsilon$ small enough, cf. (\ref {EQQ2}), and 
of Proposition (\ref{P53}) 
the equality in  Step 3 follows.

\section{The mixed bar codes. Proof of Theorem~\ref{T5}}\label{SS7}

Let $f:X\to \mathbb S^1$ be a tame angle-valued map and $\tilde f:\tilde X\to \mathbb R$ be the infinite cyclic covering of $f.$The sets 
$\mathcal B_r^{co}(f)$ and $\mathcal B^{oc}_r(f)$ which define the configuration $C^m_r(f),$  cf. Section \ref{SS3}, are equivalence classes modulo translation by integer multiple of $2\pi$ of the bar codes of $\tilde {\mathcal B}_r^{co}(f)= \mathcal B_r^{co}(\tilde f)$ and $\tilde{\mathcal B}^{oc}_r(f)= \mathcal B^{oc}_r(\tilde f).$

Note also that an interval $[a,b)$ is an element in $\mathcal B^{co}_r(\tilde f)$ with multiplicity $k$ iff $(a,b)$ is an $r-$persistence bar code  in the sense of \cite {ELZ02}  with the same multiplicity for the sublevel filtration associated with $\tilde f.$ Similarly  the interval $(a,b]$ is an element in $\mathcal B^{oc}_r(\tilde f)$ with multiplicity $k$ if $(-b, -a)$ is an $r-$persistence bar code   with the same multiplicity for the sublevel-filtration associated with $-\tilde f.$ 
Recall that a point $\langle a, b\rangle\in (\mathbb R^2\setminus \Delta )/\mathbb Z$ is in the support of $C^m_r(f)$ with multiplicity $k$  iff $[a,b)\in \mathcal B^{co}_r(f)$ with multiplicity $k$ when  $a>b,$  or $(a, b]\in \mathcal B^{oc}_r(f)$ when $a<b$ with multiplicity $k.$

The stability result in \cite{CEH07} , the MAIN THEOREM, remains valid with the same proof  if instead of real-valued tame maps one considers angle-valued tame maps $f:X\to \mathbb S^1.$ 

More precisely the assignment $ f\rightsquigarrow \mathcal B_r^{co}(f),$ with  $\mathcal B_r^{co}(f)$ viewed as a multiset (= configuration of points) in $\mathbb R^2_+/ \mathbb Z\subset (\mathbb R^2\setminus \Delta) / \mathbb Z= \mathbb T\setminus \Delta_{\mathbb T}$  is continuous w.r. to the compact open topology on 
$C_{\xi,t}(X,\mathbb S^1)$ and the topology induced by the bottleneck metric on the space of configurations $\rm {Confg} \ ((\mathbb R^2_+ / \mathbb Z)\subset \rm {Confg} \ ((\mathbb R^2\setminus \Delta) / \mathbb Z).$

Similarly the assignment $f\rightsquigarrow \mathcal B_r^{oc}(f),$ with $\mathcal B_r^{oc}(f)$ viewed as a multi set in $\mathbb R^2_-  / \mathbb Z\subset (\mathbb R^2 \setminus \Delta)  / \mathbb Z,$  is continuous w.r. to the same topologies.
The spaces  $\mathbb R^2_+:= \{(x,y)\mid x<y\}$ and  $\mathbb R^2_-:= \{(x,y)\mid x<y\}$ are invariant w.r. to the action $\mathbb Z\times \mathbb R^2\to \mathbb R^2$ given by  $(n, (x,y)) \rightarrow (x+2n\pi, y +2n\pi).$ 

The bottleneck metric  defined in \cite{CEH07} induces the bottleneck topology defined without metric in Section \ref{S1}.
 In terms of the configuration $C^m_r(f)$ the stability result  in \cite{CEH07} can be reformulated in the following way.  
 
 {\bf Theorem} 
{\it The assignment $f\mapsto C^m_r(f)$  from $C_{\xi, t}(X,\mathbb S^1)$ the space of tame maps equipped with the compact open topology 
to $\mathcal\Confg(\mathbb T\setminus \Delta)$ 
equipped with the bottleneck topology for case $(X,K)= (\mathbb T, \Delta_{\mathbb T})$ as described in Section \ref{S1} is continuous. 
\footnote { A direct verification of Theorem \ref{T5} without reference to \cite{ELZ02} can be found in the forthcoming book \cite{BUR}.}} 
 
This because the bottleneck metric induces the bottleneck topology. 
\noindent This Theorem is exactly Theorem \ref{T5}.
\vskip .1in

In parallelism with the configuration $C_r(f),$ for the proof of Theorem \ref{T6} one can identify the configuration $C^m_r(f)$ with  the map $\delta^{m, f}_r: \mathbb T\setminus \Delta_{\mathbb T} \to \mathbb Z_{\geq 0}.$ This map can be derived from the function of two variables  $T^{\tilde f}_r\colon\mathbb R^2\setminus \Delta \to \mathbb Z_{\geq 0}$ 
in a similar manner to the description of the configuration $C_r(f)$ in Section~\ref{SS5}. 
The function $T^{\tilde f}_r$ is defined by: 
\begin{equation}\label{E61}
T^{\tilde f}_r(a,b):= 
\begin{cases}
\dim \ker\bigl(H_r(\tilde X_a)\to H_r(\tilde X_b)\bigr) & \textrm{if $a<b$}\\
\dim \ker\bigl(H_r(\tilde X^a)\to H_r(\tilde X^b)\bigr) & \textrm{if $a>b$.}
\end{cases}
\end{equation}
If $f$ is tame then so is $\tilde f$ and the limit 
\begin{multline}\label{EEE62}
\delta^{m,\tilde f}_r(a,b) = \lim_{\epsilon\to 0} \Bigl( T^{\tilde f}_r(a+\epsilon, b+\epsilon ) -T^{\tilde f}_r(a-\epsilon, b+\epsilon) 
\\+T^{\tilde f}_r(a-\epsilon, b-\epsilon) - T^{\tilde f}_r(a+\epsilon, b-\epsilon)\Bigr) 
\end{multline}
exists and defines the function $\delta^{m,\tilde f}_r$ which satisfies $$\delta^{m,\tilde f}_r(a, b)= \delta^{m,\tilde f}_r(a +2\pi, b+2\pi).$$ Then, as in Section~\ref{SS5}, one defines 
the function $\delta^{m, f}_r\colon\mathbb T\setminus \Delta_T \to \mathbb Z_{\geq 0}$ by 
$$\delta^{m, f}_r(\langle a, b\rangle):= \delta^{m, \tilde f}_r(a,b).$$ 

As in Section \ref{SS5} Proposition \ref{O78}, using Proposition \ref{P36} one verifies   that $\delta^{m, f}_r$ and $C^m_r(f)$ are equal. 

\begin{proposition}\label{P71'}
If $f$ is a tame real- or angle-valued map defined on $X$, a compact ANR, then $\delta^{m,f}_r$ and $C^m_r(f)$ are equal $\mathbb Z_{\geq 0}-$valued functions defined on $\mathbb R^2\setminus \Delta $ or $\mathbb T\setminus \Delta_{\mathbb T}$. 
\end{proposition}

The proof is exactly the same as of Proposition. \ref{O78} provided one replaces (\ref{EQ46}) and (\ref{EQ47}) by (\ref{EQQQ})
which follows in a straightforward manner from  Proposition  \ref{P36}. precisely  the following calculations. 
\begin{equation}\label {EQQQ}
\begin{aligned}
T^f_r(c_i, c_j)= \sharp \{I\in \mathcal B_r^{c,o} \mid I= [c_i,c_j)\} \rm {if} \ c_i< c_j\\
T^f_r(c_i, c_j)= \sharp \{I\in \mathcal B_r^{c,o} \mid I= [c_i,c_j)\} \rm {if} \ c_i> c_j .\\
\end{aligned}
\end{equation}

In order to prove  Theorem~\ref{T6}
we first recall some notations.

\noindent For $\tilde f\colon\tilde X\to \mathbb R$ the infinite cyclic covering of the tame map $f\colon X\to \mathbb S^1,$  $a, b\in \mathbb R, a\leq b$
\begin{itemize}
\item one considers the inclusion induced linear maps :
\begin{align*} 
i_a(r)\colon H_r(\tilde X_a)&\to H_r(\tilde X),&i^{\BM}_a(r)\colon H^{\BM}_r(\tilde X_a)&\to H^{\BM}_r(\tilde X),
\\ 
i^a(r)\colon H_r(\tilde X^a)&\to H_r(\tilde X),&i^{\BM,a}(r)\colon H^{\BM}_r(\tilde X^a)&\to H^{\BM}_r(\tilde X).
\end{align*}
\begin{align*}
i_{a,b}(r)\colon H_r(\tilde X_a)&\to H_r(\tilde X_b),&i^{\BM}_{a,b}(r)\colon H^{\BM}_r(\tilde X_a)&\to H^{\BM}_r(\tilde X_b),
\\
i^{b,a}(r)\colon H_r(\tilde X^b)&\to H_r(\tilde X^a),&i^{\BM,b,a}(r)\colon H^{\BM}_r(\tilde X^b)&\to H^{\BM}_r(\tilde X^a).
\end{align*}
\item one  defines
\begin{align*}
\setcounter{enumi}{4}
\mathbb K_a(r)&:=\ker i_a(r),&\mathbb K^{\BM}_a(r)&:=\ker i^{\BM}_a(r),
\\
\mathbb K^a(r)&:=\ker i^a(r),&\mathbb K^{\BM,a}(r)&:=\ker i^{\BM,a}(r),
\end{align*}
and 
\item one denotes   by 
\begin{align*}
\tilde\iota_{a,b}(r)\colon \mathbb K_a(r)&\to \mathbb K_b(r),&\tilde\iota^{\BM}_{a,b}(r)\colon\mathbb K^{\BM}_a(r)&\to \mathbb K^{BM}_b(r),
\\
\tilde\iota^{b,a}(r)\colon \mathbb K^b(r)&\to \mathbb K^a(r),& \tilde\iota^{\BM, b,a}(r):\mathbb K^{\BM,b}(r)&\to \mathbb K^{BM,a}(r)
\end{align*}
the restrictions of $i_{a,b}(r)$, $i^{\BM}_{a,b}(r)$ and $i^{b,a}(r)$, $i^{\BM, b,a}(r)$ to these kernels. 
\end{itemize}

Note that in view of the definition (\ref {E61}) 
\begin{equation}\label{EQQQQ1}
\begin{aligned}
T^f_r(a,b) =&\ker i{a,b}(r), a<b 
\\
T^f_r (a,b)= &\ker i^{a,b}(r), a>b
\end{aligned}
\end{equation}
by elementary linear algebra argument 
\begin{equation}\label{EQQQQ2}
\begin{aligned}
\ker \tilde \iota_{a,b}(r)=&\ker i_{a,b}(r)
\\
\ker \tilde \iota^{a,b}(r)=&\ker i^{a,b}(r)
\end{aligned} 
\end{equation}
and 
In view of the calculations of the Borel--Moore homology of $\tilde X^a$, $\tilde X_a$, $\tilde X$ cf Proposition \ref{P633}, and the description  of the linear maps from  homology to Borel--Moore homology 
one concludes that 

\begin{equation}\label{EQQ3}
\begin{aligned}
\mathbb K(r)&=\mathbb K^{\BM}(r)
\\
\tilde\iota(r)&=\tilde\iota^{\BM}(r)
\end{aligned} 
\end{equation}

%

Proposition~\ref{P36} permits to describe the vector spaces $\mathbb K_a(r)$, $\mathbb K^a(r)$, 
$\ker(\tilde\iota_{a,b}(r))$, $\coker(\tilde\iota_{a,b}(r))$, $\ker(\tilde\iota^{b,a}(r))$, $\coker(\tilde\iota^{b,a}(r))$ 
in terms of mixed bar codes as in Proposition \ref {P93} below. The verification is a straightforward reading of Proposition~\ref{P36}.

\begin{proposition}\label{P93}
Suppose $f\colon X\to \mathbb S^1$ is a tame map with $\tilde f\colon\tilde X\to \mathbb R$ its infinite cyclic covering,
and $a$, $b$ real numbers with $a\leq b$. Then:
\begin{align*}
\mathbb K_a(r)      &= \kappa\bigl[\{I\in \tilde{\mathcal B}^{co}_r(f) \mid I \ni a\}\bigr],
\\
\mathbb K^a(r)      &= \kappa\bigl[\{I\in \tilde{\mathcal B}^{oc}_r (f)\mid I \ni a\}\bigr],
\\
\ker\bigl(\tilde\iota_{a,b}(r)\bigr)   &= \kappa\bigl[\{I\in \tilde{\mathcal B}^{co}_r(f) \mid I \ni a, \ b\notin I\}\bigr],
\\
\coker\bigl(\tilde\iota_{a,b}(r)\bigr) &= \kappa\bigl[\{I\in \tilde{\mathcal B}^{co}_r(f) \mid I \ni b,\  a\notin I \}\bigr],
\\
\ker\bigl(\tilde\iota^{b,a}(r)\bigr)   &= \kappa\bigl[\{I\in \tilde{\mathcal B}^{oc}_r(f) \mid I \ni b,\  a\notin I\}\bigr],
\\
\coker\bigl(\tilde\iota^{b,a}(r)\bigr) &= \kappa\bigl[\{I\in \tilde{\mathcal B}^{oc}_r(f) \mid I \ni a,\  b\notin I \}\bigr].
\end{align*}
\end{proposition}
Note that $\mathbb K_a(r)$ and $\mathbb K^a(r)$ are finite dimensional vector spaces.
\vskip .1in

%
In view of the  tameness of $f$ and  of (\ref {EEE62}), (\ref{EQQQQ1}) and (\ref {EQQQQ2}) one concludes that for $a<b$ and $\epsilon$ small enough

\begin{equation}\label {EE63}
\begin{aligned}
\delta_k^{m,\tilde f}(a,b)=&
\dim\ker\bigl( i_{a+\epsilon,b+\epsilon}(k)\bigr)-\dim\ker\bigl(i_{a-\epsilon,b+\epsilon}(k)\bigr)\\
&\quad -\dim\ker\bigl(i_{a+\epsilon,b-\epsilon}(k)\bigr) +\dim\ker\bigl(i_{a-\epsilon,b-\epsilon}(k)\bigr)\\
=&\dim\ker\bigl(\tilde\iota_{a+\epsilon,b+\epsilon}(k)\bigr)-\dim\ker\bigl(\tilde\iota_{a-\epsilon,b+\epsilon}(k)\bigr)\\
&\quad -\dim\ker\bigl(\tilde\iota_{a+\epsilon,b-\epsilon}(k)\bigr) +\dim\ker\bigl(\tilde\iota_{a-\epsilon,b-\epsilon}(k)\bigr)\\
\end{aligned}
\end{equation}
and for $a>b$ and $\epsilon $ small enough one has 
\begin{equation}\label{EE64}
\begin{aligned}
\delta_k^{m,\tilde f}(a,b)=&
\dim\ker\bigl(i^{a-\epsilon,b-\epsilon}(r)\bigr)-\dim\ker\bigl(i^{a+\epsilon,b-\epsilon}(r)\bigr)\\
&\quad -\dim\ker\bigl(i^{a-\epsilon,b+\epsilon}(r)\bigr) +\dim\ker\bigl(i^{a+\epsilon,b+\epsilon}(r)\bigr)\\
=&\dim\ker\bigl(\tilde\iota^{a-\epsilon,b-\epsilon}(r)\bigr)-\dim\ker\bigl(\tilde\iota^{a+\epsilon,b-\epsilon}(r)\bigr)\\
&\quad-\dim\ker\bigl(\tilde\iota^{a-\epsilon,b+\epsilon}(r)\bigr) +\dim\ker\bigl(\tilde\iota^{a+\epsilon,b+\epsilon}(r)\bigr).
\end{aligned}
\end{equation}

Next observe that the long exact sequence for the pair $(\tilde X,\tilde X_\alpha),$  $\alpha\in \mathbb R$
\begin{multline*}
\cdots\to 
H_{n-r}(\tilde X)\xrightarrow{j^\alpha(n-r)} 
H_{n-r}(\tilde X,\tilde X^\alpha)\xrightarrow{\partial^\alpha(n-r)}
\\\to
H_{n-r-1}(\tilde X^\alpha)\xrightarrow{i^\alpha(n-r-1)}
H_{n-1-r}(\tilde X)\to\cdots
\end{multline*}
provides the   isomorphism
\begin{equation}
\hat{\partial}^\alpha(n-r)\colon\coker\bigl(j^\alpha(n-r)\bigr)\to\ker\bigl(i^\alpha (n-r-1)\bigr)=\mathbb K^\alpha(n-r-1)
\end{equation}
which, being ``natural'' w.r.\ to the inclusion of pairs $(\tilde X, \tilde X^\beta)\subseteq (\tilde X, \tilde X^\alpha)$ for $\alpha\leq \beta$, 
makes the diagram below commutative.
\begin{equation}\label{E60}
\vcenter{
\xymatrix{
\coker\bigl(j^\beta(n-r)\bigr)\ar[d]\ar[rr]^{\hat{\partial}^\beta(n-r)} && \mathbb K^\beta(n-r-1)\ar[d]^{\tilde\iota^{\beta,\alpha}(n-r-1)} \\
\coker\bigl(j^\alpha(n-r)\bigr)\ar[rr]^{\hat{\partial}^\alpha(n-r)}       && \mathbb K^\alpha(n-r-1)
}}
\end{equation}

{\bf Proof of Theorem \ref {T6}}

Suppose now that $X=M$ is a closed $\kappa$-orientable $n-$dimensional manifold and $\alpha$ is a regular value of $\tilde f$.
Poincar\'e duality for the manifold $\tilde M$ and for the pairs $(\tilde M,\tilde M_\alpha)$ and $(\tilde M,\tilde M^\alpha)$ provides the commutative diagram 
\begin{equation}\label{E61b}
\vcenter{
\xymatrix{
\mathbb K_\alpha(r)\ar@{=}[d]\ar[r]&H_r(\tilde M_\alpha)\ar[d]\ar[r]^{i_\alpha(r)} &H_r(\tilde M)\ar[d]\\
\mathbb K^{\BM}_\alpha(r)\ar[d]^{{\rm PD_1}}\ar[r]&H^{\BM}_r(\tilde M_\alpha)\ar[d]^{{\rm PD_2}}\ar[r]^{i^{\BM}_\alpha(r)} &H^{\BM}_r(\tilde M)\ar[d]^{{\rm PD_3}}\\
(\coker  j^\alpha(n-r))^\ast \ar[r]&(H_{n-r}(\tilde M, \tilde M^\alpha))^\ast \ar[r]&(H_{n-r}(\tilde M))^\ast
}}
\end{equation}
with the bottom vertical arrows $PD_1, PD_2, PD_3$ 
isomorphisms. This is because $PD_2$ and $PD_3$  which appear in \eqref{PD1} are isomorphisms as indicated in Section \ref{SS5}. 
The diagram is natural w.r.\ to the inclusion of pairs $(X, X_\alpha)\subseteq (X, X_\beta)$, provided $\alpha$ and $\beta$ are regular values.
It leads to the following commutative diagram whose vertical arrows are all isomorphisms:
\begin{equation} \label{E62}
\vcenter{
\xymatrix{
\mathbb K_\alpha(r)           \ar[d]\ar[rr]^{\tilde\iota_{a,b}(r)}          && \mathbb K_\beta(r)          \ar[d] \\
(\coker j^\alpha(n-r))^\ast     \ar[rr]                              && (\coker j^b(n-r))^\ast      \\
\mathbb K^\alpha(n-r-1)^\ast  \ar[u]_{\hat{\partial}^a(n-r)^\ast}
\ar[rr]^{\tilde\iota^{\beta,\alpha}(n-r-1)^\ast} && \mathbb K^\beta(n-r-1)^\ast \ar[u]_{\hat{\partial}^\beta(n-r)^\ast}
}}
\end{equation}

To finalize the proof of Theorem~\ref{T6}, recall that the tameness of $f$ implies the tameness of $\tilde f$ and for $a$, $b$ critical values and $\epsilon<\epsilon(f)$,
the numbers $a\pm \epsilon, b\pm \epsilon$ are regular values, therefore by \eqref{E62} one has 
\begin{equation}\label {E65}
\tilde\iota_{a\pm\epsilon, b\pm\epsilon'}(r)=\bigl(\tilde\iota^{b\pm\epsilon, a\pm\epsilon'}(n-1-r)\bigr)^\ast.
\end{equation}

The equations \eqref{EE63}, \eqref{EE64}, and \eqref{E65} imply $\delta_r^{m,\tilde f}(a,b)=\delta_{n-r-1}^{m,\tilde f}(b,a)$ and then  
\newline $C^m_r(f)(\langle a,b\rangle)= \delta_r^{m,f}(\langle a,b\rangle)= \delta_{n-r-1}^{m,f}(\langle b, a\rangle)= C^m_{n-r-1}(f)(\langle b, a\rangle).$ equality which establishes  Theorem~\ref{T6}. 

\section{Linear relations and monodromy. Proof of Theorem~\ref{T4}}\label{SS2}

We begin this section with a discussion of linear relations.
To every linear relation $R\colon V\leadsto V$ we associate a linear relation $R_\reg\colon V_\reg\leadsto V_\reg$
on a subquotient, $V_\reg$, of $V$. In Proposition~\ref{P:AA} we show that $R_\reg$ is a linear  isomorphism.
If $V$ is a finite dimensional vector space, then, according to the Krull--Remak--Schmidt theorem, $R$
can be decomposed as a direct sum of indecomposable linear relations, $R\cong R_1\oplus\cdots\oplus R_N$,
where the factors $R_i\colon V_i\leadsto V_i$ are unique up to permutation and isomorphism. 
In this case, $R_\reg$ is isomorphic to the direct sum of factors 
which are indecomposable linear isomorphisms see Proposition~\ref{P:C}.
For linear relations on complex vector spaces $R_\reg$ can easily be derived using the detailed structure theorem in \cite{SdSW05}. 
Here we will  be concerned with 
vector spaces over arbitrary fields.
Most of this material can be developed for linear relations on modules over commutative rings, and this is the setting for the basic definitions, although in this paper we are interested only in the case of vector spaces. 

In the second part of this section, we consider the level $X_\theta=f^{-1}(\theta)$ associated with a continuous map $f\colon X\to S^1$ and a  value $\theta\in S^1$ s.t. $f^{-1}(\theta)$ is an ANR.
Using the corresponding infinite cyclic covering $\tilde X\to X$ one obtains a linear relation $R_\theta$ on $H_*(X_\theta)$, see Section~\ref{S1} or \eqref{E:RSigma} below.
We will show that $(R_\theta)_\reg$ is conjugate to the isomorphism induced by the fundamental deck transformation on
$$
\ker\Bigl(H_*(\tilde X)\to H_*^{\Nov,-}(\tilde X)\oplus H_*^{\Nov,+}(\tilde X)\Bigr),
$$
see Theorem~\ref{T:monreg}. 
Here $H_*^{\Nov,\pm}(\tilde X)$ denote Novikov type homology groups explained below.
The second part of Theorem~\ref{T:monreg} implies  that this is isomorphic to $\ker (H_r(\tilde X)\to H_r^N(X, \xi_f))$ considered in Section \ref{S4}.

This result holds true without compactness assumptions on $X$ (and with arbitrary coefficients not necessary in a field).
It implies that $R_\reg$ is a homotopy invariant of $f$.

At the end of Section~\ref{S8} we will give a proof of Theorem~\ref{T4}.

\subsection{Linear relations}\label{SS:linrel}

Suppose $V$ and $W$ are two modules over a fixed commutative ring.
Recall that a linear relation from $V$ to $W$ can be considered as a submodule $R\subseteq V\times W$. 
Notationally, we indicate this situation by $R\colon V\leadsto W$. For $v\in V$ and $w\in W$ we write $vRw$ iff $v$ is in 
relation with $w$, i.e.\ $(v,w)\in R$. Every module homomorphism $V\to W$ can be regarded as a linear relation $V\leadsto W$ 
in a natural way. If $U$ is another module, and $S\colon W\leadsto U$ is a linear relation, then
the composition $SR\colon V\leadsto U$ is the linear relation defined by $v(SR)u$ iff there exists $w\in W$ 
such that $vRw$ and $wSu$. Clearly, this is an associative composition generalizing the ordinary composition of module
homomorphisms. For the identical relations we have $R\Id_V=R$ and $\Id_WR=R$. Modules over a fixed commutative ring and linear relations 
thus constitute a category. If $R\colon V\leadsto W$ is a linear relation we define a linear relation $R^\dag\colon W\leadsto V$ 
by $wR^\dag v$ iff $vRw$. Clearly, $R^{\dag\dag}=R$ and $(SR)^\dag=R^\dag S^\dag$.

A linear relation $R\colon V\leadsto W$ gives rise to the following submodules:
\begin{align*}
\dom(R)&:=\{v\in V\mid\exists w\in W:vRw\}
\\
\img(R)&:=\{w\in W\mid\exists v\in V:vRw\}
\\
\ker(R)&:=\{v\in V\mid vR0\}
\\
\mul(R)&:=\{w\in W\mid 0Rw\}
\end{align*}
Clearly, $\ker(R)\subseteq\dom(R)\subseteq V$, and $W\supseteq\img(R)\supseteq\mul(R)$.
Note that $R$ is a homomorphism (map) iff $\dom(R)=V$ and $\mul(R)=0$. One readily verifies:

\begin{lemma}\label{L:1}
For a linear relation $R\colon V\leadsto W$ the following are equivalent:
\begin{enumerate}[(a)]
\item
$R$ is an isomorphism in the category of modules and linear relations.
\item
$\dom(R)=V$, $\img(R)=W$, $\ker(R)=0$, and $\mul(R)=0$.
\item
$R$ is an isomorphism of modules.
\end{enumerate}
In this case $R^{-1}=R^\dag$.
\end{lemma}

For a linear relation $R\colon V\leadsto V$, we introduce the following submodules:
\begin{align*}
K_+&:=\{v\in V\mid\exists k\,\exists v_i\in V:vRv_1Rv_2R\cdots Rv_kR0\}
\\
K_-&:=\{v\in V\mid\exists k\,\exists v_i\in V:0Rv_{-k}R\cdots Rv_{-2}Rv_{-1}Rv\}
\\
D_+&:=\{v\in V\mid\exists v_i\in V:vRv_1Rv_2Rv_3R\cdots\}
\\
D_-&:=\{v\in V\mid\exists v_i\in V:\cdots Rv_{-3}Rv_{-2}Rv_{-1}Rv\}
\\
D:=D_-\cap D_+&=\{v\in V\mid\exists v_i\in V:\cdots Rv_{-2}Rv_{-1}RvRv_1Rv_2R\cdots\},
\end{align*}
Clearly, $K_-\subseteq D_-\subseteq V\supseteq D_+\supseteq K_+$.
Also note that passing from $R$ to $R^\dag$, the roles of $+$ and $-$ get interchanged.
Moreover, we introduce a linear relation on the quotient module
\begin{equation}\label{E:Vreg} 
V_\reg:=\frac{D}{(K_-+K_+)\cap D}
\end{equation}
defined as the composition 
$$
V_\reg=\frac{D}{(K_-+K_+)\cap D}\overset{\pi^\dag}\leadsto D\overset\iota\leadsto V\overset R\leadsto V\overset{\iota^\dag}\leadsto D\overset\pi\leadsto\frac{D}{(K_-+K_+)\cap D}=V_\reg,
$$
where $\iota$ and $\pi$ denote the canonical inclusion and projection, respectively.
In other words, two elements in $V_\reg$ are related by $R_\reg$ iff they admit representatives in $D$
which are related by $R$. We refer to $R_\reg$ as the \emph{regular part} of $R$.

\begin{proposition}\label{P:AA}
The relation $R_\reg\colon V_\reg\leadsto V_\reg$ is an isomorphism of modules. Moreover,
the natural inclusion induces a canonical isomorphism
\begin{equation}\label{E:100}
V_\reg=\frac{D}{(K_-+K_+)\cap D}\xrightarrow\cong\frac{(K_-+D_+)\cap(D_-+K_+)}{K_-+K_+}
\end{equation}
which intertwines $R_\reg$ with the relation induced on the right hand side quotient.
\end{proposition}

\begin{proof}
Clearly, \eqref{E:100} is well defined and injective. To see that it is onto let
$$
x=k_-+d_+=d_-+k_+\in(K_-+D_+)\cap(D_-+K_+),
$$
where $k_\pm\in K_\pm$ and $d_\pm\in D_\pm$. Thus
$$
x-k_--k_+=d_+-k_+=d_--k_-\in D_+ \cap D_-=D.
$$
We conclude $x\in D+K_-+K_+$, whence \eqref{E:100} is onto. We will next show that this isomorphism intertwines
$R_\reg$ with the relation induced on the right hand side. To do so, suppose $xR\tilde x$ where
\begin{align*}
x&=k_-+d_+=d_-+k_+\in(K_-+D_+)\cap(D_-+K_+),
\\
\tilde x&=\tilde k_-+\tilde d_+=\tilde d_-+\tilde k_+\in(K_-+D_+)\cap(D_-+K_+),
\end{align*}
and $k_\pm,\tilde k_\pm\in K_\pm$ and $d_\pm,\tilde d_\pm\in D_\pm$. Note that there exist $k_+'\in K_+$ and
$\tilde k_-'\in K_-$ such that $k_+Rk_+'$ and $\tilde k_-'R\tilde k_-$. By linearity of $R$ we obtain
$$
\underbrace{(x-k_+-\tilde k_-')}_{\in D_-}R\underbrace{(\tilde x-k_+'-\tilde k_-)}_{\in D_+}.
$$
We conclude $d:=x-k_+-\tilde k_-'\in D$, $\tilde d:=\tilde x-k_+'-\tilde k_-\in D$, and $dR\tilde d$.
This shows that the relations induced on the two quotients in \eqref{E:100} coincide.
We complete the proof by showing that $R_\reg$ is an isomorphism.
Clearly, $\dom(R_\reg)=V_\reg=\img(R_\reg)$. We will next show $\ker(R_\reg)=0$. To this end suppose
$dR\tilde d$, where
$$
d\in D\quad\text{and}\quad\tilde d=\tilde k_-+\tilde k_+\in(K_-+K_+)\cap D
$$
with $\tilde k_\pm\in K_\pm$. Note that $\tilde k_-=\tilde d-\tilde k_+\in K_-\cap D_+$. Thus there exists
$k_-\in K_-\cap D_+$ such that $k_-R\tilde k_-$. By linearity of $R$, we get $(d-k_-)R\tilde k_+$, whence 
$d-k_-\in K_+$ and thus $d\in K_-+K_+$.
This shows $\ker(R_\reg)=0$. Analogously, we have $\mul(R_\reg)=0$. In view of Lemma~\ref{L:1} we conclude that
$R_\reg$ is an isomorphism of modules. 
\qed
\end{proof}

We will now specialize to linear relations on finite dimensional vector spaces and provide
another description of $V_\reg$ in this case. Consider the category whose objects are finite 
dimensional vector spaces $V$ equipped with a linear relation $R\colon V\leadsto V$ and whose 
morphisms are linear maps $\psi\colon V\to W$ such that for all $x,y\in V$ with $xRy$ we also have 
$\psi(x)Q\psi(y)$, where $W$ is another finite dimensional vector space with linear relation $Q\colon W\leadsto W$.
It is readily checked that this is an abelian category. By the Krull--Remak--Schmidt theorem, every linear
relation on a finite dimensional vector space can therefore be decomposed into a direct sum of
indecomposable ones, $R\cong R_1\oplus\cdots\oplus R_N$, where the factors are unique up to
permutation and isomorphism. The decomposition itself, however, is not canonical.

\begin{proposition}\label{P:C}
Let $R\colon V\leadsto V$ be a linear relation on a finite dimensional vector space over a field , and let
$R\cong R_1\oplus\cdots\oplus R_N$ denote a decomposition into indecomposable linear relations.
Then $R_\reg$ is isomorphic to the direct sum of factors $R_i$ whose relations are linear isomorphisms.
\end{proposition}

\begin{proof}
Since the definition of $R_\reg$ is a natural one, we clearly have
$$
R_\reg\cong(R_1)_\reg\oplus\cdots\oplus(R_N)_\reg.
$$
Consequently, it suffices to show the following two assertions:
\begin{enumerate}[(a)]
\item
If $R\colon V\leadsto V$ is an isomorphism of vector spaces, then $V_\reg=V$ and $R_\reg=R$.
\item
If $R\colon V\leadsto V$ is an indecomposable linear relation on a finite dimensional vector space which is not a linear isomorphism, then $V_\reg=0$.
\end{enumerate}
The first statement is obvious, in this case we have $K_-=K_+=0$ and $D=D_-=D_+=V$. 
To see the second assertion, note that an indecomposable linear relation $R\subseteq V\times V$ gives rise to an indecomposable representation $R\genfrac{}{}{0pt}{}{\to}{\to}V$ of the quiver $G_2$.
Since $R$ is not an isomorphism, the quiver representation has to be of the bar code type.
Using the explicit descriptions of the bar code representations, it is straight forward to conclude $V_\reg=0$.
\qed
\end{proof}

In the subsequent discussion we will also make use of the following result:

\begin{proposition}\label{P:X}
Suppose $R\colon V\leadsto V$ is a linear relation on a finite dimensional vector space. Then:
\begin{equation}\label{E:12}
D_+=D+K_+,\quad D_-=K_-+D,\quad\text{and}
\end{equation}
\begin{equation}\label{E:13}
K_-\cap D_+=K_-\cap K_+=D_-\cap K_+.
\end{equation}
\end{proposition}

For the proof we first establish two lemmas.

\begin{lemma}\label{L:2}
Suppose $R\colon V\leadsto W$ is a linear relation between vector spaces such that $\dim V=\dim W<\infty$. 
Then the following are equivalent:
\begin{enumerate}[(a)]
\item
$R$ is an isomorphism.
\item
$\dom(R)=V$ and $\ker(R)=0$.
\item
$\img(R)=W$ and $\mul(R)=0$.
\end{enumerate}
\end{lemma}

\begin{proof}
This follows immediately from the dimension formula
$$
\dim\dom(R)+\dim\mul(R)=\dim(R)=\dim\img(R)+\dim\ker(R)
$$
and Lemma~\ref{L:1}.
\qed
\end{proof}

\begin{lemma}\label{L:3}
If $V$ is finite dimensional, then the composition of relations
$$
D_+/K_+\overset{\pi^\dag}\leadsto D_+\overset\iota\leadsto V\overset{R^k}\leadsto V\overset{\iota^\dag}\leadsto D_+\overset\pi\leadsto D_+/K_+,
$$
is a linear isomorphism, for every $k\geq0$, where $\iota$ and $\pi$ denote the canonical inclusion and projection, respectively.
Analogously, the relation induced by $R^k$ on $D_-/K_-$ is an isomorphism, for all $k\geq0$. Moreover, for sufficiently large $k$,
$$
D_-=\img(R^k)\quad\text{and}\quad D_+=\dom(R^k).
$$
\end{lemma}

\begin{proof}
One readily verifies $\dom(\pi\iota^\dag R^k\iota\pi^\dag)=D_+/K_+$ and $\ker(\pi\iota^\dag R^k\iota\pi^\dag)=0$.
The first assertion thus follows from Lemma~\ref{L:2} above. Considering $R^\dag$ we obtain the second statement.
Clearly, $\dom(R^k)\supseteq\dom(R^{k+1})$, for all $k\geq0$.
Since $V$ is finite dimensional, we must have $\dom(R^k)=\dom(R^{k+1})$, for sufficiently large $k$.
Given $v\in\dom(R^k)$, we thus find $v_1\in\dom(R^k)$ such that $vRv_1$. Proceeding inductively, we construct $v_i\in\img(R^k)$ such that
$vRv_1Rv_2R\cdots$, whence $v\in D_+$. This shows $\dom(R^k)\subseteq D_+$, for sufficiently large $k$. As the converse inclusion is obvious
we get $D_+=\dom(R^k)$. Considering $R^\dag$, we obtain the last statement.
\qed
\end{proof}

\begin{proof}[Proof of Proposition~\ref{P:X}]
From Lemma~\ref{L:3} we get $\img(\pi\iota^\dag R^k)=D_+/K_+$, whence
$D_+\subseteq\img(R^k)+K_+$, for every $k\geq0$, and thus $D_+\subseteq D_-+K_+$. This implies
$D_+=D+K_+$. Considering $R^\dag$ we obtain the other equality in \eqref{E:12}.
From Lemma~\ref{L:3} we also get $\mul(\pi\iota^\dag R^k)=0$, whence
$\mul(R^k)\cap D_+\subseteq K_+$, for every $k\geq0$. This gives 
$K_-\cap D_+=K_-\cap K_+$. Considering $R^\dag$ we get the other equality in \eqref{E:13}.
\qed
\end{proof}

Let us describe the regular part of a linear transformation $\varphi\colon V\to V$ on a finite dimensional vector space $V$ more explicitly.
In this case, we clearly have $K_-=0$, $K_+=\bigcup_n\ker\varphi^n$, $D_+=V$ and $D=D_-=\bigcap_n\img\varphi^n$.
Moreover, $(K_-+K_+)\cap D=0$ according to \eqref{E:13} in Proposition~\ref{P:X}. 
Hence, the regular part of $\varphi$ coincides with the restriction $\varphi\colon\bigcap_n\img\varphi^n\to\bigcap_n\img\varphi^n$, see \eqref{E:Vreg}.
According to Proposition~\ref{P:AA}, the regular part of $\varphi$ can alternatively be described as the induced isomorphism 
$\varphi_\reg\colon V/\bigcup_n\ker\varphi^n\to V/\bigcup_n\ker\varphi^n$, for we have $V=D_-+K_+$ in view of \eqref{E:12} in Proposition~\ref{P:X}.

The following notation and observation will be used in the appendix. 
For two linear maps, $A,B\colon V\to W$, we let $R(A,B)\colon V\leadsto V$ denote the linear relation $R(A,B):= \{(v_1, v_2) \mid A(v_1)= B(v_2)\}$.

\begin{obs}\label{O87}
Suppose $A,B\colon V\to W$ are two linear maps.
\begin{enumerate}[(a)]
\item If $A',B'\colon V\to W'$ denote the composition of $A$ and $B$ with an inclusion of vector spaces, $W\subseteq W'$, then $R(A,B)=R(A',B')$.
\item If $A$ is invertible then $R(A,B)= R({\rm Id}, A^{-1}B)=R(A^{-1}B,{\rm Id})^\dag$.
\item If $A$ is invertible then $R(A,B)_\reg=((A^{-1}B)_\reg)^{-1}.$
\end{enumerate}
\end{obs}

\subsection{Monodromy}\label{S8}

Suppose $f\colon X\to S^1$ is a continuous map and let
$$
\xymatrix{
\tilde X\ar[d]\ar[r]^-{\tilde f}&\R\ar[d]\\X\ar[r]^-f&S^1
}
$$
denote the associated infinite cyclic covering. For $r\in\R$ we put $\tilde X_r=\tilde f^{-1}(r)$ and let $H_*(\tilde X_r)$ denote
its singular homology with coefficients in any fixed field. If $r_1\leq r_2$ we define a linear relation 
$$
B_{r_1}^{r_2}\colon H_*(\tilde X
_{r_1})\leadsto H_*(\tilde X
_{r_2})
$$ 
by declaring $a_1\in H_*(\tilde X_{r_1})$ to be in relation with $a_2\in H_*(\tilde X_{r_2})$ iff
their images in $H_*(\tilde X_{[r_1,r_2]})$ coincide, where $\tilde X_{[r_1,r_2]}=\tilde f^{-1}([r_1,r_2])$.

If $r_1\leq r_2\leq r_3$ we clearly have $B_{r_2}^{r_3}B_{r_1}^{r_2}\subseteq B_{r_1}^{r_3}$. 
To formulate a criterion which guarantees equality of relations, $B_{r_2}^{r_3}B_{r_1}^{r_2}=B_{r_1}^{r_3}$,
we introduce the following notation: A number $r\in\R$ is called \emph{tame value} if, for every $\varepsilon>0$, there exists 
a neighborhood $U$ of $\tilde X_r$ in $\tilde X_{[r-\varepsilon,r+\varepsilon]}$ such that each of the inclusions 
$\tilde X_r\subseteq U$, $\tilde X_{[r-\varepsilon,r]}\cap U\subseteq U$, and $\tilde X_{[r,r+\varepsilon]}\cap U\subseteq U$,
induces isomorphisms in homology. 
The crucial point is that in this case the triad $(\tilde X_{[r-\varepsilon,r+\varepsilon]};\tilde X_{[r,r+\varepsilon]},\tilde X_{[r-\varepsilon,r]})$
gives rise to a long exact Mayer--Vietoris sequence.
Note that for a tame map as considered in Section~\ref{S1}, all values are tame.

\begin{lemma}\label{L:MV}
Suppose $r_1\leq r_2\leq r_3$ and assume $r_2$ is a tame value. 
Then, as linear relations, $B_{r_2}^{r_3}B_{r_1}^{r_2}=B_{r_1}^{r_3}$.
\end{lemma}

\begin{proof}
Since $r_2$ is a tame value, we have an exact Mayer--Vietoris sequence,
$$
H_*(\tilde X
_{r_2})\to H_*(\tilde X_{[r_1,r_2]})\oplus H_*(\tilde X_{[r_2,r_3]})\to H_*(\tilde X_{[r_1,r_3]}).
$$
This immediately gives $B_{r_2}^{r_3}B_{r_1}^{r_2}\supseteq B_{r_1}^{r_3}$. 
As the converse inclusion, $B_{r_2}^{r_3}B_{r_1}^{r_2}\subseteq B_{r_1}^{r_3}$, is obvious, the lemma follows.
\qed
\end{proof}

Fix a tame value $\theta\in S^1$ of $f$ and a lift $\tilde\theta\in\R$, $e^{\mathbf i\tilde\theta}=\theta$.
Using the projection $\tilde X\to X$, we may canonically identify $\tilde X_{\tilde\theta}=X_\theta=f^{-1}(\theta)$.
Moreover, let $\tau\colon\tilde X\to\tilde X$ denote the fundamental deck transformation,  i.e.\ $\tilde f\circ\tau=\tilde f+2\pi$. 
Note that $\tau$ induces homeomorphisms between levels, $\tau\colon\tilde X_r\to\tilde X_{r+2\pi}$, and define a linear relation 
$$
R_\theta \colon H_*(X_\theta)\leadsto H_*(X_\theta)
$$ 
as the composition
\begin{equation}\label{E:RSigma}
H_*(X_\theta)=H_*(\tilde X_{\tilde\theta})\overset{B_{\tilde\theta}^{\tilde\theta+2\pi}}\leadsto H_*(\tilde X_{\tilde\theta+2\pi})\overset{\tau_*^\dag}
\leadsto H_*(\tilde X_{\tilde\theta})=H_*(X_\theta).
\end{equation}
In other words, for $a,b\in H_*(X_\theta)$ we have $aRb$ iff $aB_{\tilde\theta}^{\tilde\theta+2\pi}(\tau_*b)$,
i.e.\ iff $a$ and $\tau_*b$ coincide in $H_*(\tilde X_{[\tilde\theta,\tilde\theta+2\pi]})$. 
Particularly, we have:

\begin{lemma}\label{L:4}
If $a,b\in H_*(X_\theta)$ and $aRb$, then $a=\tau_*b$ in $H_*(\tilde X)$.
\end{lemma}

We will continue to use the notation $K_\pm$, $D_\pm$, and $R_\reg$ introduced in the previous section for this relation $R$ on $H_*(X_\theta)$.
Particularly, its regular part,  
$$
R_\reg\colon H_*(X_\theta)_\reg\to H_*(X_\theta)_\reg,
$$
is a module automorphism.

\begin{lemma}\label{L:5}
We have:
\begin{align*}
K_+&=\ker\bigl(H_*(X_\theta)\to H_*(\tilde X_{[\tilde\theta,\infty)})\bigr)
\\
K_-&=\ker\bigl(H_*(X_\theta)\to H_*(\tilde X_{(-\infty,\tilde\theta]})\bigr)
\end{align*}
Both maps are induced by the canonical inclusion $X_\theta=\tilde X_{\tilde\theta}\to\tilde X$.
\end{lemma}

\begin{proof}
We will only show the first equality, the other one can be proved along the same lines. To see the inclusion
$K_+\subseteq\ker(H_*(X_\theta)\to H_*(\tilde X_{[\tilde\theta,\infty)}))$, let $a\in K_+$. Hence, there exist $a_k\in H_*(X_\theta)$,
almost all of which vanish, such that $aRa_1Ra_2R\cdots$. In $H_*(\tilde X_{[\tilde\theta,\tilde\theta+2\pi]})$, we thus have:
$$
a=\tau_*a_1,\quad a_1=\tau_*a_2,\quad a_2=\tau_*a_3,\quad\dotsc 
$$
In $H_*(\tilde X_{[\tilde\theta,\infty)})$, we obtain:
$$
a=\tau_*a_1=\tau_*^2a_2=\tau_*^3a_3=\cdots
$$
Since some $a_k$ have to be zero, we conclude that $a$ vanishes in $H_*(\tilde X_{[\tilde\theta,\infty)})$.

To see the converse inclusion, $K_+\supseteq\ker(H_*(\tilde X_\theta)\to H_*(\tilde X_{[\tilde\theta,\infty)}))$, set 
$$
U:=\bigsqcup_{\text{$0\leq k$ even}}\tilde X_{[\tilde\theta+2\pi k,\tilde\theta+2\pi(k+1)]},\qquad
V:=\bigsqcup_{\text{$1\leq k$ odd}}\tilde X_{[\tilde\theta+2\pi k,\tilde\theta+2\pi(k+1)]}
$$
and note that $U\cup V=\tilde X_{[\tilde\theta,\infty)}$, as well as $U\cap V=\bigsqcup_{k\in\mathbb N}\tilde X_{\tilde\theta+2\pi k}$. 
Since $\theta$ is a tame value, we have an exact Mayer--Vietoris sequence
$$
\bigoplus_{k\in\mathbb N}H_*(\tilde X_{\tilde\theta+2\pi k})=H_*\Bigl(\bigsqcup_{k\in\mathbb N}\tilde X_{\tilde\theta+2\pi k}\Bigr)\to H_*(U)\oplus H_*(V)\to H_*(\tilde X_{[\tilde\theta,\infty)}).
$$
For $b\in\ker(H_*(X_\theta)\to H_*(\tilde X_{[\tilde\theta,\infty)}))$ we thus find $b_k\in H_*(\tilde X_{\tilde\theta+2\pi k})$, almost all of which vanish, such that:
\begin{align*}
b&=b_1\in H_*(\tilde X_{[\tilde\theta,\tilde\theta+2\pi]})\\
b_1+b_2&=0\in H_*(\tilde X_{[\tilde\theta+2\pi,\tilde\theta+4\pi]})\\
b_2+b_3&=0\in H_*(\tilde X_{[\tilde\theta+4\pi,\tilde\theta+6\pi]})\\
&\,\,\,\,\vdots
\end{align*}
Putting $c_k:=(-1)^{k-1}\tau^{-k}_*b_k\in H_*(\tilde X_{\tilde\theta})$, we obtain the following equalities in $H_*(\tilde X_{[\tilde\theta,\tilde\theta+2\pi]})$:
$$
b=\tau_*c_1,\quad c_1=\tau_*c_2,\quad c_2=\tau_*c_3,\quad\dotsc
$$
In other words, we have the relations $bRc_1Rc_2Rc_3R\cdots$. Since some $c_k$ has to be zero, we conclude $b\in K_+$, whence the lemma.
\qed
\end{proof}

Introduce the upwards Novikov complex as a projective limit of relative singular chain complexes,
$$
C_*^{\Nov,+}(\tilde X):=\varprojlim_{r}C_*(\tilde X,\tilde X_{[r,\infty)}),
$$
and let $H_*^{\Nov,+}(\tilde X)$ denote its homology. 
Alternatively, $C_*^{\Nov,+}(\tilde X)$ can be described as the chain complex of formal, possibly infinite, linear combinations of singular simplices in $\tilde X$
such that the number of simplices intersecting $\tilde X_{(-\infty,r]}$ is finite, for all real values $r$.
Analogously, we define a downwards Novikov complex
$C_*^{\Nov,-}(\tilde X)=\varprojlim_r C_*(\tilde X,\tilde X_{(-\infty,r]})$ and the corresponding homology,
$H_*^{\Nov,-}(\tilde X)$. We will also use similar notation for subsets of $\tilde X$.

\begin{lemma}\label{L:6}
We have:
\begin{align*}
D_+&=\ker\bigl(H_*(X_\theta)\to H_*^{\Nov,+}(\tilde X_{[\tilde\theta,\infty)})\bigr)
\\
D_-&=\ker\bigl(H_*(X_\theta)\to H_*^{\Nov,-}(\tilde X_{(-\infty,\tilde\theta]})\bigr)
\end{align*}
Both maps are induced by the canonical inclusion $X_\theta=\tilde X_{\tilde\theta}\to\tilde X$.
\end{lemma}

\begin{proof}
Using the exact Mayer--Vietoris sequence
$$
\prod_{k\in\mathbb N}H_*(\tilde X_{\tilde\theta+2\pi k})=
H_*^{\Nov,+}\Bigl(\bigsqcup_{k\in\mathbb N}\tilde X_{\tilde\theta+2\pi k}\Bigr)\to H_*^{\Nov,+}(U)\oplus H_*^{\Nov,+}(V)\to H_*^{\Nov,+}(\tilde X_{[\tilde\theta,\infty)}),
$$
this can be proved along the same lines as Lemma~\ref{L:5}.
\qed
\end{proof}

Let us introduce a complex
$$
C_*^\lf(\tilde X):=\varprojlim_r C_*(\tilde X,\tilde X_{(-\infty,-r]}\cup\tilde X_{[r,\infty)})
$$
and denote its homology by $H_*^\lf(\tilde X)$. If $f$ is proper, this is the complex of locally finite singular chains. 

\begin{lemma}\label{L:7}
We have:
\begin{align*}
K_-+K_+&=\ker\bigl(H_*(X_\theta)\to H_*(\tilde X)\bigr)
\\
K_-+D_+&=\ker\bigl(H_*(X_\theta)\to H_*^{\Nov,+}(\tilde X)\bigr)
\\
D_-+K_+&=\ker\bigl(H_*(X_\theta)\to H_*^{\Nov,-}(\tilde X)\bigr)
\\
D_-+D_+&=\ker\bigl(H_*(X_\theta)\to H_*^\lf(\tilde X)\bigr)
\end{align*}
All maps are induced by the canonical inclusion $X_\theta=\tilde X_{\tilde\theta}\to\tilde X$.
\end{lemma}

\begin{proof}
The first statement follows from the exact Mayer--Vietoris sequence
$$
H_*(\tilde X_{\tilde\theta})\to H_*(\tilde X_{(-\infty,\tilde\theta]})\oplus H_*(\tilde X_{[\tilde\theta,\infty)})\to H_*(\tilde X)
$$
and Lemma~\ref{L:5}. The second assertion follows from the exact Mayer--Vietoris sequence
$$
H_*(\tilde X_{\tilde\theta})\to H_*(\tilde X_{(-\infty,\tilde\theta]})\oplus H_*^{\Nov,+}(\tilde X_{[\tilde\theta,\infty)})\to H_*^{\Nov,+}(\tilde X)
$$
and Lemma~\ref{L:5} and \ref{L:6}. Similarly, one can check the third equality. To see the last statement
we use the exact Mayer--Vietoris sequence
$$
H_*(\tilde X_{\tilde\theta})\to H_*^{\Nov,-}(\tilde X_{(-\infty,\tilde\theta]})\oplus H_*^{\Nov,+}(\tilde X_{[\tilde\theta,\infty)})\to H_*^\lf(\tilde X)
$$
and Lemma~\ref{L:6}.
\qed
\end{proof}

\begin{lemma}\label{L:8}
We have
$$
\ker\Bigl(H_*(\tilde X)\to H_*^{\Nov,-}(\tilde X)\oplus H_*^{\Nov,+}(\tilde X)\Bigr)
\subseteq\img\bigl(H_*(\tilde X_{\tilde\theta})\to H_*(\tilde X)\bigr),
$$
where all maps are induced by the tautological inclusions.
\end{lemma}

\begin{proof}
This follows from the following commutative diagram of exact Mayer--Vi\-e\-to\-ris sequences:
$$
\xymatrix{
H^\lf_{*+1}(\tilde X)\ar[r]^-\partial & H_*(\tilde X)\ar[r] & H_*^{\Nov,-}(\tilde X)\oplus H_*^{\Nov,+}(\tilde X)
\\
H^\lf_{*+1}(\tilde X)\ar@{=}[u]\ar[r]^-\partial & H_*(\tilde X_{\tilde\theta})\ar[u]\ar[r] & H_*^{\Nov,-}(\tilde X_{(-\infty,\tilde\theta]})\oplus H_*^{\Nov,+}(\tilde X_{[\tilde\theta,\infty)}) \ar[u]
}
$$
A similar argument was used in \cite[Lemma~2.5]{HL99b}.
\qed
\end{proof}

\begin{theorem}\label{T:monreg}
The inclusion $\iota\colon X_\theta=\tilde X_{\tilde\theta}\to\tilde X$ induces a canonical isomorphism
$$
H_*(X_\theta)_\reg=\frac D{(K_-+K_+)\cap D}
\xrightarrow\cong\ker\Bigl(H_*(\tilde X)\to H_*^{\Nov,-}(\tilde X)\oplus H_*^{\Nov,+}(\tilde X)\Bigr),
$$
intertwining $R_\reg$ with the monodromy isomorphism induced by the deck transformation $\tau\colon\tilde X\to\tilde X$ on the right hand side. 
Moreover, working with coefficients in a field, and assuming that $H_*(X_\theta)$ is finite dimensional, the common kernel on the right hand side above
coincides with
$$
\ker\bigl(H_*(\tilde X)\to H_*^{\Nov,-}(\tilde X)\bigr)
=\ker\bigl(H_*(\tilde X)\to H_*^{\Nov,+}(\tilde X)\bigr).
$$
Particularly, in this case the latter two kernels are finite dimensional too.
\end{theorem}

\begin{proof}
It follows immediately from Lemma~\ref{L:7} and \ref{L:8} that $\iota_*\colon H_*(X_\theta)\to H_*(\tilde X)$ induces an isomorphism
$$
\frac{(K_-+D_+)\cap(D_-+K_+)}{K_-+K_+}\xrightarrow\cong\ker\Bigl(H_*(\tilde X)\to H_*^{\Nov,-}(\tilde X)\oplus H_*^{\Nov,+}(\tilde X)\Bigr).
$$
In view of Lemma~\ref{L:4}, this isomorphism intertwines the isomorphism induced by $R$ on the left hand side,
with the monodromy isomorphism on the right hand side. Combining this with Proposition~\ref{P:AA} we obtain the first assertion. 
For the second statement it suffices to show
\begin{equation}\label{E:543}
\ker\bigl(H_*(\tilde X)\to H_*^{\Nov,+}(\tilde X)\bigr)
\subseteq\ker\Bigl(H_*(\tilde X)\to H_*^{\Nov,-}(\tilde X)\oplus H_*^{\Nov,+}(\tilde X)\Bigr),
\end{equation}
as the converse inclusion is obvious, and the corresponding statement for the 
downward Novikov homology can be derived analogously.
To this end, suppose $a\in\ker\bigl(H_*(\tilde X)\to H_*^{\Nov,+}(\tilde X)\bigr)$.
Then there exists $k$ such that $\tau^k_*a$ is contained in the image of $H_*(\tilde X_{(-\infty,\tilde\theta]})\to H_*(\tilde X)$.
Using the exact Mayer--Vietoris sequence
$$
H_*(\tilde X_{\tilde\theta})\to H_*(\tilde X_{(-\infty,\tilde\theta]})\oplus H_*^{\Nov,+}(\tilde X_{[\tilde\theta,\infty)})\to H_*^{\Nov,+}(\tilde X)
$$
we conclude, that $\tau_*^ka$ is contained in the image of $H_*(\tilde X_{\tilde\theta})\to H_*(\tilde X)$.
Thus $\tau_*^ka$ is contained in $\iota_*(D_+)$, see Lemma~\ref{L:7}.
Since $H_*(X_\theta)$ is assumed to be a finite dimensional vector space, we have
$\iota_*(D_-)=\iota_*(D)=\iota_*(D_+)$, see~\eqref{E:12}. Using Lemma~\ref{L:7} we thus conclude 
$\tau_*^ka$ is contained in the kernel on the right hand side of \eqref{E:543}.
Since this common kernel is invariant under the isomorphism
$\tau_*\colon H_*(\tilde X)\to H_*(\tilde X)$, we conclude that $a$ has to be contained in the common
kernel too, whence the theorem.
\qed
\end{proof}

We conclude this section with a proof of Theorem~\ref{T4}.
Suppose $X$ is a compact ANR and let $f\colon X\to S^1$ be a tame map as in Section~\ref{S1}.
Fix regular and critical angles, $0<t_1<\theta_1<\cdots<t_m<\theta_m\leq2\pi$, and consider the associated $G_{2m}$-representation $\rho_r=\{V_i,\alpha_i,\beta_i\}$, see Section~\ref{SS3}.
Note that the linear relation $R^\theta_r$ on $H_r(X_\theta)$ introduced in Section~\ref{S1} is just the degree $r$ part of the relation considered in this section, see \eqref{E:RSigma}.
From Lemma~\ref{L:MV} we immediately obtain:

\begin{lemma}
The following equalities of relations on $H_r(X_\theta)$ hold true:
\begin{enumerate}[(a)]
\item 
If $\theta=\theta_i$, then 
$R_r^\theta=\alpha_i\beta_{i-1}^\dag\alpha_{i-1}\cdots\beta_1^\dag\alpha_1\beta_m^\dag\alpha_m\cdots\alpha_{i+2}\beta_{i+1}^\dag\alpha_{i+1}\beta_i^\dag$.
\item
If $\theta=t_i$, then 
$R_r^\theta=\beta_{i-1}^\dag\alpha_{i-1}\beta_{i-2}^\dag\cdots\beta_1^\dag\alpha_1\beta_m^\dag\alpha_m\cdots\beta_{i+1}^\dag\alpha_{i+1}\beta_i^\dag\alpha_i$.
\end{enumerate}
\end{lemma}

\begin{lemma}
Suppose $\rho=\{V_i,\alpha_i,\beta_i\}$ is a $G_{2m}$-representation with Jordan blocks $\bigoplus_{J\in\mathcal J}T(J)$.
Then, for all $1\leq i\leq m$, the following hold true:
\begin{enumerate}[(a)]
\item 
$\bigl(\alpha_i\beta_{i-1}^\dag\alpha_{i-1}\cdots\beta_1^\dag\alpha_1\beta_m^\dag\alpha_m\cdots\alpha_{i+2}\beta_{i+1}^\dag\alpha_{i+1}\beta_i^\dag\bigr)_\reg$
is conjugate to $\bigoplus_{J\in\mathcal J}T(J)$.
\item
$\bigl(\beta_{i-1}^\dag\alpha_{i-1}\beta_{i-2}^\dag\cdots\beta_1^\dag\alpha_1\beta_m^\dag\alpha_m\cdots\beta_{i+1}^\dag\alpha_{i+1}\beta_i^\dag\alpha_i\bigr)_\reg$ 
is conjugate to $\bigoplus_{J\in\mathcal J}T(J)$.
\end{enumerate}
\end{lemma}

\begin{proof}
W.l.o.g.\ it suffices to consider an indecomposable representation $\rho$.
For such a $\rho$, however, the statement follows immediately from the classification of indecomposable representations discussed in Section~\ref{S2}, see also Proposition~\ref{P:C}.
\qed
\end{proof}

Combining the preceding two lemmas, we obtain Theorem~\ref{T4}.

\section{Proof of Theorem~\ref{T7}}\label {S9}

Suppose $f\colon X\to \mathbb S^1$ is a tame map. 
For $0<\theta' \leq \theta'' \leq 2\pi$ we will use the notation $X_{[\theta',\theta'']}:=f^{-1}([\theta',\theta''])$, and write $X(\theta):=X_{[\theta,\theta]}= f^{-1}(\theta)$.

Let $0<\theta_1<\theta_2<\cdots<\theta_N\leq 2\pi$ be the collection of all critical values and put $\epsilon(f):=\min\{|\theta_{i+1}-\theta_i|:1\leq i\leq N\}$ where $\theta_{N+1}:=\theta_1+2\pi$.
Note that any bar code $I\in \widetilde {\mathcal B}_r(f)$ has the left end of the form $\theta_i+ 2\pi k$ and the right end of the form $\theta_j +2\pi k'$ where $i,j\in\{1,\dotsc,N\}$ and $k,k'\in\mathbb Z$.
Put $l(I):=\theta_i$ and $r(I):=\theta_j$.
The numbers $l(I)$ and $r(I)$ are well defined for barcodes in $\mathcal B_r(f)$ which can be considered as equivalency classes of elements in $\widetilde {\mathcal B}_r(f)$.

\begin{proposition}\label {P71}
For any tame map $f\colon X\to \mathbb S^1$ and $0<\epsilon<\epsilon(f)$ we have:
\begin{multline*}
\dim H_r\bigl(X_{[\theta_i-\epsilon,\theta_i+\epsilon]},X(\theta_i-\epsilon)\bigr)
\\=\sharp\{I\in{\mathcal B}^c_r(f) \mid l(I)=\theta_i\}
+\sharp\{I\in{\mathcal B}^o_{r-1}(f) \mid r(I)=\theta_i\}
\\+\sharp\{I\in{\mathcal B}^{co}_r(f) \mid l(I)=\theta_i\}
+\sharp\{I\in{\mathcal B}^{co}_{r-1}(f) \mid r(I)=\theta_i\}.
\end{multline*}
\end{proposition}

\begin{proof}
By the long exact homology sequence of the pair $(X_{[\theta_i-\epsilon, \theta_i+\epsilon]}, X{(\theta_i-\epsilon)})$, 
\begin{multline}\label{E:ytrewq}
\dim H_r\bigl(X_{[\theta_i -\epsilon, \theta_i +\epsilon]}, X{(\theta_i-\epsilon)}\bigr)
\\=\dim\coker\Bigl(H_r(X{(\theta_i-\epsilon)})\xrightarrow{i_r}H_r\bigl(X_{[\theta_i-\epsilon, \theta_i+\epsilon]})\Bigr)
\\+\dim\ker\Bigl(H_{r-1}(X{(\theta_i-\epsilon)})\xrightarrow{i_{r-1}}H_{r-1}(X_{[\theta_i-\epsilon, \theta_i+\epsilon]})\Bigr).
\end{multline}
According to Proposition~\ref{P36}, there exist isomorphisms $\omega_1(r)$ and $\omega_2(r)$ such that the diagram
$$
\xymatrix{
H_r(X(\theta_i-\epsilon))\ar[r]^-{i_r}&H_r(X_{[\theta_i-\epsilon,\theta_i+\epsilon]})
\\
\kappa[S_1(r)]\ar[u]_\cong^-{\omega_1(r)}\ar[r]&\kappa[S_2(r)]\ar[u]^\cong_-{\omega_2(r)}
}
$$
commutes, where
\begin{align*}
S_1(r)&=\{I\in\widetilde{\mathcal B}_r(f)\mid\theta_i-\epsilon\in I\}\sqcup\widetilde{\mathcal J}_r(f),
\\
S_2(r)&=\{I\in\widetilde{\mathcal B}_r(f)\mid\theta_i\in I\}\sqcup\widetilde{\mathcal J}_r(f),
\end{align*}
and the lower horizontal arrow in the diagram denotes the canonical map associated with the subsets $S_1(r)$ and $S_2(r)$ of $\widetilde{\mathcal B}_r(f)\sqcup\widetilde{\mathcal J}_r(f)$.
From this description one readily obtains $\coker(i_r)\cong\kappa[S_3(r)]$ and $\ker(i_{r-1})\cong\kappa[S_4(r)]$, where
\begin{align*}
S_3(r)&=\{I\in\mathcal B^c_r\mid l(I)=\theta_i\}\sqcup\{I\in\mathcal B^{co}_r\mid l(I)=\theta_i\},
\\
S_4(r)&=\{I\in\mathcal B^o_{r-1}\mid r(I)=\theta_i\}\sqcup\{I\in\mathcal B^{co}_{r-1}\mid r(I)=\theta_i\},
\end{align*}
and thus
\begin{align*}
\dim\coker(i_r)&=\sharp\{I\in\mathcal B^c_r\mid l(I)=\theta_i\}+\sharp\{I\in\mathcal B^{co}_r\mid l(I)=\theta_i\},
\\
\dim\ker(i_{r-1})&=\sharp\{I\in\mathcal B^o_{r-1}\mid r(I)=\theta_i\}+\sharp\{I\in\mathcal B^{co}_{r-1}\mid r(I)=\theta_i\}.
\end{align*}
Combining these equations with \eqref{E:ytrewq} we obtain the proposition.
%
\end{proof}

Let $M$ be a closed manifold of dimension $n$, and suppose $f\colon M\to \mathbb S^1$ is a Morse map, i.e., all critical points are non-degenerated. 
Let $\mathcal X(f)$ denote the set of critical points of $f$.
For $r=0,\dotsc,n$ and $i=1,\dotsc,N$ let
$$
\mathcal X_{r,i}(f):=\bigl\{x\in\mathcal X(f):\textrm{$\operatorname{ind}(x)=r$ and $f(x)=\theta_i$}\bigr\}
$$ 
denote the set of critical points of Morse index $r$ corresponding to the critical value $\theta_i$.
Moreover, put $c_{r,i}:=\sharp \mathcal X_{r,i}$.

Recall that the Morse Lemma, see \cite[Lemma~2.2]{Mi3}, asserts that for every non-degenerate critical point $x$ of $f$ there exists an open neighborhood $U_x$ of $0$ in $\mathbb R^n$ 
and a diffeomorphism onto its image, $\varphi_x\colon U_x\to M$, such that $\varphi_x(0)=x$ and
$$
f(\varphi_x(t_1,\dotsc,t_n))=f(x)-t_1^2-\cdots-t_k^2+t^2_{k+1}+\cdots+t_n^2
$$  
holds for all $(t_1,\dotsc,t_n)\in U_x$, where $k=\operatorname{ind}_f(x)$.
In particular, $\mathcal X(f)$ is finite, for $M$ is assumed to be compact.
We fix Morse coordinates $\varphi_x\colon U_x\to M$ as above for every critical point $x\in\mathcal X(f)$.
Moreover, we assume $0<\epsilon<\epsilon(f)$ 
\footnote{A Morse map is tame when $M$ is compact.}
is sufficiently small such that 
$D^n(\epsilon):=\{(t_1,\dotsc,t_n):\sum_{i+1}^nt_i^2\leq\epsilon\}\subseteq U_x$ for all $x\in\mathcal X(f)$.

%

\begin{proposition}[Morse theorem]\label{P72}   
If $\epsilon>0$ is sufficiently small, then
$$
\dim H_r\bigl(M_{[\theta_i -\epsilon, \theta_i+\epsilon]},M{(\theta_i-\epsilon)}\bigr)=c_{r,i},
$$
for every $r=0,\dotsc,n$ and all $i=1,\dotsc,N$.
\end{proposition}

The proof of this proposition can be found in any book in Morse theory, see for instance \cite[Section~\S5]{Mi3}. The idea is simple. 
For every critical point $x\in\mathcal X(f)$ one defines $B_x:=\varphi_x(D^{\operatorname{ind}(x)}(\epsilon))$, 
where $D^k(\epsilon):=\bigl\{(t_1,\dotsc,t_k,0,\dotsc,0):\sum_{i=1}^kt_i^2 \leq \epsilon\bigr\}$. 
For each $i=1,\dotsc,N$ one considers 
$$
X(i):=M{(\theta_i-\epsilon)}\cup\bigcup_{x\in\mathcal X(f):f(x)=\theta_i}B_x\subseteq M_{[\theta_i -\epsilon, \theta_i+\epsilon]}.
$$
As in \cite[Section~\S3]{Mi3} one can verify that $M_{[\theta_i-\epsilon,\theta_i+\epsilon]}$ retracts by deformation to $X(i)$.
The deformation is obtained using the flow of the gradient vector field $-\operatorname{grad}_g(f)$ where $g$ is a conveniently choosen Riemannian metric.
Consequently,
\begin{multline*}
H_r\bigl(M_{[\theta_i-\epsilon,\theta_i+\epsilon]},M{(\theta_i-\epsilon)}\bigr)
=H_r\bigl(X(i),M{(\theta_i-\epsilon)}\bigr)
\\=\bigoplus_{x\in\mathcal X_{r,i}(f)}H_r\bigl(D^r_x(\epsilon),\partial D^r_x(\varepsilon)\bigr)
\cong\kappa^{c_{r,i}}.
\end{multline*}

From Propositions~\ref{P71} and \ref {P72} we get
$$
c_r(f)=\sum_{i=1}^Nc_{r,i}=\sharp\mathcal B_r^c(f)+\sharp\mathcal B_{r-1}^o(f)+\sharp\mathcal B_r^{co}(f)+\sharp\mathcal B_{r-1}^{co}(f).
$$
Combining this with Theorem~\ref{T1}\itemref{T1:a}, we obtain the statement for angle-vcalued maps in Theorem~\ref{T7}.
A real-valued map can be viewed as an angle-valued map after composition with an embedding of $\mathbb R$ in $\mathbb S^1$.
In this case the Novikov--Betti numbers coincide with the Betti numbers, whence the statement for real-valued maps in Theorem~\ref{T7} follows from the statement for angle-valued maps.

\appendix
\addcontentsline{toc}{section}{Appendices}

\section{An example}\label{App1}

Consider the space $X$ obtained from $Y$ described in Figure~\ref{cv-ex} by identifying its right end $Y_1$ (a union of three circles) to the left end 
$Y_0$ (a union of three circles) following the map $\phi\colon Y_1\to Y_0$ given by the matrix 
\begin{equation*}
\begin{pmatrix}
3&3 & 0\\
2&3&-1\\
1 & 2&3            
\end{pmatrix}.
\end{equation*}
The meaning of this matrix as a map $\phi$ is the following: circle (1) is divided in 6 parts, circle (2) in 8 parts and and circle (3) in 4 parts; 
the first three parts of circle (1) wrap clockwise around circle (1) to cover it three times, the next two wrap clockwise around circle (2) to cover it twice 
and around circle three to cover it three times. Similarly circle (2) and (3) wrap over circles (1), (2) and (3) as indicated by the matrix. 
The first part of circle (3) wraps counterclockwise around circle (2).
The map $f\colon X\to S^1$ is induced by the projection of $Y$ on the interval $[0,2\pi]$.

\begin{figure}
\includegraphics{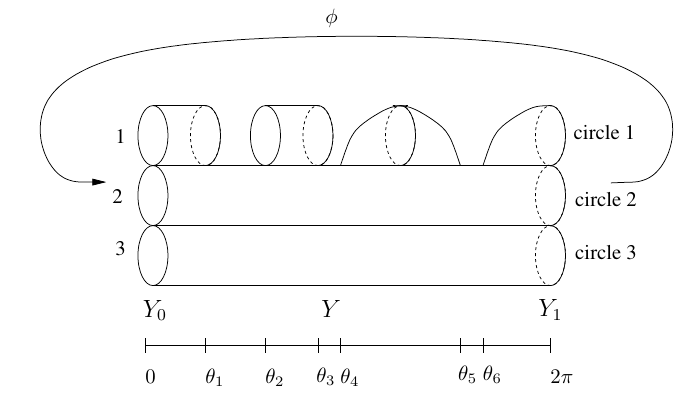}
\caption{Example of $r$-invariants for a circle valued map.}
\label{cv-ex}
\end{figure}

\subsubsection*{The critical angles} 

Clearly the critical angles of $f$ are 
$$
\{ \theta_0=0= 2\pi, \theta_1, \theta_2, \theta_2, \theta_3, \theta_4, \theta_5, \theta_6\}.
$$

\subsubsection*{The Jordan blocks} 

The $r$-monodromy of $f$ calculated at $\theta=0$ is given by the regular part of the linear relation $R( A_r, B_r)$ with 
$ A_r:= \phi_r\colon H_r (Y_1)\to H_r (Y)$ induced by $\phi$ and $B_r:=i_r\colon H_r(Y_1)\to H_r(Y)$ induced by the inclusion $Y_1\subset Y$. 
Since $H_2(Y_1)= 0$ there is no monodromy  for $r=2$ and for $r=0$  one has $R_\reg(A_0, B_0)=\Id$ which leads to 
$$
\mathcal J_0(f) =\{(1,1)\}.
$$
For $r=1$ the reader can see from the picture above that $H_1(Y_1)=\kappa^3$ generated by the circles 1, 2, 3, and $H_1(Y)=\kappa^4$ generated by the circles 1, 2, 3, 
and an additional generator coming from the small cylinder above $[\theta_2, \theta_3]$. In this case 
$$
A_1=
\begin{pmatrix}
3&3 & 0\\
2&3&-1\\
1 & 2&3\\            
0& 0& 0\\
\end{pmatrix} 
\qquad\text{and}\qquad
B_1=
\begin{pmatrix}
0& 0& 0\\
0& 1& 0\\
0& 0& 1\\
0& 0& 0\\
\end{pmatrix}.
$$
Let 
$$
A=
\begin{pmatrix}
3&3 & 0\\
2&3&-1\\
1 & 2&3\\            
\end{pmatrix} 
\qquad\text{and}\qquad
B=
\begin{pmatrix}
0& 0& 0\\
0& 1& 0\\
0& 0& 1\\
\end{pmatrix}.
$$
In view of Observation~\ref{O87} one has $R(A_1, B_1)=R(A,B)$, and since $A$ is invertible, 
$$
R_\reg( A_1, B_1)=R_\reg(A,B)=R_\reg(\Id,A^{-1}B)=(R_\reg(A^{-1}B,\Id))^{-1}\cong\begin{pmatrix} 2&1\\0&2\\\end{pmatrix},
$$
hence 
$$
\mathcal J_1(f)= \{(2,2)\}.
$$
 
\subsubsection*{The bar codes} 

In view of Proposition~\ref{P36}(b) by inspections of $f^{-1}([\theta, \theta'])$ one concludes that $\mathcal B_0(f)=\emptyset$,  $\mathcal B_2(f)= \emptyset$, and in dimension 1 one has: 
one \emph{closed} bar code $[\theta_2, \theta_3]$;   
one \emph{open} bar code $(\theta_4, \theta_5)$; and 
one \emph{open-closed} bar code $(\theta_6, \theta_1 +2\pi]$.

\section{More examples}\label{App2}

The three examples below support the comments \itemref{C:a}, \itemref{C:b}, and \itemref{C:c} about the differences between the stability results Theorems~\ref{T2} and \ref{T5} in the introduction. 
In all three examples we construct a $1$-parameter family of maps, $f_\epsilon\colon X\to\mathbb S^1$, and analyze the dependence of the barcode structure on the parameter $\epsilon$.

As in Appendix~\ref{App1}, the maps in these examples are derived from maps $\overline f_\epsilon\colon Y\to[0,2\pi]$ by identifying $\overline f_\epsilon^{-1}(0)$ with $\overline f_\epsilon^{-1}(2\pi)$ and $0$ with $2\pi$, respectively.

\begin{example}\label{Ex:B1}
Take $Y=[0,2\pi]$.
For $0\leq\epsilon<\pi$ define a map $\overline f_\epsilon\colon[0,2\pi]\to [0,2\pi]$ by
\begin{equation*}
\overline f_\epsilon(x):=
\begin{cases}
\frac{\pi+\epsilon}{\pi-\epsilon}x&\text{if $0\leq x\leq \pi-\epsilon$,}\\
-x+2\pi&\text{if $\pi-\epsilon \leq x\leq \pi+\epsilon$, and}\\
\frac{\pi+\epsilon}{\pi-\epsilon}x-\frac{4\pi\epsilon}{\pi-\epsilon}&\text{if $\pi+\epsilon\leq x\leq 2\pi$.}
\end{cases}
\end{equation*}
Figure~\ref{Fig:B1} displays the graph of $\overline f_\epsilon$.
\begin{figure}
\begin{center}
\begin{tikzpicture}
\draw [<->]  (0,5) -- (0,0) -- (5,0);
\draw [thick] (0,0) -- (1.8,2.2);
\draw [thick] (1.8,2.2) -- (2.2,1.8);
\draw [thick] (2.2,1.8) -- (4,4);
\node at (0.7,4.7) {$\overline f_\epsilon(x)$};
\node at (4.7,0.3) {$x$};
\draw [dotted] (4,4) -- (0,4);
\draw [dotted] (4,4) -- (4,0);
\draw [dotted] (2.2,1.8) -- (2.2,0);
\draw [dotted] (2.2,1.8) -- (0,1.8);
\draw [dotted] (1.8,2.2) -- (0,2.2);
\draw [dotted] (1.8,2.2) -- (1.8,0);
\node at (-0.7,4) {2$\pi$};
\node at (-0.7,2.2) {$\pi$+$\epsilon$};
\node at (-0.7,1.8) {$\pi$-$\epsilon$};
\node at (1.7,-0.3) {$\pi$-$\epsilon$};
\node at (2.4,-0.3) {$\pi$+$\epsilon$};
\node at (4,-0.3) {2$\pi$};
\end{tikzpicture}
\end{center}
\caption{The graph of the map $\overline f_\epsilon\colon[0,2\pi]\to[0,2\pi]$ in Example~\ref{Ex:B1}.}
\label{Fig:B1}
\end{figure}
Clearly, $f_\epsilon$ can be made is arbitrary closed to $f_0$.
For $\epsilon=0$ there are no barcodes and only one Jordan cell, $\mathcal J_0(f_0)=\{(\kappa^1,\operatorname{id})\}$.
For $0<\epsilon<\pi$ we have the same Jordan block, $\mathcal J_0(f_\epsilon)=\{(\kappa^1,\operatorname{id})\}$, and in addition two barcodes, $\mathcal B_0(f_\epsilon)=\{[\pi-\epsilon,\pi+\epsilon),(\pi-\epsilon,\pi+\epsilon]\}$.
\end{example}

\begin{example}\label{Ex:B2}
Take $Y=Y_1\cup Y_2\cup Y_3\subseteq\mathbb R^2$, where
\begin{align*}
Y_1&=\{(x,0)\mid0\leq x\leq 2\pi\},
\\
Y_2&=\{(x,y)\mid(x-\pi/2)^2+y^2=(\pi/4)^2,y\geq0\},\,\,\text{and}
\\
Y_3&=\{(x,y)\mid y=x-\pi,\,5\pi/4\leq x\leq7\pi/4\}.
\end{align*}
For $0\leq\epsilon\leq\pi/4$ define $\overline f_\epsilon\colon Y\to[0,2\pi]$ as composition $\overline f_\epsilon=l_\epsilon\circ p$, where $p\colon Y\to [0,2\pi]$ denotes the coordinate projection given by $p(x,y)=x$ and $l_\epsilon\colon[0,2\pi]\to[0,2\pi]$ is the piecewise linear map defined by
$$
l_\epsilon(t):=
\begin{cases}
\frac{t}{\pi/4}(\pi/4+\epsilon)&\text{if $0\leq t\leq\pi/4$,}\\
\frac{3\pi/4-t}{\pi/2}(\pi/4+\epsilon)+\frac{t-\pi/4}{\pi/2}(3\pi/4-\epsilon)&\text{if $\pi/4\leq t\leq3\pi/4$,}\\
\frac{5\pi/4-t}{\pi/2}(3\pi/4-\epsilon)+\frac{t-3\pi/4}{\pi/2}(5\pi/4+\epsilon)&\text{if $3\pi/4\leq t\leq5\pi/4$,}\\
\frac{7\pi/4-t}{\pi/2}(5\pi/4+\epsilon)+\frac{t-5\pi/4}{\pi/2}(7\pi/4-\epsilon)&\text{if $5\pi/4\leq t\leq7\pi/4$, and}\\
\frac{2\pi-t}{\pi/4}(7\pi/4-\epsilon)+\frac{t-7\pi/4}{\pi/4}2\pi&\text{if $7\pi/4\leq t\leq2\pi$.}
\end{cases}
$$
The graph of the map $l_\epsilon$ is displayed in Figure~\ref{Fig:B2b}, the space $Y$ and the map $\overline f_\epsilon$ are indicated in Figure~\ref{Fig:B2}.
\begin{figure}
\begin{center}
\begin{tikzpicture}
\draw [<->]  (0,5) -- (0,0) -- (5,0);
\draw [dotted] (0,0) -- (5,5);
\draw [thick] (0,0) -- (0.5,0.8);
\draw [thick] (0.5,0.8) -- (1.5,1.2);
\draw [thick] (1.5,1.2) -- (2.5,2.8);
\draw [thick] (2.5,2.8) -- (3.5,3.2);
\draw [thick] (3.5,3.2) -- (4,4);
\node at (0.7,4.7) {$l_\epsilon(t)$};
\node at (4.7,0.3) {$t$};
\draw [dotted] (4,4) -- (0,4);
\draw [dotted] (4,4) -- (4,0);
\draw [dotted] (0.5,0) -- (0.5,0.8);
\draw [dotted] (0.5,0.8) -- (0,0.8);
\draw [dotted] (1.5,0) -- (1.5,1.2);
\draw [dotted] (1.5,1.2) -- (0,1.2);
\draw [dotted] (2.5,0) -- (2.5,2.8);
\draw [dotted] (2.5,2.8) -- (0,2.8);
\draw [dotted] (3.5,0) -- (3.5,3.2);
\draw [dotted] (3.5,3.2) -- (0,3.2);
\node at (0,-0.3) {$0$};
\node at (0.5,-0.3) {$\frac\pi4$};
\node at (1.5,-0.3) {$\frac{3\pi}4$};
\node at (2.5,-0.3) {$\frac{5\pi}4$};
\node at (3.5,-0.3) {$\frac{7\pi}4$};
\node at (4,-0.3) {$2\pi$};
\node at (-0.6,0) {$0$};
\node at (-0.6,0.75) {$\frac\pi4+\epsilon$};
\node at (-0.6,1.25) {$\frac{3\pi}4-\epsilon$};
\node at (-0.6,2.75) {$\frac{5\pi}4+\epsilon$};
\node at (-0.6,3.25) {$\frac{7\pi}4-\epsilon$};
\node at (-0.6,4) {$2\pi$};
\end{tikzpicture}
\end{center}
\caption{The graph of the map $l_\epsilon\colon[0,2\pi]\to[0,2\pi]$ in Example~\ref{Ex:B2}.}
\label{Fig:B2b}
\end{figure}
\begin{figure}
\begin{center}
\begin{tikzpicture}
\draw [thick] (2,0) to [out=90,in=180] (3,1);
\draw [thick] (3,1) to [out=0,in=90] (4,0);
\draw (1,0) -- (9,0);
\draw (6,0) -- (8,2);
\node at (1,-0.5) {0};
\node at (3,-0.5) {$\pi$/2};
\node at (7,-0.5) {$3\pi$/2};
\draw (1,-1.5) -- (9,-1.5);
\draw [->] (5,-0.4) -- (5,-1.2);
\node at (5.4, -0.8) {$\overline{f}_{\epsilon}$};
\node at (9,-0.5) {2$\pi$};
\end{tikzpicture}
\end{center}
\caption{The map $\overline f_\epsilon\colon Y\to[0,2\pi]$ in Example~\ref{Ex:B2}.}
\label{Fig:B2}
\end{figure}
For $0\leq\epsilon<\pi/4$ we have one Jordan cell, $\mathcal J_0(f_\epsilon)=\{(\kappa^1,\operatorname{id})\}$, two $0$-barcodes, and no $1$-barcodes, that is, $\mathcal B_0(f_\epsilon)=\{(\pi/4+\epsilon,3\pi/4-\epsilon),(5\pi/a+\epsilon,7\pi/4-\epsilon]\}$.
For $\epsilon=\pi/4$ we have the same Jordan cell, $\mathcal J_0(f_{\pi/4})=\{(\kappa^1,\operatorname{id})\}$, no $0$-barcodes, and one $1$-barcode, $\mathcal B_1(f_{\pi/4})=\{[\pi/2,\pi/2]\}$.
\end{example}

\begin{example}\label{Ex:B3}
Take $Y=Y_1\cup Y_2\subseteq[0,2\pi]\times\mathbb S^1$ where $Y_1=[0,2\pi]\times\{p\}$, $Y_2=[2\pi/3,4\pi/3]\times\mathbb S^1$, and $p\in\mathbb S^1$ is a base point. 
For $0\leq\epsilon<4\pi/3$ we define a map $\overline f_\epsilon\colon Y\to[0,2\pi]$ by
$$
\overline f_\epsilon(x,y):=
\begin{cases}
\frac{2\pi+3\epsilon}{2\pi}x&\text{if $0\leq x\leq2\pi/3$,}\\
\frac{2\pi-6\epsilon}{2\pi}x+3\epsilon&\text{if $2\pi/3\leq x\leq4\pi/3$, and}\\ 
\frac{2\pi+3\epsilon}{2\pi}x-3\epsilon&\text{if $4\pi/3\leq x\leq2\pi$.}
\end{cases}
$$
The space $Y$ and the map $\overline f_\epsilon$ are illustrated in Figure~\ref{Fig:B3}.
\begin{figure}
\includegraphics[width=\textwidth]{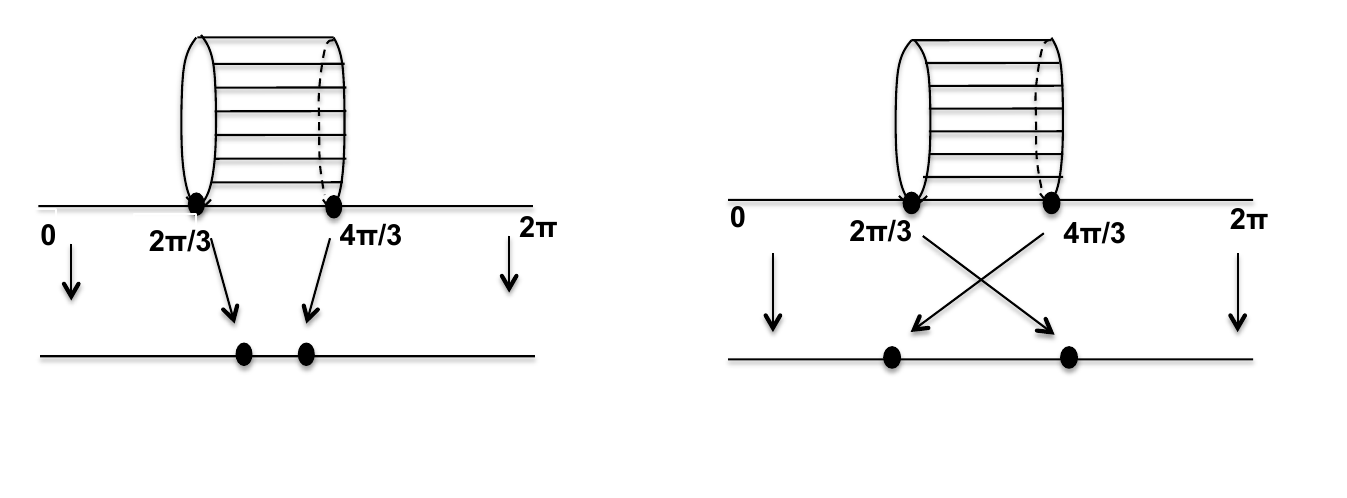}
\caption{The map $\overline f_\epsilon\colon Y\to[0,2\pi]$ in Example~\ref{Ex:B3} for $0\leq\epsilon<\pi/3$ (left) and $\pi/3<\epsilon<4\pi/3$ (right).}
\label{Fig:B3}
\end{figure}
For $0\leq\epsilon\leq\pi/3$ we have one Jordan cell, $\mathcal J_0(f_\epsilon)=\{(\kappa^1,\operatorname{id})\}$, no $0$-barcodes, and one $1$-barcode, $\mathcal B_1(f_\epsilon)=\{[2\pi/3+\epsilon, 4\pi/3-\epsilon]\}$.
For $\pi/3<\epsilon<4\pi/3$ we have the same Jordan cell, $\mathcal J_0(f_\epsilon)=\{(\kappa^1,\operatorname{id})\}$, two $0$-barcodes and one $1$-barcode, that is, $\mathcal B_0(f_\epsilon)=\{[4\pi/3-\epsilon,2\pi/3+\epsilon),(4\pi/3-\epsilon,2\pi/3+\epsilon]\}$, and $\mathcal B_1(f_\epsilon)=\{[4\pi/3-\epsilon,2\pi/3+\epsilon]\}$.
\end{example}

Denote by $f^{X_1}_\epsilon\colon X_1\to\mathbb S^1$ and $f^{X_2}_\epsilon\colon X_2\to\mathbb S^1$ the maps $f_\epsilon\colon X\to\mathbb S^1$ described in Examples~\ref{Ex:B2} and \ref{Ex:B3}, respectively.
Moreover, let $Q$ denote the Hilbert cube, that is, the product of countably many copies of the unit interval.
Note that $X_1\times Q$ and $X_2\times Q$ are homeomorphic compact ANRs.
Moreover, for $i=1,2$ we have $\mathcal B_r(f^{X_i}_\epsilon\colon X_i\to\mathbb S^1)=\mathcal B_r(f^{X_i}_\epsilon\circ p_i\colon X_i\times Q\to\mathbb S^1)$ where $p_i\colon X_i\times Q\to X_i$ denotes the canonical projection.
Hence, one can clearly provide four homotopic tame maps, $h_1, h_2, h_3, h_4\colon X_1\times Q\to\mathbb S^1$, with $\mathcal B_r(h_1)=\mathcal B_r(f^{X_1}_0)$, $\mathcal B_r(h_2)=\mathcal B_r(f^{X_1}_{\pi/4})$, $\mathcal B_r(h_3)=\mathcal B_r(f^{X_2}_{\pi/3})$, and $\mathcal B_r(h_4)=\mathcal B_r(f^{X_2}_{2\pi/3})$.

\section{Structure of finitely generated modules over principal ideal domains}\label{SS10}

In this appendix we recall basic facts about modules over principal ideal domains, and we provide more specific information about modules over the principal ideal domain of Laurant polynomials, $\kappa[t^{-1},t]$.


Recall that an \emph{integral domain} is a commutative ring with unit $1\neq0$ which has no zero divisors.
An integral domain $R$ is called \emph{principal ideal domain} (PID) if every ideal $I\subseteq R$ is generated by a single element $a\in R$, that is, $I=Ra$.
Familiar examples of principal ideal domains are $\mathbb Z$, the ring of integers; $\kappa[t]$, the ring of polynomials of one variable $t$ with coefficients in a field $\kappa$; 
and $\kappa[t^{-1},t]$, the ring of Laurent polynomials of one variable $t$ with coefficients in the field $\kappa$.


Let $M$ be a module over a principal ideal domain $R$.
Recall that $M$ is called \emph{free} if it admits a basis $\{x_i\}_{i\in I}$, i.e.\ if it is isomorphic to $\bigoplus_{i\in I}R$ for some index set $I$.
In this case, the cardinality of the basis is uniquely determined and referred to as the \emph{dimension} of the free $R$-module $M$, see~\cite[Chapter~III\S7]{La}.

A proof of the following basic fact can be found in \cite[Theorem~III.7.3]{La} or \cite[Theorem~I.5.1]{HS}.

\begin{theorem}
Suppose $M$ is a submodule of a free module $F$ over a principal ideal domain.
Then $M$ is free and its dimension is at most the dimension of $F$.
\end{theorem}

The preceeding result readily implies that submodules of finitely generated modules over a principal ideal domain are finitely generated, see \cite[Corollary~III\S7.2]{La}.

Let $M$ be a module over a principal ideal domain $R$.
The \emph{torsion submodule} of $M$ is defined to be the submodule of all torsion elements, $\operatorname{Tor}(M):=\{x\in M \mid\text{$\exists\lambda\in R\setminus0$ such that $\lambda x=0$}\}$.
If $\operatorname{Tor}(M)=0$, then $M$ is called \emph{torsion free}.
If $\operatorname{Tor}(M)=M$, then $M$ is called \emph{torsion module}.

We have the following fundamental structure theorem, see \cite[Theorem~III\S7.3]{La}.

\begin{theorem}
If $M$ is a finitely generated module over a principal ideal domain, then $M/\operatorname{Tor}(M)$ is free, and there exists an isomorphism $M\cong\operatorname{Tor}(M)\oplus(M/\operatorname{Tor}(M))$.
\end{theorem}

In other words, if $M$ is a finitely generated module over a principal ideal domain $R$, then there exists a decomposition $M\cong T\oplus F$ where $T$ is a finitely generated torsion module and $F$ is a finite dimensional free module.
Moreover, the summands $T$ and $F$ are uniquely determined, up to isomorphism.
More precisely, $T\cong\operatorname{Tor}(M)$ and $F\cong R^m$ where $m\in\mathbb N_0$ denotes the dimension of $M/\operatorname{Tor}(M)$.

%

Let us now consider the pricipal ideal domain $R=\kappa[t^{-1},t]$ where $\kappa$ is a field.
A module over this ring is exactly the same thing as a pair $(M,T)$ where $M$ a $\kappa$-vector space and $T\colon M\to M$ is a $\kappa$-linear isomorphism. 
The vector space $M$ is the underlying vector space of the module and the $\kappa$-linear isomorphism $T$ is defined by multiplication by $t$, its inverse being the multiplication by $t^{-1}$.
Note that $M$ is a finite dimensional $\kappa$-vector space if and only if the module is a finitely generated torsion module.
Hence, finitely generated torsion modules over $\kappa[t^{-1},t]$ can equivalent be regarded as pairs $(M,T)$ where $M$ is a finite dimensional $\kappa$-vector space and $T\colon M\to M$ is an isomorphism.

If $M$ is a finitely generated module over $\kappa[t^{-1},t]$, then its torsion submodule coincides with the kernel of the homomorphism obtained by tensorizing the natural inclusion $\kappa[t^{-1},t]\subseteq\kappa[t^{-1},t]]$ with $M$, that is,
$$
\operatorname{Tor}(M)=\ker\Bigl(M\to\kappa[t^{-1},t]]\otimes_{\kappa[t^{-1},t]}M\Bigr).
$$
Here $\kappa[t^{-1},t]]$ denotes the field of Laurant series in one variable.
 
 
Consider a $G_{2m}$-representation $\rho$ and a decomposition
$$
\rho\cong\bigoplus_{I\in \mathcal B(\rho)}\rho_I\oplus\bigoplus_{J\in \mathcal J(\rho)}\rho_J
$$ 
as in Section~\ref{S2}.
Then the infinite cyclic covering $\tilde\rho$ is a finitely generated module over $\kappa[t^{-1},t]$ in a natural way.
Its free part is isomorphic to the vector space $\kappa[\widetilde{\mathcal B}(\rho)]$ equipped with the isomorphism $T$ induced by the translation $\tau(I)=I+2\pi$.
Its torsion part is isomorphic to the pair $(V,T)$ where $V=\bigoplus_{J\in \mathcal J(\rho)}V_J$ and $T=\bigoplus_{J\in \mathcal J(\rho)} T_J$.
Clearly $V= \kappa[\widetilde{\mathcal J}(\rho)]$.
The underling vector space of  this module is $\kappa[ \widetilde {\mathcal B}(\rho)\sqcup \widetilde {\mathcal J}(\rho)]$.

  
\end{document}